\documentclass[11pt]{amsart}
\usepackage{mathtools,amssymb,amsfonts,amsthm,amscd,amsmath,nccmath,amsgen,upref,enumerate,nicefrac,esint,setspace,fancyhdr}
\usepackage[margin=1in]{geometry}

\newtheorem{cor}{Corollary}
\newtheorem{lem}[cor]{Lemma}
\newtheorem{prop}[cor]{Proposition}

\newtheorem{thm}[cor]{Theorem}

\newtheorem{assumption}[cor]{Assumption}
\newtheorem{remark}[cor]{Remark}
\newtheorem{definition}[cor]{Definition}

\numberwithin{cor}{section}

\DeclareMathOperator{\C}{C}

\renewcommand{\d}{\delta}

\renewcommand{\P}{\mathbb{P}}

\renewcommand{\b}{\beta}

\renewcommand{\O}{\Omega}

\newcommand{\sgn}{\text{sgn}}

\newcommand{\F}{\mathcal F}

\newcommand{\R}{\mathbb{R}}

\newcommand{\N}{\mathbb{N}}

\newcommand{\E}{\mathbb{E}}

\newcommand{\norm}[1]{\| #1 \|}

\newcommand{\ve}{\varepsilon}

\newcommand{\abs}[1]{|#1|}
\providecommand{\ud}[1]{\, \mathrm{d} #1}
\providecommand{\dx}{\ud{x}}
\providecommand{\dy}{\ud{y}}
\providecommand{\dxi}{\ud \xi}
\providecommand{\deta}{\ud{\eta}}

\providecommand{\ds}{\ud{s}}
\providecommand{\dt}{\ud{t}}

\providecommand{\dd}{\ud}

\author{Benjamin Fehrman}
\address{Louisiana State University, Baton Rouge 70802, Louisiana, USA}
\email{fehrman@math.lsu.edu}


\date{\today}

\makeatother

\begin{document}

\title{Stochastic PDEs with correlated, non-stationary Stratonovich noise of Dean--Kawasaki type}

\begin{abstract}
The results of the author and Gess \cite{FG21} develop a robust well-posedness theory for a broad class of conservative stochastic PDEs, with both probabilistically stationary and non-stationary Stratonovich noise, and with irregular noise coefficients like the square root.  However, one case left untreated by \cite{FG21} is the case of SPDEs that combine conservative, non-stationary Stratonovich noise with square root-like nonlinearities.  Such equations arise naturally in the fluctuating hydrodynamics of inhomogenous systems, and a new analysis is required to handle certain discontinuous coefficients appearing in their It\^o formulations.  We treat the discontinuities by showing that the equation exhibits a novel regularization of the logarithm of the solution, and establish the well-posedness by building on the concept of a stochastic kinetic solution introduced in \cite{FG21}.
\end{abstract}

\maketitle

\markright{}

\section{Introduction}

The purpose of this paper is to establish the well-posedness of nonnegative solutions to certain conservative stochastic PDEs with non-stationary noise, which we illustrate through the inhomogenous Dean--Kawasaki equation with correlated Stratonovich noise
\begin{equation}\label{i_eq} \partial_t\rho = \nabla\cdot a(x)\nabla\phi(\rho)-\nabla\cdot(\sqrt{\rho}\circ s(x)\dd\xi)\;\;\textrm{in}\;\;U\times(0,\infty), \end{equation}
with no-flux boundary conditions on a smooth bounded domain $U\subseteq\R^d$.  In the equation, $\circ$ denotes Stratonovich integration, $a$ is a spatially inhomogenous and uniformly elliptic diffusion matrix that satisfies $a=ss^t$, $\phi$ is nondegenerate, and $\xi$ is a spatially correlated, probabilistically stationary noise.  The main result of the paper is the existence, uniqueness, and pathwise stability of solutions.

\begin{thm}[cf.\ Theorem~\ref{thm_unique}, Theorem~\ref{thm_rks_ex}]  Under Assumptions~\ref{assume_d} and \ref{assume_n}, for every $\rho_0\in L^1(\O;L^1(U))$ there exists a unique stochastic kinetic solution of \eqref{i_eq} in the sense of Definition~\ref{sol_def}.  Furthermore, any two solutions $\rho_1$ and $\rho_2$ satisfy
\[\max_{t\in[0,T]}\norm{\rho_1(\cdot,t)-\rho_2(\cdot,t)}_{L^1(U)}\leq \norm{\rho_{0,1}-\rho_{0,2}}_{L^1(U)},\]
and the solutions satisfy the estimates of Propositions~\ref{prop_e}, \ref{prop_h1}, \ref{prop_kolm}, and \ref{prop_kcc}.  \end{thm}

These techniques, with those of the author and Gess \cite{FG21}, can be adapted in a straightforward way to treat a general class of equations of the type
\begin{equation}\label{i_eq_1}\partial_t\rho = \nabla\cdot a\nabla \phi(\rho)-\nabla\cdot(\nu(\rho)+\sigma(\rho)\circ s\dxi)+\lambda(\rho)\dd\xi'+\eta(\rho)\;\;\textrm{in}\;\;U\times(0,\infty),\end{equation}
with Neumann boundary conditions and non-stationary noises $\xi$ and $\xi'$.  In particular, the existence and uniqueness does not depend on the uniform ellipticity of \eqref{i_eq}, and can be made to handle degenerate diffusions like $\phi(\rho)=\rho^m$, for $m\in(0,\infty)$, that include all porous media and fast diffusion type equations.  See Section~\ref{sec_compare}, Remark~\ref{remark_unique_1}, and Remark~\ref{remark_exist_11} for an explanation of these extensions, and for a detailed comparison to \cite{FG21}.

The reason we choose to focus on \eqref{i_eq} is that it best illustrates some difficulties left untreated by the results of \cite{FG21}, which are due specifically to the combination of an inhomogenous diffusion matrix $s$ and the square root noise coefficient, which create the discontinuous coefficient in \eqref{i_2}, and due to the Neumann boundary condition. Furthermore, following Dean \cite{D96} and Kawasaki \cite{K94}, equation \eqref{i_eq} is the natural Dean--Kawasaki equation with correlated, Stratonovich noise associated to the empirical densities
\begin{equation}\label{i_02}m_n(x,t) = n^{-1}\textstyle\sum_{i=1}^n\delta_0(x-X^i_t),\end{equation}
for independent realizations $X^i_t$ of the Markov process with generator $\nabla\cdot a\nabla$ and reflecting boundary conditions.  The challenges in treating \eqref{i_eq} become more apparent after writing the equation in the It\^o form
\begin{equation}\label{i_2}\partial_t\rho = \nabla\cdot a\nabla\phi(\rho)-\nabla\cdot (\sqrt{\rho}s\dd\xi)+\frac{\langle \xi\rangle_1}{8}\nabla\cdot\big(a\nabla\log(\rho)\big)+\frac{\langle \xi\rangle_1}{4}\nabla\cdot\big(\mathbf{1}_{\{\rho>0\}}s(\nabla\cdot s^t)\big),\end{equation}
for the spatially constant quadratic variation $\langle\xi\rangle_1$ of the noise, for the indicator function $\mathbf{1}_{\{\rho>0\}}$ of the set where $\rho$ is positive, and for the vector-valued matrix divergence $(\nabla\cdot s^t)_j  = \partial_k s_{kj}$.  We now remark further on the appearance of the discontinuous coefficient $\mathbf{1}_{\{\rho>0\}}$, which is perhaps somewhat surprising.

A less careful derivation based solely on \eqref{i_eq} might suggest to study the equation
\begin{equation}\label{i_3}\partial_t\tilde{\rho} = \nabla\cdot a\nabla\phi(\tilde{\rho})-\nabla\cdot (\sqrt{\tilde{\rho}}s\dd\xi)+\frac{\langle \xi\rangle_1}{8}\nabla\cdot\big(a\nabla\log(\tilde{\rho})\big)+\frac{\langle \xi\rangle_1}{4}\nabla\cdot\big(s(\nabla\cdot s^t)\big),\end{equation}
which, however, exhibits several deficiencies.  The first is that \eqref{i_3} evidently does not preserve the nonnegativity of the initial data, which is an important qualitative property for the solutions if they are to accurately describe the density fluctuations \eqref{i_02}, or if they are to have any application to the non-equilibrium statistical mechanics theories of fluctuating hydrodynamics (see, for example, Spohn \cite{Spo2012}) or macroscopic fluctuation theory (see, for example, Bertini, De Sole, Gabrielli, Jona-Lasinio, and Landim \cite{BDSGJLL15}).

A second difficulty is that \eqref{i_3} is not stable with respect to regularizations of the square root.  Namely, for a smooth approximation $\sigma$ of the square root satisfying $\sigma(0)=0$, the It\^o form of the regularized equation $\partial_t\rho = \nabla\cdot a\nabla\phi(\rho)-\nabla\cdot(\sigma(\rho)\circ s\dd\xi)$ is the equation
\begin{equation}\label{i_4} \partial_t\rho   = \nabla\cdot a\nabla\phi(\rho)-\nabla\cdot(\sigma(\rho)s\dd\xi)+\frac{\langle \xi\rangle_1}{2}\nabla\cdot\big(\sigma'(\rho)^2a\nabla\rho\big) +\frac{\langle \xi\rangle_1}{2}\nabla\cdot\big(\sigma(\rho)\sigma'(\rho)s(\nabla\cdot s^t)\big). \end{equation}
We study smooth approximations to \eqref{i_eq} of the type \eqref{i_4} in Section~\ref{sec_rkss_exist}, where we obtain stable estimates for the solutions.  In particular, if $\rho_0$ is nonnegative then the solution $\rho$ is nonnegative, and the estimates of Propositions~\ref{prop_e}, \ref{prop_kolm}, and \ref{prop_kcc} allow to construct a solution in Theorem~\ref{thm_rks_ex} by passing to the limit $\sigma(\rho)\rightarrow\sqrt{\rho}$.  With respect to this limit, we have that
\[\sigma(\rho)\sigma'(\rho)\rightarrow \frac{1}{2}\mathbf{1}_{\{\rho>0\}}\;\;\textrm{and}\;\;\sigma'(\rho)^2\nabla\rho \rightarrow\frac{1}{4\rho}\nabla\rho = \frac{1}{4}\nabla\log(\rho),\]
which recovers \eqref{i_2}.  The solutions constructed in this way will necessarily be nonnegative, and will therefore not solve equation \eqref{i_4}, in general, if $s(\nabla\cdot s^t)\neq 0$.

This illustrates a final new difficulty arising due to the discontinuous coefficient $\mathbf{1}_{\{\rho>0\}}$.  It is not only necessary to deal with the discontinuity of $\mathbf{1}_{\{\rho>0\}}$ but, in the proof of uniqueness, it is necessary to distinguish the class of nonnegative solutions to \eqref{i_2} from the potentially signed solutions of \eqref{i_3}.  We do this by showing that the nonnegative solutions of \eqref{i_eq} constructed above almost surely exhibit a dichotomy.  On the event $\rho_0=0$, the solution is zero.  While, on the event $\rho_0\neq 0$, the set $\{\rho=0\}$ has measure zero in the domain and on the boundary in the sense of trace.  We explain these details in Sections~\ref{sec_transport} and \ref{sec_parabolic_time}, where we show that equation \eqref{i_eq} exhibits a surprising regularization of the logarithm $\log(\rho)$ and its integrated time averages.  In Section~\ref{sec_compare} we compare the methods of this paper to those of \cite{FG21}, and in Section~\ref{sec_litlit} we give a complete overview of the literature.

\subsection{Stochastic transport and parabolicity}\label{sec_transport}  Equation \eqref{i_2} can be viewed as a correction to the It\^o equation
\begin{equation}\label{tr_0} \partial_t\rho  = \Delta \rho-\nabla\cdot (\sqrt{\rho}\dd\xi)=\Delta\rho - \frac{1}{2\sqrt{\rho}}\nabla\rho\cdot\dd\xi-\sqrt{\rho}(\nabla\cdot \dd\xi),\end{equation}
where for simplicity here we have taken $a=s=I_{d\times d}$ and $\phi(\rho)=\rho$.  To understand the effect of the stochastic transport term appearing second on the righthand side, consider the much simplified equation, for $a,b\in\R$ and $B_t$ a standard $d$-dimensional Brownian motion,
\begin{equation}\label{tr_1} \partial_t v = a\Delta v+\nabla v\cdot b\dd B_t.\end{equation}
An application of It\^o's formula proves formally that, for $\tilde{v}$ satisfying the potentially backward heat equation $\partial_t\tilde{v} = (a-\nicefrac{b^2}{2})\Delta\tilde{v}$, we have that $v(x,t) = \tilde{v}(x+bB_t,t)$.  That is, the stochastic transport in \eqref{tr_1} effects the parabolic structure of the equation, and equation \eqref{tr_1} is uniformly parabolic if and only if $\nicefrac{b^2}{2}<a$.  If we therefore think formally in terms of parabolicity that
\[\partial_t\rho  = a\Delta v+\nabla v\cdot b\dd B_t = \textrm{``}\big(a-\frac{b^2}{2})\Delta\rho,\textrm{''}\]
then a formal application of the same reasoning to \eqref{tr_0} yields that, in terms of parabolicity,
\[\partial_t\rho  = \Delta \rho-\nabla\cdot (\sqrt{\rho}\dd\xi) = \textrm{``}\Delta\big(\rho-\frac{\langle\xi\rangle_1}{8}\log(\rho)\big)\textrm{''}-\sqrt{\rho}(\nabla\cdot \dd\xi),\]
and therefore that there exists no regime of smallness for $\langle\xi\rangle_1$ that guarantees equation \eqref{tr_0} is uniformly parabolic.  Indeed, equation \eqref{tr_0} behaves like a backward heat equation in the region where the solution $\rho<\nicefrac{\langle \xi\rangle_1}{8}$.  It is for this reason that we consider the corrected equation, for $\theta\in[0,\infty)$,
\begin{equation}\label{tr_2} \partial_t\rho  = \Delta\rho-\nabla\cdot(\sqrt{\rho}\dd\xi)+\frac{\theta\langle\xi\rangle_1}{4}\Delta\log(\rho),\end{equation}
where $\theta=0$ corresponds to the It\^o equation, $\theta=\nicefrac{1}{2}$ corresponds to the Stratonovich formulation of \eqref{tr_0}, and $\theta=1$ corresponds to the Kilmontovich formulation of \eqref{tr_0}.  The above reasoning shows that \eqref{tr_2} is uniformly parabolic if and only if $\theta\in[\nicefrac{1}{2},\infty)$, and it is for this reason that we consider the Stratonovich case $\theta=\nicefrac{1}{2}$.  The entirety of the analysis in this paper would remain valid in the case $\theta\in(\nicefrac{1}{2},\infty)$.  In fact, the regime $\theta\in(\nicefrac{1}{2},\infty)$ is fundamentally different, since then the correction term more than compensates the loss of parabolicity due to the stochastic transport and adds regularity.  However, in the case $\theta\in[0,\nicefrac{1}{2})$ some results, such as the pathwise $L^1(U)$-contraction of Theorem~\ref{thm_unique}, are simply false due to the fact that, in regions where the solutions are small, the respective $L^1(U)$-norms will expand like solutions of a backward heat equation.

\subsection{Persistence of time-averaged parabolicity}\label{sec_parabolic_time} We now explain that, contrary to Section~\ref{sec_transport} above, the parabolicity added by the Stratonovich-to-It\^o correction does persist in a time-averaged sense.  We first prove in Proposition~\ref{prop_h1} that, on an approximate level, for $\langle\rho_0\rangle_U = \fint_U\rho_0$,
\begin{align} \label{int_0}
&\E\big[\max_{t\in[0,T]}\norm{(\rho-\langle\rho_0\rangle_{U})}^2_{H^{-1}(U)}\big]+\langle\xi\rangle_1\E\big[\langle\rho_0\rangle_{U}\medint\int_0^T\medint\int_{U}\abs{\log(\rho\wedge 1)}\big]
\\ \nonumber & \leq c\big(\norm{(\rho_0-\langle\rho_0\rangle_{U})}^2_{H^{-1}(U)}+\langle\xi\rangle_1^2\big(1+\E\big[\langle\rho_0\rangle^2_{U}\big]\big)\big).
\end{align}
The remarkable aspect of this estimate is that it implies---using only the $L^1$-boundedness and nonvanishing of $\rho_0$---that the logarithm $\log(\rho)$ is $\P$-a.s.\ space-time integrable on the event that $\rho_0\neq 0$, a property that fails for the heat equation.  Indeed, if $\overline{\rho}$ solves the heat equation  in $B_2\times(0,T]$ with initial data $\rho_0 = \mathbf{1}_{B_{\nicefrac{1}{2}}}$, then for all $t\in(0,1]$ and $\abs{x}\geq 1$, for some $c\in(0,\infty)$,
\[\overline{\rho}(x,t)\leq c\abs{B_{\nicefrac{1}{2}}}(2\pi t)^{-\frac{d}{2}}\exp\big(-\frac{\abs{x}^2}{8t}\big)\;\;\textrm{and}\;\;\log(\overline{\rho}(x,t))\leq -\frac{\abs{x}^2}{8t}+\log( c\abs{B_{\nicefrac{1}{2}}})-\frac{d}{2}\log(2\pi t).\]
Due to the divergence of the leading term as $t\rightarrow 0$, the composition $\log(\overline{\rho})$ is not space time integrable despite $\langle \rho_0\rangle_{B_2}=\abs{B_{\nicefrac{1}{2}}}/\abs{B_2}>0$.  The validity of estimate \eqref{int_0} is due to the structure of the equation, and to the fact that the singularity of the logarithm at $\rho\simeq 0$ is negative.

We will now show that it is also possible to obtain positive regularity of $\log(\rho)$ in a time-averaged sense.  First, in Proposition~\ref{prop_kolm}, we use estimate \eqref{int_0} to establish regularity of the solution in time, seen as a map into $H^{-1}(U)$:  for all $\beta\in(0,\nicefrac{1}{2})$ sufficiently small depending on the integrability of the initial data,
\begin{equation}\label{int_2}\E\big[\norm{\rho}_{\C^{0,\beta}([0,T];H^{-1}(U))}]<\infty.\end{equation}
We then prove in Proposition~\ref{prop_kcc} that, for $\rho_{r,t}=\rho(\cdot,t)-\rho(\cdot,r)$ for every $r\leq t\in[0,T]$,
\[\E\big[\norm{\nabla \medint\int_r^t\big(\phi(\rho)+\frac{\langle\xi\rangle_1}{8}\log(\rho)\big)}^{2p}_{L^2(U)}\big]\leq c\big(\E\big[\norm{\rho_{r,t}}^{2p}_{H^{-1}(U)}\big]+\langle\xi\rangle_1^{p}(t-r)^\frac{p}{2}\E\big[\langle\rho_0\rangle^p_U\big]\big),\]
which implies using \eqref{int_2} that the map $L(p)\colon[0,T]\rightarrow H^1(U)$ defined by
\begin{equation}\label{int_3}t\in[0,T]\rightarrow L(\rho)_t=\medint\int_0^t\big(\phi(\rho)+\frac{\langle\xi\rangle_1}{8}\log(\rho)-\langle \phi(\rho)+\frac{\langle\xi\rangle_1}{8}\log(\rho)\rangle_U\big)\in H^1(U),\end{equation}
satisfies $\E\big[\norm{L(\rho)}_{\C^{0,\beta}([0,T];H^1(U))}\big]<\infty$, for every $\beta\in(0,\nicefrac{1}{2})$ sufficiently small depending on the integrability of the initial data, where the righthand side is understood to be zero on the event that the initial data $\rho_0=0$.

The collection of estimates \eqref{int_0}, \eqref{int_2}, and \eqref{int_3} are new even in the settings covered by the results of \cite{FG21}, and they are simply false for the heat equation, as explained above.  They demonstrate that some regularizing properties persist from the Stratonovich-to-It\^o correction, both in terms of integrability and in terms of time-averaged regularity.  The integrability of $\log(\rho)$ and the positive regularity of the map $L(\rho)$ prove $\P$-a.s.\ that the set $\{\rho=0\}$ has zero measure in $U\times[0,T]$, and has zero measure in the sense of a trace on $\partial U\times[0,T]$ interpreted in the sense of Definition~\ref{sol_def} and Remark~\ref{remark_sol}.  There is therefore a dichotomy in the solution:  $\rho=0$ almost everywhere on the event $\rho_0=0$, and on the event $\rho_0\neq 0$ we have that $\rho$ is $\P$-a.s.\ almost everywhere positive.  This means almost surely that the indicator $\mathbf{1}_{\{\rho>0\}}=0$ or $\mathbf{1}_{\{\rho>0\}}=1$ is actually continuous, and can therefore be handled as is done in step \eqref{u_32} of the uniqueness proof of Theorem~\ref{thm_unique}.

\subsection{Comparison to \cite{FG21}}\label{sec_compare}  We recall equation \eqref{i_4}, where the It\^o form of a general Stratonovich equation $\partial_t\rho = \nabla\cdot a\nabla\phi(\rho)-\nabla\cdot(\sigma(\rho)\circ s\dd\xi)$ is the equation
\[\partial_t\rho   = \nabla\cdot a\nabla\phi(\rho)-\nabla\cdot(\sigma(\rho)s\dd\xi)+\frac{\langle \xi\rangle_1}{2}\nabla\cdot\big(\sigma'(\rho)^2a\nabla\rho\big) +\frac{\langle \xi\rangle_1}{2}\nabla\cdot\big(\sigma(\rho)\sigma'(\rho)s(\nabla\cdot s^t)\big).\]
As remarked in the introduction, even if $\xi$ is a probabilistically stationary noise, the presence of an inhomogenous diffusion coefficient $s$ creates a non-stationary noise $s\xi$.  The results of \cite[Theorem~4.6, Theorem~5.25]{FG21} do apply to non-stationary noise under the additional assumption \cite[Assumption~4.1]{FG21}, which requires that the product $\sigma\sigma'$ extends to a continuous function on $[0,\infty)$ satisfying $\sigma(0)\sigma'(0)=0$.  This is not the case for $\sigma(\eta)=\sqrt{\eta}$, as explained above, for which $\sigma\sigma'$ extends continuously to $[0,\infty)$ but fails to vanish at $\eta=0$.  In this way, the results of \cite{FG21} apply to the equation $\partial_t\rho = \nabla\cdot a\nabla\rho^m-\nabla\cdot(\rho^{\nicefrac{m}{2}}\circ s\dd\xi)$, for every $m\in(1,\infty)$, but not for $m=1$.  One primary purpose of this paper is therefore to develop a robust well-posedness theory for equations like \eqref{i_eq} with non-stationary noise which, as demonstrated by \eqref{i_02}, arise naturally to describe density fluctuations in spatially inhomogenous systems.

Furthermore, the proof of uniqueness must also account for the discontinuous coefficient $\mathbf{1}_{\{\rho>0\}}$ in the domain and on the boundary, which is done by demonstrating the dichotomy explained in the introduction.  Namely, for nonnegative initial data, the solutions almost surely vanish on the event $\rho_0=0$, and they are almost surely nowhere vanishing on the event $\rho_0\neq 0$.  Finally, even in the case of probabilistically stationary noise, the proof of existence in \cite{FG21} was substantially complicated by the fact that the logarithm of the solution was not known to be integrable.  This substantially complicated the regularity estimates in time, which required to localize the solution away from regions where it approached zero, as done in \cite[Proposition~5.11]{FG21}, and required the introduction of a new metric for the strong $L^1$-topology in \cite[Definition~5.19, Lemma~5.20]{FG21} to establish the tightness in law of the approximating solutions.  The estimates of Propositions~\ref{prop_h1}, \ref{prop_kolm}, and \ref{prop_kcc} allow to avoid these complications by establishing the integrability of the logarithm of the solution and its time averaged regularity, as explained in Section~\ref{sec_parabolic_time} above, and to establish the initial time continuity \eqref{exis_134} and to construct a solution in Theorem~\ref{thm_rks_ex}.

\subsection{Comments on the literature}\label{sec_litlit}  Conservative SPDEs modeled on the Dean--Kawasaki equation \eqref{i_eq} with correlated noise have recently received a lot of attention in the literature.  The results of this paper are most closely related to those of the author and Gess \cite{FG21}, which considered general equations like \eqref{i_eq_1} on the $d$-dimensional torus.  These results were extended to the whole space by the author and Gess \cite{FehGes24}, to equations with interaction kernels by Wang, Wu, and Zhang \cite{DKInt}, and to equations on smooth, bounded domains with Dirichlet boundary conditions by Popat \cite{Pop24}.  Another purpose of this work is to extend \cite{FG21} to the Neumann case.  Other works on SPDEs with conservative noise include, among others, Lions, Perthame, and Souganidis \cite{LPS13-2,LPS13,LPS14}, Dirr, the author, and Gess \cite{DirFehGes19}, Friz and Gess \cite{FG16}, Gess and Souganidis \cite{GS14,GS17,GS16-2}, the author and Gess \cite{FG17}, Dareiotis and Gess \cite{DG18}, and Clini \cite{Cli2023}.

Several of these works, including \cite{Cli2023,FG17,FG21,Pop24} as well as the earlier works of Debussche and Vovelle \cite{DebVov2010}, Hofmanov\'a \cite{Hof2013}, Debussche, Hofmanov\'a, and Vovelle \cite{DebHofVov2016}, and the author and Gess \cite{FehGes20211}, are based on the equation's kinetic form.  The kinetic formulation was introduced in the context of deterministic conservation laws by Lions, Perthame, and Tadmor \cite{LPT94} and Perthame \cite{Per1998}.  See also the works of Bendahmane and Karlsen \cite{BenKar2004}, of Chen and Perthame \cite{CP03}, of De Lellis, Otto, and Westdickenberg \cite{DeLOttWes2003}, and of Karlsen and Riseboro \cite{KarRis2003}.

The Dean--Kawasaki equation driven by space-time white noise was introduced by Dean \cite{D96} and Kawasaki \cite{K94} as a model for the density fluctuations of certain interacting diffusive systems.  For some recent progress, we refer to Donev, Fai, and Vanden-Eijnden \cite{DFVE14}, Donev and Vanden-Eijnden \cite{DVE14}, Lehmann, Konarovskyi, von Renesse \cite{LKR19}, and Konarovskyi and von Renesse \cite{KR17,KR15}, Sturm and von Renesse \cite{StuvRe2009}, and Cornalba, Shardlow, and Zimmer \cite{CorShaZim2019,CorShaZim2020}.  For an explanation of the connection between macroscopic fluctuation theory and fluctuating hydrodynamics in the context of the Dean--Kawasaki equation, we refer to Bouchet, Gaw\c edzki, and Nardini \cite{BGN16}.  For an overview of macroscopic fluctuation theory, see Bertini, De Sole, Gabrielli, Jona-Lasinio, and Landim \cite{BDSGJLL15}.

We finally remark that several recent works have developed the connection between conservative SPDEs like \eqref{i_eq} and the fluctuations and large deviations of certain interacting particle systems.  The author and Gess \cite{FehGes19} and \cite{DirFehGes19} have shown along appropriate scaling limits that the small-noise large deviations of generalized equations based on \eqref{i_eq} coincide with those of the zero range and symmetric simple exclusion processes, and these results were extended to equations with interaction kernels by Wu and Zhang \cite{DKIntLDP}.  Cornalba and Fischer \cite{CF23} show that a discretization of the Dean--Kawasaki equation approximates to arbitrary order density fluctuations of independent Brownian motions, and their work was later extended by Cornalba, Fischer, Ingmanns, and Raithel \cite{CFIR23} to weakly interacting diffusions.  See also the related works of Djurdjevac, Kremp, and Perkowski \cite{DKP22}, Gess, Wu, and Zhang \cite{GesWuZha2024}, and Clini and the author \cite{CliFeh2024}.

\section{Stochastic kinetic solutions of \eqref{i_eq}}

We now establish the well-posedness and stability of solutions to \eqref{i_eq}.  The section is organized as follows.  The domain and randomness in the equation are fixed in Section~\ref{sec_setting}.   We define a stochastic kinetic solution in Section~\ref{sec_rkss}, we prove that such solutions are unique in Section~\ref{sec_rkss_unique}, and we prove that such solutions exist in Section~\ref{sec_rkss_exist}.

\subsection{The setting and the randomness in the equation}\label{sec_setting}  We will make the following assumptions on the diffusion coefficient and the noise.

\begin{assumption}\label{assume_d}  Let $d\in\N$, let $T\in(0,\infty)$, let $U\subseteq\R^d$ be a nonempty, smooth, and bounded domain, and assume that there exists $s\in \C^2(U;\R^{d\times d})$ such that $a=ss^t$ and that, for some $\lambda\leq\Lambda\in(0,\infty)$, for every $v\in \R^d$, $\langle av,v\rangle\geq \lambda\abs{v}^2$ and $\langle a^{-1}v,v\rangle \geq\Lambda^{-1}\abs{v}^2$.  Furthermore, assume that $\phi\in \C([0,\infty))\cap\C^{1,\nicefrac{1}{2}}_{\textrm{loc}}((0,\infty))$ satisfies that $\phi(0)=0$ and, for some $\tilde{\lambda}\leq\tilde{\Lambda}\in(0,\infty)$, that $\tilde{\lambda}\leq \phi'(\eta)\leq \tilde{\Lambda}$ for every $\eta\in(0,\infty)$.
\end{assumption}

\begin{assumption}\label{assume_n}  Let $(\O,\F,\P)$ be a probability space, let $(\F_t)_{t\in[0,\infty)}$ be a complete, right-continuous filtration on $(\O,\F)$, and let $\{B_k\}_{k\in\N}$ be independent $d$-dimensional $\F_t$-adapted Brownian motions.  Let $F=(f_k)_{k\in\N}$ for $\C^2$-smooth functions $f_k$, let $\xi^F$ to be the $d$-dimensional noise $\xi^F(x,t) = \textstyle\sum_{k=1}^\infty f_k(x)B_k(t)$, and assume that $0<\langle \xi^F\rangle_1 = \textstyle\sum_{k=1}^\infty f_k^2<\infty$ is spatially constant and that
\[\sup_{x\in U}\langle \nabla\cdot s\xi^F\rangle_1 = \sup_{x\in U}\textstyle\sum_{k=1}^\infty \big(\abs{s^t\nabla f_k}^2+f_k^2\abs{\nabla\cdot s^t}^2\big)<\infty.\]
 \end{assumption}

\subsection{Stochastic kinetic solutions.}\label{sec_rkss}  The solution theory is based on the equation's kinetic formulation.  Since the following derivation is based on the work of Chen and Perthame \cite{CP03} (see also Perthame \cite{PerthameKinetic}), and full details in this context can be found in \cite[Section~2]{FG21}, we only briefly recall some of the main ideas here.  We first remark on the no-flux Neumann boundary condition.  Based on equation \eqref{i_2}, this means formally that, for the outer unit normal $\nu$ to $U$,
\[\big(a\nabla\phi(\rho)-\sqrt{\rho}s\dd\xi^F+\frac{\langle \xi\rangle_1}{8}a\nabla\log(\rho)+\frac{\langle \xi\rangle_1}{4}\mathbf{1}_{\{\rho>0\}}s(\nabla\cdot s^t)\big)\cdot \nu=0\;\;\textrm{on}\;\;\partial U,\]
which, however, must be interpreted weakly on the level of the equation.  We first do this on the level of the equation's entropy formulation:  if $S$ is a smooth function satisfying $S(0)=S'(0)=0$, an application of It\^o's formula and the no-flux boundary condition show that, for every $\psi\in\C^\infty(\overline{U})$,
\begin{align*}
& \dd\Big(\medint\int_{U}S(\rho)\psi\Big)  =\Big(-\medint\int_{U}(S'(\rho)\nabla\psi)\cdot a\nabla\phi(\rho) - \medint\int_{U} (S''(\rho)\psi)\phi'(\rho)\nabla\rho\cdot a\nabla\rho\Big)\dt
\\  & -\Big(\frac{\langle\xi^F\rangle_1}{8}\medint\int_{U}(S'(\rho)\nabla\psi)\cdot a\nabla\log(\rho)+\frac{\langle\xi^F\rangle_1}{4}\medint\int_{U}\mathbf{1}_{\{\rho>0\}}s(\nabla\cdot s^t)\cdot (S'(\rho)\nabla\psi)\Big)\dt
\\ & +\Big(\frac{1}{2}\medint\int_{U}(S''(\rho)\psi)\rho\langle\nabla\cdot s\xi^F\rangle_1+\frac{\langle\xi^F\rangle_1}{4}\medint\int_{U}(S''(\rho)\psi)\mathbf{1}_{\{\rho>0\}}s(\nabla\cdot s^t)\cdot\nabla\rho\Big)\dt
\\ & +\medint\int_{U} (S''(\rho)\psi)\sqrt{\rho}\nabla\rho\cdot s\dd\xi^F+\medint\int_U (S'(\rho)\nabla\psi)\cdot \sqrt{\rho}s\dd\xi^F.
\end{align*}
However, in general, for nonnegative $\psi$ and for convex functions $S$ the composition $S(\rho)\psi$ will only satisfy the entropy inequality that the lefthand side of this equation is less than or equal to the righthand side.  The kinetic formulation quantifies this inequality exactly.  Precisely, we introduce an additional velocity variable $\eta\in\R$ and define the kinetic function $\overline{\chi}\colon\R^2\rightarrow\{-1,0,1\}$ as $\overline{\chi}(s,\eta) = \mathbf{1}_{\{0<\eta<s\}}-\mathbf{1}_{\{s<\eta<0\}}$, and the kinetic function $\chi$ of the solution $\rho$ as the composition
\[\chi_t(x,\eta) = \overline{\chi}(\rho(x,t),\eta)=\mathbf{1}_{\{0<\eta<\rho(x,t)\}},\]
where the final equality follows from the nonnegativity of $\rho$.  It follows from the distributional equalities, for the one-dimensional Dirac delta distribution $\delta_0$ and $\delta_{\rho(x,t)}=\delta_0(\eta-\rho(x,t))$,
\[\nabla_x\chi_t(x,\eta) = \delta_{\rho(x,t)}\nabla\rho(x,t)\;\;\textrm{and}\;\;\partial_\xi\chi_t(x,\eta) = \delta_0-\delta_{\rho(x,t)},\]
and the equality $\int_\R\chi_t(x,\eta)S'(\eta)\deta = S(\rho(x,t))-S(0)$ that, for every nonnegative $\psi$, convex $S$ satisfying $S(0)=S'(0)=0$, and $\Psi(\xi,x) = S'(\xi)\psi(x)$,
\begin{align*}
& \dd\Big(\medint\int_{U}\medint\int_\R\chi\Psi\Big)  \leq \Big(-\medint\int_{U} (\nabla_x \Psi)(x,\rho)\cdot a\nabla\phi(\rho) - \medint\int_{U} (\partial_\eta\Psi)(x,\rho)\phi'(\rho) \nabla\rho\cdot a\nabla\rho\Big)\dt
\\  & -\Big(\frac{\langle\xi^F\rangle_1}{8}\medint\int_{U} (\nabla_x\Psi)(x,\rho)\cdot a\nabla\log(\rho)+\frac{\langle\xi^F\rangle_1}{4}\medint\int_{U}\mathbf{1}_{\{\rho>0\}}s(\nabla\cdot s^t)\cdot (\nabla_x\Psi)(x,\rho)\Big)\dt
\\ & +\Big(\frac{1}{2}\medint\int_{U}(\partial_\eta\Psi)(x,\rho)\rho\langle\nabla\cdot s\xi^F\rangle_1+\frac{\langle\xi^F\rangle_1}{4}\medint\int_{U}(\partial_\eta\Psi)(x,\rho)\mathbf{1}_{\{\rho>0\}}s(\nabla\cdot s^t)\cdot\nabla\rho\Big)\dt
\\ & +\medint\int_{U}(\nabla_x\Psi)(x,\rho)\cdot\sqrt{\rho} s\dd\xi^F+\medint\int_U(\partial_\eta\Psi)(x,\rho)\sqrt{\rho}\nabla\rho\cdot s\dd\xi^F,
\end{align*}
where here we write $(\nabla\Psi)(x,\rho)$ to mean the gradient of $\Psi$ evaluated at the point $(x,\rho)$ as opposed to the full gradient of the composition $\Psi(x,\rho)$.

The kinetic formulation quantifies the above entropy inequality exactly using a kinetic defect measure, which is a nonnegative Radon measure $q$ on $U\times\R\times[0,T]$ that satisfies, in the sense of measures,
\[\delta_{\rho}\phi'(\eta)\big(\nabla\rho\cdot a\nabla\rho\big) \leq q.\]
We then have using the density of linear combinations of functions of the type $S'(\xi)\psi(x)$ in $\C^\infty(\overline{U}\times\R)$ that, for every $\Psi\in\C^\infty(\overline{U}\times\R)$,
\begin{align*}
& \medint\int_\R\medint\int_{U}\chi \Psi\Big|_{r=0}^t =  -\medint\int_0^t\medint\int_{U} (\nabla_x \Psi)(x,\rho)\cdot a\nabla\phi(\rho) - \medint\int_0^t\medint\int_\R\medint\int_{U} (\partial_\eta\Psi) q
\\  & -\frac{\langle\xi^F\rangle_1}{8}\medint\int_0^t\medint\int_{U} (\nabla_x\Psi)(x,\rho)\cdot a\nabla\log(\rho)-\frac{\langle\xi^F\rangle_1}{4}\medint\int_0^t\medint\int_{U}\mathbf{1}_{\{\rho>0\}}s(\nabla\cdot s^t)\cdot (\nabla_x\Psi)(x,\rho)
\\ & +\frac{1}{2}\medint\int_0^t\medint\int_{U}(\partial_\eta\Psi)(x,\rho)\rho\langle\nabla\cdot s\xi^F\rangle_1+\frac{\langle\xi^F\rangle_1}{4}\medint\int_0^t\medint\int_{U}(\partial_\eta\Psi)(x,\rho)\mathbf{1}_{\{\rho>0\}}s(\nabla\cdot s^t)\cdot\nabla\rho
\\ & +\medint\int_0^t\medint\int_{U}(\nabla_x\Psi)(x,\rho)\cdot\sqrt{\rho} s\dd\xi^F+\medint\int_0^t\medint\int_U(\partial_\eta\Psi)(x,\rho)\sqrt{\rho}\nabla\rho\cdot s\dd\xi^F.
\end{align*}
The above equation allows to consider test functions that are compactly supported on $\overline{U}\times(0,\infty)$ in the velocity variable, which allows to localize the solution away from its zero set and therefore the singularities of the logarithm and square root.  Furthermore, by also localizing the solution away from infinity, it allows to treat nonnegative initial data that is only integrable.

\begin{definition}\label{d_measure}  Under Assumption~\ref{assume_n}, a kinetic measure is a map $q$ from $\O$ to the space of nonnegative, locally finite Radon measures on $U\times\R\times[0,T]$ that satisfies the property that, for every $\psi\in\C^\infty_c(U\times\R)$,
\[(\omega,t)\in\O\times[0,T]\rightarrow \medint\int_0^t\medint\int_\R\medint\int_{U}\psi(\xi,x)\dd q(\omega)\;\;\textrm{is $\F_t$-predictable.}\]
\end{definition}

 \begin{definition}\label{sol_def}  Under Assumptions~\ref{assume_d} and \ref{assume_n}, let $\rho_0\in L^1(\O;L^1(U))$ be $\F_0$-measurable.  A \emph{stochastic kinetic solution} of \eqref{i_eq} with initial data $\rho_0$ is a nonnegative, $\P$-a.s.\ continuous $L^1(U)$-valued $\F_t$-predictable function $\rho$ that satisfies the following three properties.
\begin{enumerate}[(i)]
\item \emph{Preservation of mass}:  $\P$-a.s.\ for every $t\in[0,T]$,
\begin{equation}\label{d_mass}\norm{\rho(\cdot,t)}_{L^1(U)}=\norm{\rho_0}_{L^1(U)}.\end{equation}
\item \emph{Local regularity}: $\P$-a.s.\ we have that, for every $K\in [1,\infty)$,
\begin{equation}\label{d_s2} (\rho\wedge K)\vee K^{-1}\in L^2([0,T];H^1(U)).\end{equation}
\item \emph{Non-vanishing}:  $\P$-a.s.\ on the event $\{\rho_0=0\}$ we have that $\rho=0$ and, in the sense of Remark~\ref{remark_sol}, $\P$-a.s.\ on the event $\{\rho_0\neq 0\}$ we have that the $(d+1)$-dimensional and $d$-dimensional Lebesgue measures
\begin{equation}\label{s_log} \abs{\{(x,t)\in U\times[0,T]\colon \rho(x,t)=0\}}=\abs{\{(x,t)\in \partial U\times[0,T]\colon \rho(x,t)=0\}}=0.\end{equation}
\end{enumerate}
Furthermore, there exists a kinetic measure $q$ that satisfies the following three properties.
\begin{enumerate}[(i)]
\setcounter{enumi}{3}
\item \emph{Regularity}: $\P$-a.s.\ as nonnegative measures on $U\times\R\times[0,T]$,
\begin{equation}\label{d_s3}\delta_{\rho(x,t)}\phi'(\eta)\big(\nabla\rho\cdot a\nabla\rho\big)\leq q.\end{equation}
\item \emph{Vanishing at infinity}:  $\P$-a.s.\ we have that
\begin{equation}\label{d_s4}\liminf_{M\rightarrow\infty}\big(q(U\times[M,M+1]\times[0,T])\big)=0.\end{equation}
\item \emph{The equation}: $\P$-a.s.\ for every $\psi\in \C^\infty_c(\overline{U}\times(0,\infty))$ and $t\in[0,T]$,
\begin{align}\label{sol_eq}
& \medint\int_\R\medint\int_{U}\chi(\xi,x,r)\psi(\xi,x) = \medint\int_\R\medint\int_{U}\overline{\chi}(\rho_0)\psi(\xi,x) -\medint\int_0^t\medint\int_{U} (\nabla_x \psi)(x,\rho)\cdot a\nabla\phi(\rho)- \medint\int_0^t\medint\int_{U}\medint\int_\R (\partial_\eta\psi) q
\\  \nonumber & -\frac{\langle\xi^F\rangle_1}{8}\medint\int_0^t\medint\int_{U} (\nabla_x\psi)(x,\rho)\cdot a\nabla\log(\rho)-\frac{\langle\xi^F\rangle_1}{4}\medint\int_0^t\medint\int_{U}\mathbf{1}_{\{\rho>0\}}s(\nabla\cdot s^t)\cdot (\nabla_x\psi)(x,\rho)
\\ \nonumber & +\frac{1}{2}\medint\int_0^t\medint\int_{U}(\partial_\eta\psi)(x,\rho)\rho\langle\nabla\cdot s\xi^F\rangle_1+\frac{\langle\xi^F\rangle_1}{4}\medint\int_0^t\medint\int_{U}(\partial_\eta\psi)(x,\rho)\mathbf{1}_{\{\rho>0\}}s(\nabla\cdot s^t)\cdot\nabla\rho
\\ \nonumber & +\medint\int_{U}(\nabla_x\psi)(x,\rho)\cdot\sqrt{\rho} s\dd\xi^F+\medint\int_U(\partial_\eta\psi)(x,\rho)\sqrt{\rho}\nabla\rho\cdot s\dd\xi^F.
\end{align}
\end{enumerate}
\end{definition}

\begin{remark}\label{remark_sol}  The local regularity \eqref{d_s2} is designed to handle two potential problems.  The first is the lack of regularity on the set $\{\rho=0\}$, which is less important for $\rho$ itself but is necessary to treat the logarithm and porous media type nonlinearities.  The second is the lack of integrability of $\nabla\rho$ on the set where $\rho$ is large, when considering initial data that is not $L^2$-integrable.

Concerning the non-vanishing condition \eqref{s_log}, it follows from \eqref{d_s2} that for every $K\in[1,\infty)$ the function $(\rho\vee K)\wedge K^{-1}$ has a trace on the boundary.  In the case that $\rho$ itself does not have a trace, we understand the non-vanishing condition \eqref{s_log} in the sense that
\[\abs{\cap_{K=1}^\infty \{(x,t)\in\partial U\times[0,T]: (\rho\wedge 1)\vee K^{-1}=K^{-1}\}}=0.\]
We lastly observe that that the coefficient $\mathbf{1}_{\{\rho>0\}}$ is unnecessary in the final term of the penultimate line of \eqref{sol_eq} because $\psi$ is compactly supported in $(0,\infty)$ in the velocity variable, so $\mathbf{1}_{\{\rho>0\}}=1$ on the support of $\psi$ and its derivatives. The discontinuity does reappear in the proof of uniqueness, however, when removing certain cutoff functions in the velocity variable in step \eqref{u_32} of Theorem~\ref{thm_unique}. \end{remark}

\subsection{Uniqueness of stochastic kinetic solutions}\label{sec_rkss_unique}  The proof of uniqueness is based formally on differentiating the equality
\[\medint\int_{U}\abs{\rho_1-\rho_2} = \medint\int_{U}\medint\int_\R\abs{\chi_1-\chi_2}^2 = \medint\int_{U}\medint\int_\R\chi_1-\chi_2-2\chi_1\chi_2,\]
for the kinetic function $\chi_i$ of a solution $\rho_i$, which follows from the nonnegativity of the solutions and the fact that the the kinetic functions $\chi_i$ are $\{0,1\}$-valued.  However, due to the singularities of the coefficients and the initial data that is only $L^1$-integrable, it is necessary to introduce cutoff functions in the velocity variable that localize the solutions away from regions in which they approach zero and infinity.  Property \eqref{d_s4} of Definition~\ref{sol_def} is used to remove the cutoff at infinity, and Proposition~\ref{prop_zero} is used to remove the cutoff near zero.

\begin{prop}\label{prop_zero}  Under Assumptions~\ref{assume_d} and \ref{assume_n}, if $\rho$ is a stochastic kinetic solution of \eqref{i_eq} in the sense of Definition~\ref{sol_def} with nonnegative, $\F_0$-measurable initial data $\rho_0\in L^1(\O;L^1(U))$, it follows $\P$-a.s.\ that
\[\lim_{\beta\rightarrow 0}\E[\beta^{-1}q(U\times [\nicefrac{\beta}{2},\beta]\times[0,T])]=\liminf_{\beta\rightarrow 0}\big(\beta^{-1}q(U\times [\nicefrac{\beta}{2},\beta]\times[0,T])\big)= 0.\]
\end{prop}

\begin{proof}  If $\rho_0=0$ it follows from \eqref{d_mass} that $\rho=0$ and then from \eqref{sol_eq} that $q=0$, for which the claim follows.  We therefore assume that $\rho_0\neq 0$ and using \eqref{s_log} that $\mathbf{1}_{\{\rho>0\}}=1$ almost everywhere on $U\times[0,T]$.  It follows from a straightforward regularization argument by convolution and the vanishing of the measures at infinity \eqref{d_s4} that, for every $\beta\in(0,1)$, the function $\psi_\beta$ satisfying $\psi_\beta(0)=\psi_\beta'(0)=0$ and $\psi_\beta''(\eta)=(\nicefrac{\beta}{2})^{-1}\mathbf{1}_{[\nicefrac{\beta}{2},\beta]}(\eta)$ is an admissible test function for equation \eqref{sol_eq}.  We therefore have $\P$-a.s.\ for every $\b\in(0,1)$ that
\begin{align}\label{u_1}
& \medint\int_{U}\psi_\beta(\rho)\Big|_{r=0}^{r=T}+\medint\int_0^T\medint\int_{U}\medint\int_\R(\nicefrac{\beta}{2})^{-1}\mathbf{1}_{[\nicefrac{\beta}{2},\beta]} q = \frac{1}{2}\medint\int_0^T\medint\int_{U}(\nicefrac{\beta}{2})^{-1}\mathbf{1}_{[\nicefrac{\beta}{2}\leq \rho\leq \beta]}\rho\langle\nabla\cdot s\xi^F\rangle_1
\\ \nonumber & +\frac{\langle\xi^F\rangle_1}{4}\medint\int_0^T\medint\int_{U}s(\nabla\cdot s^t)\cdot(\nicefrac{\beta}{2})^{-1}\mathbf{1}_{[\nicefrac{\beta}{2}\leq\rho\leq\beta]}\nabla\rho + \medint\int_0^T\medint\int_{U}(\nicefrac{\beta}{2})^{-1}\mathbf{1}_{[\nicefrac{\beta}{2}\leq\rho\leq \beta]}\sqrt{\rho}\nabla\rho\cdot s\dd\xi^F.
\end{align}
For the first term on the righthand side of \eqref{u_1}, we have that
\[\frac{1}{2}\medint\int_0^T\medint\int_{U}(\nicefrac{\beta}{2})^{-1}\mathbf{1}_{[\nicefrac{\beta}{2}\leq \rho\leq \beta]}\rho\langle\nabla\cdot s\xi^F\rangle_1\leq \medint\int_0^T\medint\int_{U}\mathbf{1}_{[\nicefrac{\beta}{2}\leq \rho\leq \beta]}\langle\nabla\cdot s\xi^F\rangle_1,\]
and therefore, using the boundedness of $\langle\nabla\cdot s\xi^F\rangle_1$ and the dominated convergence theorem, we have $\P$-a.s.\ that
\begin{equation}\label{u_01}\lim_{\beta\rightarrow 0}\frac{1}{2}\medint\int_0^T\medint\int_{U}(\nicefrac{\beta}{2})^{-1}\mathbf{1}_{[\nicefrac{\beta}{2}\leq \rho\leq \beta]}\rho\langle\nabla\cdot s\xi^F\rangle_1 = 0.\end{equation}
For the second term on the righthand side of \eqref{u_1}, using the local regularity \eqref{d_s2} of the solution, which implies the existence of a trace for $\psi_\beta'(\rho)$ in the sense of Evans \cite[Section~5.5, Theorem~1]{Eva2010}, and the $\C^2$-regularity of $s$, we have, for the outer unit normal $\nu$ of the smooth domain $U$,
\begin{align*}
\frac{\langle\xi^F\rangle_1}{4}\medint\int_0^T\medint\int_{U}s(\nabla\cdot s^t)\cdot(\nicefrac{\beta}{2})^{-1}\mathbf{1}_{[\nicefrac{\beta}{2}\leq\rho\leq\beta]}\nabla\rho & = \frac{\langle\xi^F\rangle_1}{4}\medint\int_0^T\medint\int_{\partial U} \psi_\beta'(\rho)\big(s(\nabla\cdot s^t)\cdot\nu\big)
\\ & \quad -\frac{\langle\xi^F\rangle_1}{4}\medint\int_0^T\medint\int_U \psi'_\beta(\rho)\nabla\cdot (s(\nabla\cdot s^t)).
\end{align*}
Since by definition we have that $\psi_\beta'(\rho)\rightarrow \mathbf{1}_{\{\rho >0\}}$ as $\beta\rightarrow 0$, and since $\mathbf{1}_{\{\rho>0\}}=1$ almost everywhere in $U$ and on the boundary, it follows using the integration by parts formula in the final line that
\begin{align}\label{u_02}
& \lim_{\beta\rightarrow 0}\frac{\langle\xi^F\rangle_1}{4}\medint\int_0^T\medint\int_{U}s(\nabla\cdot s^t)\cdot(\nicefrac{\beta}{2})^{-1}\mathbf{1}_{[\nicefrac{\beta}{2}\leq\rho\leq\beta]}\nabla\rho
\\ \nonumber & = \frac{\langle\xi^F\rangle_1}{4}\medint\int_0^T\medint\int_{\partial U} \big(s(\nabla\cdot s^t)\cdot\nu\big) -\frac{\langle\xi^F\rangle_1}{4}\medint\int_0^T\medint\int_U \nabla\cdot (s(\nabla\cdot s^t))=0.
\end{align}
For the stochastic integral in \eqref{u_1}, we have simply that, for every $\beta\in(0,1)$,
\begin{equation}\label{u_03}\E\Big[\medint\int_0^T\medint\int_{U}(\nicefrac{\beta}{2})^{-1}\mathbf{1}_{[\nicefrac{\beta}{2}\leq\rho\leq \beta]}\sqrt{\rho}\nabla\rho\cdot s\dd\xi^F\Big]=0.\end{equation}
Finally, for the first term on the lefthand side of \eqref{u_1}, since we have by definition that $\psi_\beta(\rho)\rightarrow\rho$ as $\beta\rightarrow 0$, we have using the preservation of mass \eqref{d_mass} that
\begin{equation}\label{u_04} \lim_{\beta\rightarrow 0}\Big(\medint\int_{U}\psi_\beta(\rho)\Big|_{r=0}^{r=T}\Big)=\medint\int_U \rho(x,T)-\medint\int_{U}\rho_0(x) = 0.\end{equation}
Returning to \eqref{u_1}, it follows from \eqref{u_01}, \eqref{u_02}, \eqref{u_03}, and \eqref{u_04} that
\[\lim_{\beta\rightarrow 0}\E\Big[\medint\int_0^T\medint\int_{U}\medint\int_\R(\nicefrac{\beta}{2})^{-1}\mathbf{1}_{[\nicefrac{\beta}{2},\beta]} q\Big]=0,\]
from which the final claim follows using Fatou's lemma.   \end{proof}

\begin{lem}\label{u_lem}  Let $(X,\mathcal{S},\mu)$ be a measure space and let $K\in\mathbb{N}$.  If $\{f_k\colon X\rightarrow\mathbb{R}\}_{k\in\{1,2,\ldots,K\}}\subseteq L^1(X)$, and if for every $k\in\{1,2,\ldots,K\}$ the $\{B_{n,k}\}_{n\in\mathbb{N}}\subseteq \mathcal{S}$ are pairwise disjoint subsets, then
\[\liminf_{n\rightarrow\infty}\big(n\textstyle\sum_{k=1}^K\medint\int_{B_{n,k}}\abs{f_k}\dd\mu\big)=0.\]
\end{lem}

\begin{proof}   The proof follows by contradiction.  If not, for some $\ve\in(0,1)$ we have that
\[\liminf_{n\rightarrow\infty}\big(n\textstyle\sum_{k=1}^K\medint\int_{B_{n,k}}\abs{f_k}\dd\mu\big)\geq\ve,\]
which implies that there exists $N\in\mathbb{N}$ such that, for every $n\geq N$,
\[n\textstyle\sum_{k=1}^K\medint\int_{B_{n,k}}\abs{f_k}\dd\mu\geq \frac{\ve}{2}.\]
For every $k\in\{1,2,\ldots,K\}$ let $\mathcal{I}_{N,k}\subseteq[N,N+1,\ldots)$ be defined by
\[\mathcal{I}_{N,k}=\big\{n\in[N,N+1,\ldots)\colon \medint\int_{B_{n,k}}\abs{f_k}\dd\mu\geq \frac{\ve}{2Kn}\big\}.\]
Since $[N,N+1,\ldots)=\cup_{k=1}^K\mathcal{I}_{N,k}$ and $\textstyle\sum_{n=N}^\infty\frac{1}{n}=\infty$, there must exist at least one $k_0\in\{1,2,\ldots,K\}$ such that $\textstyle\sum_{n\in \mathcal{I}_{N,k_0}}\frac{1}{n}=\infty$.  This contradicts the assumption that $f_{k_0}\in L^1(X)$, since by disjointness of the $B_{n,k_0}$ we have that
\[\infty=\textstyle\sum_{n\in\mathcal{I}_{N,k_0}}\frac{1}{n}\leq \frac{2K}{\ve}\textstyle\sum_{n\in\mathcal{I}_{N,k_0}}\medint\int_{B_{n,k_0}}\abs{f_{k_0}}\dd\mu\leq \frac{2K}{\ve}\medint\int_X \abs{f_{k_0}}<\infty,\]
which completes the proof.  \end{proof}

\begin{thm}\label{thm_unique}  Under Assumptions~\ref{assume_d} and \ref{assume_n}, let $\rho_{1,0},\rho_{2,0}\in L^1(\O;L^1(U))$ be nonnegative and $\F_0$-measurable.  If $\rho_1,\rho_2$ are stochastic kinetic solutions of \eqref{i_eq} in the sense of Definition~\ref{sol_def} with initial data $\rho_{1,0},\rho_{2,0}$ respectively then, $\P$-a.s.,
\[\max_{t\in[0,T]}\norm{\rho_1(\cdot,t)-\rho_2(\cdot,t)}_{L^1(U)}\leq\norm{\rho_{1,0}-\rho_{2,0}}_{L^1(U)}.\]
\end{thm}

\begin{proof}   If either $\rho_{1,0}=0$ or $\rho_{2,0}=0$ the claim is an immediate consequence of the nonnegativity and preservation of mass \eqref{d_mass} in Definition~\ref{sol_def}.  We therefore assume that neither $\rho_{1,0}$ or $\rho_{2,0}$ are identically zero.  Let $\chi_1$ and $\chi_2$ be the kinetic functions of $\rho_1$ and $\rho_2$ and for every $\d\in(0,1)$ and $i\in\{1,2\}$ let $\chi^{\d}_{t,i}(y,\eta)=(\chi_{t,i}(y,\cdot)*\kappa^{\d})(\eta)$ for one-dimensional convolution kernels $\kappa^\d$ of scale $\d\in(0,1)$ on $\R$.  For every $\b\in(0,1)$ let $\varphi_\beta$ be defined by $\varphi_\beta(\eta)=0$ for $\eta\in[0,\nicefrac{\beta}{2}]$, $\varphi_\beta(\eta)=1$ for $\eta\in[\beta,\infty)$, and by the linear interpolation between $0$ and $1$ on $[\nicefrac{\beta}{2},\beta]$.  And, for every $M\in[2,\infty)$, let $\zeta_M(\eta)=1$ for $\eta\in[0,M]$, $\zeta_M(\eta)=0$ for $\eta\in[M+1,\infty)$, and be defined by the linear interpolation between $1$ and $0$ on the interval $[M,M+1]$.

Due to the fact that the kinetic functions $\chi_i$ are $\{0,1\}$-valued, for every $t\in[0,T]$,
\begin{align}\label{u_5}
\norm{\rho_1(x,t)-\rho_2(x,t)}_{L^1(U)} & = \norm{\chi_{t,1}(\xi,x)-\chi_{t,2}(\xi,x)}_{L^2(\R\times U)}
\\ \nonumber &  = \medint\int_\R\medint\int_{U}\chi_{t,1}(\xi,x)+\chi_{t,2}(\xi,x)-2\chi_{t,2}(\xi,x)\chi_{t,2}(\xi,x)\dx\dxi
\\ \nonumber & =\lim_{\beta\rightarrow 0}\lim_{M\rightarrow\infty}\lim_{\d\rightarrow 0}\medint\int_\R\medint\int_{U}\big(\chi^{\d}_{t,1}+\chi^{\d}_{t,2}-2\chi^{\d}_{t,1}\chi^{\d}_{t,2}\big)\varphi_\beta\zeta_M.
\end{align}
We will differentiate the righthand side of \eqref{u_5} and then justify taking the limits $\d\rightarrow 0$, $M\rightarrow\infty$, and $\beta\rightarrow 0$ in that order.  We will first differentiate the singleton terms involving $\chi^\d_{t,i}$ for $i\in\{1,2\}$.  It is a consequence of equation \eqref{sol_eq} in Definition~\ref{sol_def} that, for every $\d\in(0,1)$ and $\eta\in (\d,\infty)$---which guarantees that the support of $\kappa^\d(v-\eta)$ is contained compactly in $v\in(0,\infty)$, and that the support of  $\overline{\kappa}^{\d}_{s,i}(y,\eta)=\kappa^{\d}(\rho_i(y,s)-\eta)$ is contained in $\{\rho_i>0\}$---we have that
\begin{align}\label{u_9}
 \medint\int_U\chi^{\d}_{s,i}(y,\eta)\rvert_{s=0}^t & =\medint\int_\R\medint\int_U\chi^\d_{s,i}(y,v)\kappa^\d(v-\eta)\rvert_{s=0}^t
\\ \nonumber & = \medint\int_0^t\partial_\eta\Big(\medint\int_\R\medint\int_U \kappa^\d q_i\Big)-\frac{1}{2}\medint\int_0^t\partial_\eta\Big(\medint\int_U \overline{\kappa}^\d_{s,i}\rho_i \langle \nabla\cdot s\xi^F\rangle_1\Big)
\\ \nonumber & \quad -\frac{\langle\xi^F\rangle_1}{4}\medint\int_0^t\partial_\eta\Big(\medint\int_{U}\overline{\kappa}^\d_{s,i}(y,\eta)s(\nabla\cdot s^t)\cdot\nabla\rho\Big)-\medint\int_0^t\partial_\eta\Big(\medint\int_U\overline{\kappa}^\d_{s,i}\sqrt{\rho}\nabla\rho\cdot s\dd\xi^F\Big),
\end{align}
where due to the choice of $\d$ and $\eta$ we have that $\mathbf{1}_{\{\rho_i>0\}}=1$ on the support of $\overline{\kappa}^\d_{s,i}$.  It is a consequence of \eqref{u_9} that, for every $\beta\in(0,1)$ and $\d\in(0,\nicefrac{\beta}{2})$,
\begin{equation}\label{u_18}
\medint\int_\R\medint\int_{U}\chi^{\d}_{s,i}(y,\eta)\varphi_\beta(\eta)\zeta_M(\eta)\dy\deta\rvert^t_{s=0} = I^{i,\textrm{cut}}_t+I^{i,\textrm{mart}}_t,
\end{equation}
for the cutoff term defined by
\begin{align*}
& I^{i,\textrm{cut}}_t =-\medint\int_0^t\medint\int_{\R}\medint\int_{U}(\kappa^{\d}*q_i)(y,\eta)\partial_\eta(\varphi_\beta(\eta)\zeta_M(\eta))
 \\  & +\frac{1}{2}\medint\int_0^t\medint\int_\R\medint\int_{U}(\langle\nabla\cdot s\xi^F\rangle_1 \rho_i) \big(\overline{\kappa}^{\d}_{s,i}\partial_\eta(\varphi_\beta\zeta_M)\big) +\frac{\langle\xi^F\rangle_1}{4}\medint\int_0^t\medint\int_\R\medint\int_{U} (s(\nabla\cdot s^t)\cdot\nabla\rho_i)\big(\overline{\kappa}^{\d}_{s,i}\partial_\eta(\varphi_\beta\zeta_M)\big),
 \end{align*}
where as in the above, the choice of $\beta$ and $\d$ guarantees that $\mathbf{1}_{\{\rho_i>0\}}=1$ on the support of $\overline{\kappa}^\d_{s,i}\varphi_\beta$.  The martingale term in \eqref{u_18} is defined by
\begin{align*} & I^{i,\textrm{mart}}_t = \medint\int_0^t\medint\int_\R\medint\int_{U}\big(\overline{\kappa}^{\d}_{s,i}\partial_\eta(\varphi_\beta\zeta_M)\big)\sqrt{\rho_i}\nabla\rho_i\cdot s\dd\xi^F.
\end{align*}
Observe that $I^{i,\textrm{cut}}_t$ and $I^{i,\textrm{mart}}_t$ also depend on $\d,\beta\in(0,1)$ and $M\in\N$.

It remains to treat the mixed term on the righthand side of \eqref{u_5}.  We first observe that, for every $\d\in(0,1)$ and $\eta\in(\d,\infty)$ the local regularity of the solutions \eqref{d_s2} implies that as a function of the spatial variable $\chi^\d_{s,i}\in L^2([0,T];H^1(U))$ with
\begin{equation}\label{u_020}\nabla_x\chi^\d_{s,i} = \overline{\kappa}^\d_{s,i}\nabla\rho_i\;\;\textrm{and}\;\;\partial_\eta\chi^\d_{s,i} = \kappa^\d(-\eta)-\overline{\kappa}^\d_{s,i},\end{equation}
where, due to the choice of $\b\in(0,1)$ and $\d\in (0,\nicefrac{\beta}{2})$, we have that $\kappa^\d(-\eta)\varphi_\beta(\eta)=0$.  It then follows from Ito's formula (see, for example, Krylov \cite{Kry2013}, and see also Section~\ref{sec_Ito}) and \eqref{u_020} that, for every $\beta\in(0,1)$ and $\d\in (0,\nicefrac{\beta}{2})$, 
\begin{equation}\label{u_10}
\medint\int_\R\medint\int_U \chi^{\d}_{s,1}(y,\eta)\chi^{\d}_{s,2}(y,\eta)\varphi_\beta(\eta)\zeta_M(\eta)\rvert_{s=0}^t = I^{\textrm{meas}}_t+I^{\textrm{mix},\textrm{mart}}_t+I^{\textrm{cov}}_t+I^{\textrm{mix},\textrm{cut}}_t,
\end{equation}
for the measure, martingale, covariance, and cutoff terms defined as follows.  The measure term $I^{\textrm{meas}}_t$ is defined by
\begin{align*}
I^{\textrm{meas}}_t & = -\medint\int_0^t\medint\int_\R\medint\int_U \nabla\rho_2\cdot a\nabla\phi(\rho_1)\overline{\kappa}^\d_{s,1}\overline{\kappa}^\d_{s,2}\varphi_\beta\zeta_M-\medint\int_0^t\medint\int_\R\medint\int_U \nabla\rho_1\cdot a\nabla\phi(\rho_2)\overline{\kappa}^\d_{s,1}\overline{\kappa}^\d_{s,2}\varphi_\beta\zeta_M
\\ & \quad +\medint\int_0^t\medint\int_\R\medint\int_U \big(\kappa^\d*q_1\big)\overline{\kappa}^\d_{2,s}+\medint\int_0^t\medint\int_\R\medint\int_U \big(\kappa^\d*q_2\big)\overline{\kappa}^\d_{1,s}.
\end{align*}
The mixed martingale term $I^{\textrm{mix},\textrm{mart}}_t$ is defined by
\begin{align*}
& I^{\textrm{mix},\textrm{mart}}_t  = \medint\int_0^t\medint\int_\R\medint\int_U \varphi_\beta\zeta_M\sqrt{\rho_1}\overline{\kappa}^\d_{s,1}\overline{\kappa}^\d_{s,2}\nabla\rho_2\cdot\dd \xi^F+ \medint\int_0^t\medint\int_\R\medint\int_U \varphi_\beta\zeta_M\sqrt{\rho_2}\overline{\kappa}^\d_{s,2}\overline{\kappa}^\d_{s,1}\nabla\rho_1\cdot\dd \xi^F
\\ & \quad -\medint\int_0^t\medint\int_\R\medint\int_U\varphi_\beta\zeta_M\sqrt{\rho_1}\overline{\kappa}^\d_{s,1}\overline{\kappa}^\d_{s,2}\nabla\rho_1\cdot\dd \xi^F-\medint\int_0^t\medint\int_\R\medint\int_U\varphi_\beta\zeta_M\sqrt{\rho_2}\overline{\kappa}^\d_{s,1}\overline{\kappa}^\d_{s,2}\nabla\rho_2\cdot\dd \xi^F
\\ & \quad +\medint\int_0^t\medint\int_\R\medint\int_U \partial_\eta(\varphi_\beta\zeta_M)\sqrt{\rho_1}\overline{\kappa}^\d_{s,1}\chi^\d_{s,2}\nabla\rho_1\cdot \dd\xi^F+\medint\int_0^t\medint\int_\R\medint\int_U \partial_\eta(\varphi_\beta\zeta_M)\sqrt{\rho_2}\overline{\kappa}^\d_{s,2}\chi^\d_{s,1}\nabla\rho_2\cdot \dd\xi^F.
\end{align*}
The covariation term $I^{\textrm{cov}}_t$ takes the form
\begin{align}\label{u_14}
I^{\textrm{cov}}_t  & = - \frac{\langle\xi^F\rangle_1}{8}\medint\int_0^t\medint\int_\R\medint\int_{U}\Big((\frac{1}{\rho_1}a\nabla\rho_1\cdot\nabla\rho_2)+(\frac{1}{\rho_2}a\nabla\rho_2\cdot\nabla\rho_1\Big)\overline{\kappa}^\d_{s,1}\overline{\kappa}^\d_{s,2}\varphi_\beta\zeta_M
\\ \nonumber & \quad -\frac{\langle \xi^F\rangle_1}{2}\medint\int_0^t\medint\int_\R\medint\int_{U}\Big((s(\nabla\cdot s^t)\cdot \nabla\rho_2)+(s(\nabla\cdot s^t)\cdot\nabla\rho_1)\Big)\overline{\kappa}^\d_{s,1}\overline{\kappa}^\d_{s,2}\varphi_\beta\zeta_M
\\ \nonumber & \quad - \frac{1}{2}\medint\int_0^t\medint\int_\R\medint\int_{U}\Big((\langle\nabla\cdot s\xi^F\rangle_1 \rho_1)+(\langle\nabla\cdot s\xi^F\rangle_1 \rho_2)\Big)\overline{\kappa}^\d_{s,1}\overline{\kappa}^\d_{s,2}\varphi_\beta\zeta_M
\\ \nonumber & \quad +\textstyle\sum_{k=1}^\infty\textstyle\sum_{j=1}^d\medint\int_0^t\medint\int_\R\medint\int_U \Big(\textstyle\sum_{i=1}^d\partial_i(\sqrt{\rho_1}f_ks_{ij})\Big)\Big(\textstyle\sum_{r=1}^d\partial_r(\sqrt{\rho_2}f_ks_{rj})\Big)\overline{\kappa}^\d_{s,1}\overline{\kappa}^\d_{s,2}\varphi_\beta\zeta_M,
\end{align}
where the second term on the righthand side of \eqref{u_14} accounts for both the $x$-derivatives and $\eta$-derivatives of $\chi^\d_{s,i}$ appearing in the final terms of the second and third line of \eqref{sol_eq}.  The $\eta$-derivative of the cutoff term appears in $I^{\textrm{mix},\textrm{cut}}_t$.  It then follows from the structure of the noise in Assumption~\ref{assume_n} that the final term of \eqref{u_14} satisfies
\begin{align}\label{u_033}
&  \textstyle\sum_{k=1}^\infty\textstyle\sum_{j=1}^d\medint\int_0^t\medint\int_\R\medint\int_U \Big(\textstyle\sum_{i=1}^d\partial_i(\sqrt{\rho_1}f_ks_{ij})\Big)\Big(\textstyle\sum_{r=1}^d\partial_r(\sqrt{\rho_2}f_ks_{rj})\Big)\varphi_\beta\zeta_M
\\ \nonumber & = \frac{\langle\xi^F\rangle_1}{4}\medint\int_0^t\medint\int_\R\medint\int_U \frac{1}{\sqrt{\rho_1}\sqrt{\rho_2}} s^t\nabla\rho_1\cdot s^t\nabla\rho_2\overline{\kappa}^\d_{s,1}\overline{\kappa}^\d_{s,2}\varphi_\beta\zeta_M
\\ \nonumber & \quad + \frac{\langle\xi^F\rangle_1}{2}\medint\int_0^t\medint\int_\R\medint\int_U \Big(\frac{\sqrt{\rho_1}}{\sqrt{\rho_2}} s(\nabla\cdot s^t)\nabla\rho_2+\frac{\sqrt{\rho_2}}{\sqrt{\rho_1}} s(\nabla\cdot s^t)\nabla\rho_1\Big)\overline{\kappa}^\d_{s,1}\overline{\kappa}^\d_{s,2}\varphi_\beta\zeta_M
\\ \nonumber & \quad  +\medint\int_0^t\medint\int_\R\medint\int_U \langle \nabla\cdot s\xi^F\rangle_1\sqrt{\rho_1}\sqrt{\rho_2}\overline{\kappa}^\d_{s,1}\overline{\kappa}^\d_{s,2}\varphi_\beta\zeta_M,
\end{align}
and therefore, in view of \eqref{u_14}, \eqref{u_033}, and $a=ss^t$,
\begin{align}\label{u_17}
 I^{\textrm{cov}}_t  & = - \frac{\langle\xi^F\rangle_1}{8}\medint\int_0^t\medint\int_\R\medint\int_{U}\Big(\frac{1}{\sqrt{\rho_1}}-\frac{1}{\sqrt{\rho_2}}\Big)^2\big(s^t\nabla\rho_1\cdot s^t\nabla\rho_2\big)\overline{\kappa}^\d_{s,1}\overline{\kappa}^\d_{s,2}\varphi_\beta\zeta_M
\\ \nonumber &\quad  -\frac{\langle \xi^F\rangle_1}{2}\medint\int_0^t\medint\int_\R\medint\int_{U}\big(1-\frac{\sqrt{\rho_1}}{\sqrt{\rho_2}}\big)\big((s(\nabla\cdot s^t)\cdot \nabla\rho_2\big)\overline{\kappa}^\d_{s,1}\overline{\kappa}^\d_{s,2}\varphi_\beta\zeta_M
\\ \nonumber &\quad  -\frac{\langle\xi^F\rangle_1}{2}\medint\int_0^t\medint\int_\R\medint\int_U \big(1-\frac{\sqrt{\rho_2}}{\sqrt{\rho_1}}\big)\big(s(\nabla\cdot s^t)\cdot\nabla\rho_1)\big)\overline{\kappa}^\d_{s,1}\overline{\kappa}^\d_{s,2}\varphi_\beta\zeta_M
\\ \nonumber &\quad - \frac{1}{2}\medint\int_0^t\medint\int_\R\medint\int_{U}\Big(\sqrt{\rho_1}-\sqrt{\rho_2}\Big)^2\langle\nabla\cdot s\xi^F\rangle_1\overline{\kappa}^\d_{s,1}\overline{\kappa}^\d_{s,2}\varphi_\beta\zeta_M.
\end{align}
Finally, the cutoff term  $I^{\textrm{mix},\textrm{cut}}_t$ is defined by
\begin{align*}
& I^{\textrm{mix},\textrm{cut}}_t =-\medint\int_0^t\medint\int_{\R}\medint\int_{U}\Big((\kappa^{\d}*q_1)\chi^{\d}_{s,2}\partial_\eta(\varphi_\beta\zeta_M)+(\kappa^{\d}*q_2)\chi^{\d}_{s,1}\partial_\eta(\varphi_\beta\zeta_M)\Big)
\\ & \quad +\frac{1}{2}\medint\int_0^t\medint\int_\R\medint\int_{U}\Big((\langle\nabla\cdot s\xi^F\rangle_1 \rho_1\overline{\kappa}^{\d}_{s,1})\chi^{\d}_{s,2}\partial_\eta(\varphi_\beta\zeta_M)+(\langle\nabla\cdot s\xi^F\rangle_1 \rho_2\overline{\kappa}^{\d}_{s,2})\chi^{\d}_{s,1}\partial_\eta(\varphi_\beta\zeta_M)\Big)
\\ & \quad +\frac{\langle\xi^F\rangle_1}{4}\medint\int_0^t\medint\int_\R\medint\int_{U}\Big((s(\nabla\cdot s^t)\cdot\nabla\rho_1\overline{\kappa}^{\d}_{s,1})\chi^{\d}_{s,2}\partial_\eta(\varphi_\beta\zeta_M)+(s(\nabla\cdot s^t)\cdot\nabla\rho_2\overline{\kappa}^{\d}_{s,2})\chi^{\d}_{s,1}\partial_\eta(\varphi_\beta\zeta_M)\Big).
 \end{align*}
Returning to \eqref{u_5}, it follows from \eqref{u_18} and \eqref{u_10} that, for the martingale and cutoff terms defined respectively by the rule $I^{\textrm{mart}}_t= I^{1,\textrm{mart}}_t+I^{2,\textrm{mart}}_t-2I^{\textrm{mix},\textrm{mart}}_t$,
\begin{align}\label{u_40}
& \medint\int_\R\medint\int_{U}\big(\chi^{\d}_{s,1}+\chi^{\d}_{s,2}-2\chi^{\d}_{s,1}\chi^{\d}_{s,2}\big)\varphi_\beta\zeta_M\rvert_{s=0}^t
\\ \nonumber & =-2I^{\textrm{meas}}_t-2I^{\textrm{cov}}_t+I^{\textrm{mart}}_t+I^{\textrm{cut}}_t.
\end{align}
We will handle the four terms on the righthand side of \eqref{u_40} separately.

\textbf{The measure term}.  We first rewrite the measure term in the form
\begin{align}\label{measure_1}
I^{\textrm{meas}}_t & = -\medint\int_0^t\medint\int_\R\medint\int_U (\phi'(\rho_1)^\frac{1}{2}-\phi'(\rho_2)^\frac{1}{2})^2\nabla\rho_2\cdot a\nabla\rho_1\overline{\kappa}^\d_{s,1}\overline{\kappa}^\d_{s,2}\varphi_\beta\zeta_M
\\ \nonumber & \quad -2\medint\int_0^t\phi'(\rho_1)^\frac{1}{2}\phi'(\rho_2)^\frac{1}{2}\medint\int_\R\medint\int_U \nabla\rho_2\cdot a\nabla\rho_1\overline{\kappa}^\d_{s,1}\overline{\kappa}^\d_{s,2}\varphi_\beta\zeta_M
\\ \nonumber & \quad +\medint\int_0^t\medint\int_\R\medint\int_U \big(\kappa^\d*q_1\big)\overline{\kappa}^\d_{2,s}+\medint\int_0^t\medint\int_\R\medint\int_U \big(\kappa^\d*q_2\big)\overline{\kappa}^\d_{1,s}.
\end{align}
The support of the convolution kernels, the definition of the cutoff functions, the boundedness in $\d\in(0,1)$ of $(\d\overline{\kappa}^\d_{s,1})$, the local $\C^{1,\nicefrac{1}{2}}$-regularity of $\phi$, the local regularity of the solutions \eqref{d_s2}, and the dominated convergence theorem prove that the first term on the righthand side satisfies, for some $c\in(0,\infty)$,
\begin{align*}
& \limsup_{\d\rightarrow 0}\abs{\medint\int_0^t\medint\int_\R\medint\int_U (\phi'(\rho_1)^\frac{1}{2}-\phi'(\rho_2)^\frac{1}{2})^2\nabla\rho_2\cdot a\nabla\rho_1\overline{\kappa}^\d_{s,1}\overline{\kappa}^\d_{s,2}\varphi_\beta\zeta_M}
\\ & \leq c\limsup_{\d\rightarrow 0}\abs{\medint\int_0^t\medint\int_\R\medint\int_U\mathbf{1}_{\{0<\abs{\rho_1-\rho_2}<c\d\}}\abs{\nabla\rho_1}\abs{\nabla\rho_2}(\d\overline{\kappa}^\d_{s,1})\overline{\kappa}^\d_{s,2}\varphi_\beta\zeta_M}=0.
\end{align*}
Since the regularity property of the kinetic measures \eqref{d_s3}, H\"older's inequality, and Young's inequality prove that, $\P$-a.s.\ for every $t\in[0,T]$ the final two terms on the righthand side of \eqref{measure_1} are nonnegative, we conclude $\P$-a.s.\ for every $t\in[0,T]$ that
\begin{equation}\label{u_41} \limsup_{\d\rightarrow0}(-2I^{\textrm{meas}}_t)\leq 0.\end{equation}

\textbf{The covariance term}.  It follows from the local Lipschitz continuity of the functions $\sqrt{\eta}$ and $\nicefrac{1}{\sqrt{\eta}}$ on $(0,\infty)$, the definitions of the cutoff functions $\varphi_\beta$ and $\zeta_M$, the fact that $\abs{\rho_1-\rho_2}<2\d$ on the support of the convolution kernels $\overline{\kappa}^\d_{s,1}\overline{\kappa}^\d_{s,2}$, the choice $\d\in(0,\nicefrac{\b}{2})$, the boundedness and $\C^2$-regularity of $s$, and \eqref{u_17} that the covariance term satisfies, $\P$-a.s.\ for every $t\in[0,T]$, for some $c\in(0,\infty)$ depending on $\beta$ and $M$,
\begin{align}\label{u_042}
 \lim_{\ve\rightarrow 0}\abs{I^{\textrm{cov}}_t} & \leq  c\langle\xi^F\rangle_1\d\medint\int_0^t\medint\int_\R\medint\int_U \mathbf{1}_{\{0<\abs{\rho_1-\rho_2}<2\d\}}\abs{\nabla\rho_1}\abs{\nabla\rho_2}(\d\overline{\kappa}^{\d}_{s,1})\overline{\kappa}^{\d}_{s,2}\varphi_\beta\zeta_M
 \\ \nonumber & \quad +c\langle\xi^F\rangle_1\medint\int_0^t\medint\int_\R\medint\int_U \mathbf{1}_{\{0<\abs{\rho_1-\rho_2}<2\d\}} \big(\abs{\nabla\rho_1}+\abs{\nabla\rho_2}\big)(\d\overline{\kappa}^{\d}_{s,1})\overline{\kappa}^{\d}_{s,2}\varphi_\beta\zeta_M
\\ \nonumber & \quad + c \d\medint\int_0^t\medint\int_\R\medint\int_U \mathbf{1}_{\{0<\abs{\rho_1-\rho_2}<c\d\}} \langle \nabla\cdot s\xi^F\rangle_1 (\delta\overline{\kappa}^\d_{s,1})\overline{\kappa}^\d_{s,2}\varphi_\beta\zeta_M.
\end{align}
The boundedness of $(\d\overline{\kappa}^\d_{s,1})$ independently of $\d\in(0,1)$ allows to bound and integrate away the convolution kernels.  It then follows using the local regularity of the solutions \eqref{d_s2} that the first and third terms on the righthand side of \eqref{u_042} are of order $\d$.  The dominated convergence theorem proves that the second term vanishes in the limit $\d\rightarrow 0$.  This completes the analysis of the covariance term, for which we have, $\P$-a.s.\ for every $t\in[0,T]$,
\begin{equation}\label{u_46}\limsup_{\d\rightarrow 0}\abs{2I^{\textrm{cov}}_t}=0.\end{equation}

\textbf{The martingale term}.  The martingale term takes the form
\begin{align}\label{u_50}
I^{\textrm{mart}}_t & =  \medint\int_0^t\medint\int_\R\medint\int_{U}(1-2\chi^\d_{s,2})\big(\overline{\kappa}^{\d}_{s,1}\partial_\eta(\varphi_\beta\zeta_M)\big)\sqrt{\rho_1} s^t\nabla\rho_1\cdot \dd\xi^F
\\ \nonumber & \quad +\medint\int_0^t\medint\int_\R\medint\int_U(1-2\chi^\d_{s,1})\big(\overline{\kappa}^{\d}_{s,2}\partial_\eta(\varphi_\beta\zeta_M)\big)\sqrt{\rho_2} s^t\nabla\rho_2\cdot \dd\xi^F
\\ \nonumber & \quad + \medint\int_0^t\medint\int_\R\medint\int_{U}\overline{\kappa}^\d_{s,1}\overline{\kappa}^\d_{s,2}\varphi_\beta\zeta_M \big(\sqrt{\rho_2}-\sqrt{\rho_1}\big)\nabla\rho_1\cdot s\dd\xi^F
\\ \nonumber & \quad + \medint\int_0^t\medint\int_\R\medint\int_{U}\overline{\kappa}^\d_{s,1}\overline{\kappa}^\d_{s,2}\varphi_\beta\zeta_M \big(\sqrt{\rho_1}-\sqrt{\rho_2}\big)\nabla\rho_2\cdot s\dd\xi^F.
\end{align}
For the first two terms on the righthand side of \eqref{u_50}, an explicit computation proves that pointwise and therefore in $L^p(U\times[0,T])$ for every $p\in[1,\infty)$,
\begin{equation}\label{u_51}\lim_{\d\rightarrow 0}\medint\int_\R\overline{\kappa}^\d_{s,1}(1-2\chi^\d_{s,2})\big)\partial_\eta(\varphi_\beta\zeta_M) =  \sgn(\rho_1-\rho_2)\partial_\eta( \varphi_\beta\zeta_M)(\rho_1).\end{equation}
The final two terms on the righthand side of \eqref{u_50} are treated analogously to the second term on the righthand side of \eqref{u_042}.  Using the local Lipschitz continuity of $\sqrt{\eta}$ on $(0,\infty)$, the definition of the cutoff functions, the fact that $\abs{\rho_1-\rho_2}<2\d$ on the support of the convolution kernels $\overline{\kappa}^\d_{s,1}\overline{\kappa}^\d_{s,2}$, and the Burkh\"older--David--Gundy inequality (see, for example, Revuz and Yor \cite[Chapter~4, Theorem~4.1]{RevYor1999}), after passing to a deterministic subsequence $\d\rightarrow 0$, these terms vanish $\P$-a.s.\ for every $t\in[0,T]$.  We therefore have using \eqref{u_51} and the Burkh\"older--David--Gundy inequality that, after passing to a deterministic subsequence $\d\rightarrow 0$, $\P$-a.s.\ for every $t\in[0,T]$,
\begin{align}\label{u_52}
\lim_{\d\rightarrow 0} I^{\textrm{mart}}_t & =  \medint\int_0^t\medint\int_{U}\sgn(\rho_1-\rho_2)\partial_\eta(\varphi_\beta\zeta_M)(\rho_1)\sqrt{\rho_1}s^t\nabla\rho_1\cdot \dd\xi^F
\\ \nonumber & \quad +\medint\int_0^t\medint\int_U\sgn(\rho_2-\rho_1)\partial_\eta(\varphi_\beta\zeta_M)(\rho_2)\sqrt{\rho_2}s^t\nabla\rho_2\cdot \dd\xi^F.
\end{align}
It then follows from \eqref{u_52}, the definition of the cutoff functions, the Burkh\"older--David--Gundy inequality, H\"older's inequality, the boundedness of $s$, the preservation of mass \eqref{d_mass}, and the structure of the noise in Assumption~\ref{assume_n} that, for some $c\in(0,\infty)$ depending on the measure $\abs{U}$,
\begin{align}\label{u_53}
 \E\big[\max_{t\in[0,T]}\abs{\lim_{\d\rightarrow 0} I^{\textrm{mart}}_t}\big] \leq & c\langle \xi^F\rangle^\frac{1}{2}_1\textstyle\sum_{i=1}^2\Big(\E\big[\big(\medint\int_0^T\medint\int_U\mathbf{1}_{\{\nicefrac{\b}{2}\leq \rho_i\leq \beta\}}\beta^{-1}\abs{s^t\nabla\rho_i}^2\big)^\frac{1}{2}\big]
\\ \nonumber & \quad +c\textstyle\sum_{i=1}^2\E\big[\langle\rho_0\rangle_U^\frac{1}{2}\big(\medint\int_0^T\medint\int_U\mathbf{1}_{\{M\leq \rho_i\leq M+1\}}\abs{s^t\nabla\rho_i}^2\big)^\frac{1}{2}\big].
\end{align}
The regularity of the parabolic defect measure \eqref{d_s3}, the structure $a=ss^t$, and \eqref{u_53} then prove that
\begin{align}\label{u_54}
 \E\big[\max_{t\in[0,T]}\abs{\lim_{\d\rightarrow 0} I^{\textrm{mart}}_t}\big] \leq & c\langle \xi^F\rangle^\frac{1}{2}_1\textstyle\sum_{i=1}^2\Big(\E\big[\big(\beta^{-1}q_i(U\times[0,T]\times[\nicefrac{\b}{2},\b])\big)^\frac{1}{2}\big]
\\ \nonumber & \quad +c\textstyle\sum_{i=1}^2\E\big[\langle\rho_0\rangle_U^\frac{1}{2}\big( q_i(U\times[0,T]\times[M,M+1])\big)^\frac{1}{2}\big].
\end{align}
It is finally a consequence of the vanishing of the measures at infinity \eqref{d_s4}, Proposition~\ref{prop_zero}, \eqref{u_54}, and Fatou's lemma that, $\P$-a.s.\ for every $t\in[0,T]$ along random subsequences $M\rightarrow\infty$ and $\beta\rightarrow 0$,
\begin{equation}\label{u_55} \lim_{\beta\rightarrow 0}\lim_{M\rightarrow\infty}\lim_{\d\rightarrow 0}\abs{I^{\textrm{mart}}_t}=0,\end{equation}
which completes the analysis of the martingale term.

\textbf{The cutoff term.}  The cutoff term takes the form
\begin{align}\label{u_26}
& I^{\textrm{cut}}_t  =  \medint\int_0^t\medint\int_{\R}\medint\int_{U}\Big((\kappa^{\d}*q_1)(2\chi^{\d}_{s,2}-1)+(\kappa^{\d}*q_2)(2\chi^{\d}_{s,1}-1)\Big)\partial_\eta(\varphi_\beta\zeta_M)
\\ \nonumber & +\frac{1}{2}\medint\int_0^t\medint\int_\R\medint\int_{U}\Big((\langle\nabla\cdot s\xi^F\rangle_1 \rho_1\overline{\kappa}^{\d}_{s,1})(1-2\chi^{\d}_{s,2})+(\langle\nabla\cdot s\xi^F\rangle_1 \rho_2\overline{\kappa}^{\d}_{s,2})(1-2\chi^{\d}_{s,1})\Big)\partial_\eta(\varphi_\beta\zeta_M)
\\ \nonumber & +\frac{\langle\xi^F\rangle_1}{4}\medint\int_0^t\medint\int_\R\medint\int_{U}\Big((s(\nabla\cdot s^t)\cdot\nabla\rho_1\overline{\kappa}^{\d}_{s,1})(1-2\chi^{\d}_{s,2})+(s(\nabla\cdot s^t)\cdot\nabla\rho_2\overline{\kappa}^{\d}_{s,2})(1-2\chi^{\d}_{s,1})\Big)\partial_\eta(\varphi_\beta\zeta_M).
 \end{align}
The analysis leading to \eqref{u_51} and \eqref{u_26} and the boundedness of the kinetic functions prove that, $\P$-a.s.\ for every $t\in[0,T]$,
\begin{align}\label{u_27}
& \lim_{\d\rightarrow 0}I^{\textrm{cut}}_t  \leq   \medint\int_0^t\medint\int_{\R}\medint\int_{U}\abs{\partial_\eta(\varphi_\beta\zeta_M)}\big(q_1+q_2\big)
\\ \nonumber & +\frac{1}{2}\medint\int_0^t\medint\int_{U}\langle\nabla\cdot s\xi^F\rangle_1 \sgn(\rho_1-\rho_2) \big(\rho_1\partial_\eta(\varphi_\beta\zeta_M)(\rho_1)-\rho_2\partial_\eta(\varphi_\beta\zeta_M)(\rho_2)\big)
\\ \nonumber & +\frac{\langle\xi^F\rangle_1}{4}\medint\int_0^t\medint\int_{U} \sgn(\rho_1-\rho_2) s(\nabla\cdot s^t)\cdot\big(\partial_\eta(\varphi_\beta\zeta_M)(\rho_1)\nabla\rho_1-\partial_\eta(\varphi_\beta\zeta_M)(\rho_2)\nabla\rho_2\big).
 \end{align}
For the first term on the righthand side of \eqref{u_27}, we have using the definition of the convolution kernels that, $\P$-a.s.\ for every $t\in[0,T]$, 
\begin{align}\label{u_028}
& \medint\int_0^t\medint\int_{\R}\medint\int_{U}\abs{\partial_\eta(\varphi_\beta\zeta_M)}\big(q_1+q_2\big)
\\ \nonumber & \leq c\textstyle\sum_{i=1}^2\big(\beta^{-1}q_i(U\times[0,T]\times[\nicefrac{\beta}{2},\beta])+q_i(U\times[0,T]\times[M,M+1])\big),
\end{align}
which due to the vanishing of the measures at infinity \eqref{d_s4} and Proposition~\ref{prop_zero} vanishes along random subsequences $M\rightarrow\infty$ and $\beta\rightarrow 0$.  For the second term on the righthand side of \eqref{u_27}, we have that, $\P$-a.s.\ for every $t\in[0,T]$, for some $c\in(0,\infty)$ depending on $\langle\nabla\cdot s\xi^F\rangle_1$,
\begin{align}\label{u_029}
& \frac{1}{2}\medint\int_0^t\medint\int_{U}\langle\nabla\cdot s\xi^F\rangle_1 \sgn(\rho_1-\rho_2) \big(\rho_1\partial_\eta(\varphi_\beta\zeta_M)(\rho_1)-\rho_2\partial_\eta(\varphi_\beta\zeta_M)(\rho_2)\big)
\\ \nonumber & \leq c\textstyle\sum_{i=1}^2 \medint\int_0^T\medint\int_U \big(\mathbf{1}_{\{\nicefrac{\beta}{2}\leq\rho_i\leq\beta\}}+(M+1)\mathbf{1}_{\{M\leq \rho_i\leq M+1\}}\big),
\end{align}
which vanishes along subsequences $M\rightarrow\infty$ and $\beta\rightarrow 0$ using Lemma~\ref{u_lem} for the limit $M\rightarrow\infty$ and the dominated convergence theorem for the limit $\beta\rightarrow 0$.  For the final term on the righthand side of \eqref{u_27}, since $\varphi_\beta=1$ on the support of $\partial_\eta\zeta_M$ and $\zeta_M=1$ on the support of $\partial_\eta\varphi_\beta$, we have that
\begin{align}\label{u_30}
& \frac{\langle\xi^F\rangle_1}{4}\medint\int_0^t\medint\int_{U} \sgn(\rho_1-\rho_2) s(\nabla\cdot s^t)\cdot\big(\partial_\eta(\varphi_\beta\zeta_M)(\rho_1)\nabla\rho_1-\partial_\eta(\varphi_\beta\zeta_M)(\rho_2)\nabla\rho_2\big)
\\ \nonumber & =\frac{\langle\xi^F\rangle_1}{4}\medint\int_0^t\medint\int_{U} \sgn(\rho_1-\rho_2) s(\nabla\cdot s^t)\cdot\big(\partial_\eta(\zeta_M)(\rho_1)\nabla\rho_1-\partial_\eta(\zeta_M)(\rho_2)\nabla\rho_2\big)
\\ \nonumber & \quad +\frac{\langle\xi^F\rangle_1}{4}\medint\int_0^t\medint\int_{U} \sgn(\rho_1-\rho_2) s(\nabla\cdot s^t)\cdot\big(\partial_\eta(\varphi_\beta)(\rho_1)\nabla\rho_1-\partial_\eta(\varphi_\beta)(\rho_2)\nabla\rho_2\big).
\end{align}
For the first term on the righthand side of \eqref{u_30} it is a consequence of the $\C^1$-boundedness of $s$, the regularity of the kinetic measures \eqref{d_s3}, the vanishing of the kinetic measures at infinity \eqref{d_s4}, and H\"older's inequality that, $\P$-a.s.\ for every $t\in[0,T]$, for some $c\in(0,\infty)$,
\begin{align}\label{u_31}
& \liminf_{M\rightarrow\infty} \frac{\langle\xi^F\rangle_1}{4}\medint\int_0^t\medint\int_{U} \sgn(\rho_1-\rho_2) s(\nabla\cdot s^t)\cdot\big(\partial_\eta(\zeta_M)(\rho_1)\nabla\rho_1-\partial_\eta(\zeta_M)(\rho_2)\nabla\rho_2\big)
\\ \nonumber & \leq c\liminf_{M\rightarrow\infty} \textstyle\sum_{i=1}^2 q_i(U\times[0,T]\times[M,M+1])^\frac{1}{2} =0.
\end{align}
For the second term on the righthand side of \eqref{u_30}, we first observe from the definition of $\varphi_\beta$ and the local regularity of the solutions \eqref{d_s2} that $\varphi_\beta(\rho_i)$ has a trace on $\partial U$.  It is then a consequence of the $\C^2$-regularity of $s$ that, after integrating by parts, $\P$-a.s.\ for every $t\in[0,T]$, for $\nu$ the outer normal to $U$,
\begin{align}\label{u_32}
& \frac{\langle\xi^F\rangle_1}{4}\medint\int_0^t\medint\int_{U} \sgn(\rho_1-\rho_2) s(\nabla\cdot s^t)\cdot\big(\partial_\eta(\varphi_\beta)(\rho_1)\nabla\rho_1-\partial_\eta(\varphi_\beta)(\rho_2)\nabla\rho_2\big)
\\ \nonumber &  = \medint\int_U s(\nabla\cdot s^t)\cdot \nabla \Big(\medint\int_0^t\abs{\varphi_\beta(\rho_1)-\varphi_\beta(\rho_2)}\Big)
\\ \nonumber & = \frac{\langle\xi^F\rangle_1}{4}\medint\int_{\partial U}\Big(\medint\int_0^t \abs{\varphi_\beta(\rho_1)-\varphi_\beta(\rho_2)}\Big)\big(s(\nabla\cdot s^t)\cdot\nu\big)-\medint\int_0^t\medint\int_U \abs{\varphi_\beta(\rho_1)-\varphi_\beta(\rho_2)}\nabla\cdot\big(s(\nabla\cdot s^t)).
\end{align}
Since by definition $\varphi_\beta(\rho_i)\rightarrow\mathbf{1}_{\{\rho_i>0\}}$ as $\beta\rightarrow 0$, it is a consequence of the non-vanishing of the solutions \eqref{s_log} and Remark~\ref{remark_sol} that the righthand side of \eqref{u_32} vanishes in the limit $\beta\rightarrow 0$.  We therefore conclude using \eqref{u_028}, \eqref{u_029}, \eqref{u_30}, \eqref{u_31}, and \eqref{u_32} that, along random subsequences $M\rightarrow\infty$ and $\beta\rightarrow 0$, $\P$-a.s.\ for every $t\in[0,T]$,
\begin{equation}\label{u_33}\lim_{\beta\rightarrow 0}\lim_{M\rightarrow\infty}\lim_{\d\rightarrow 0} I^{\textrm{cut}}_t=0,\end{equation}
which completes the analysis of the cutoff term.

\textbf{The conclusion}. Estimates \eqref{u_40}, \eqref{u_41}, \eqref{u_46}, \eqref{u_55}, and \eqref{u_33} prove that, $\P$-a.s.\ for every $t\in[0,T]$, there exist random subsequences $\d,\beta\rightarrow 0$ and $M\rightarrow\infty$ such that
\begin{align*}
& \medint\int_\R\medint\int_{U}\abs{\chi_{s,1}-\chi_{s,2}}^2\rvert_{s=0}^{s=t}
\\ & =\lim_{\beta\rightarrow 0}\lim_{M\rightarrow\infty}\lim_{\d\rightarrow 0} \medint\int_\R\medint\int_{U}\big(\chi^{\d}_{s,1}+\chi^{\d}_{s,2}-2\chi^{\d}_{s,1}\chi^{\d}_{s,2}\big)\varphi_\beta\zeta_M\rvert_{s=0}^t
\\ \nonumber & = \lim_{\beta\rightarrow 0}\lim_{M\rightarrow\infty}\lim_{\d\rightarrow 0}\big(-2I^{\textrm{meas}}_t-2I^{\textrm{cov}}_t+I^{\textrm{mart}}_t+I^{\textrm{cut}}_t\big) \leq 0.
\end{align*}
We therefore have that, for every $t\in[0,T]$,
\begin{align*}
\medint\int_{U}\abs{\rho_1(\cdot,t)-\rho_2(\cdot,t)} & =\medint\int_\R\medint\int_{U}\abs{\chi_{t,1}-\chi_{t,2}}^2\leq \medint\int_\R\medint\int_{U}\abs{\overline{\chi}(\rho_{1,0})-\overline{\chi}(\rho_{2,0})}^2=\medint\int_{U}\abs{\rho_{1,0}-\rho_{2,0}},
\end{align*}
which completes the proof.  \end{proof}

\begin{remark}\label{remark_unique_1}We remark that, under Assumptions~\ref{assume_d} and \ref{assume_n}, Theorem~\ref{thm_unique} applies without changes to the equation
\[\partial_t\rho = \nabla\cdot a\nabla\phi(\rho)-\nabla\cdot (\phi^{\nicefrac{1}{2}}(\rho)\circ s\dd\xi),\]
and to more general noise terms $(\sigma(\rho)\circ s\dd\xi)$ for coefficients $\sigma$ satisfying \cite[Assumption~4.1]{FG21}.\end{remark}

\subsection{Existence of stochastic kinetic solutions}\label{sec_rkss_exist}

In this section, we will first establish the existence of solutions to the regularized equation
\begin{equation}\label{e_1}\partial_t\rho = \nabla\cdot a\nabla\phi(\rho)- \nabla\cdot(\sigma(\rho)s\dd\xi^F)+\frac{\langle\xi^F\rangle_1}{2}\nabla\cdot(\sigma'(\rho)^2a\nabla\rho)+\frac{\langle\xi^F\rangle_1}{2}\nabla\cdot(\sigma(\rho)\sigma'(\rho)s(\nabla\cdot s^t)),\end{equation}
for smooth and bounded approximation $\sigma$ of the square root and for nonnegative initial data $\rho_0\in L^2(\O;L^2(U))$.

\begin{definition}\label{sol_smooth}  Under Assumptions~\ref{assume_d} and \ref{assume_n}, for a smooth and bounded $\sigma\in\C^2(\R)$, and for a nonnegative, $\F_0$-measurable $\rho_0\in L^2(\O;L^2(U))$, a solution of \eqref{e_1} is a nonnegative, $\F_t$-predictable, continuous $L^2(U)$-valued process $\rho\in L^2(\O;L^2([0,T];H^1(U)))$ that satisfies, for every $t\in[0,T]$ and $\psi\in\C^\infty_c(U)$,
\begin{align*}
\medint\int_{U}\rho(x,s)\psi(x)\Big|_{s=0}^{s=t} & =-\medint\int_0^t\medint\int_{U}\nabla\psi\cdot a\nabla\phi(\rho)+\medint\int_0^t\medint\int_{U}\sigma(\rho)\nabla\psi\cdot s\dd\xi^F
\\ & \quad -\frac{\langle\xi^F\rangle_1}{2}\medint\int_0^t\medint\int_{U}\sigma'(\rho)^2a\nabla\rho\cdot\nabla\psi-\frac{\langle\xi^F\rangle_1}{2}\medint\int_0^t\medint\int_{U}\sigma(\rho)\sigma'(\rho)s(\nabla\cdot s^t)\cdot\nabla\psi.
\end{align*}
\end{definition}

\begin{prop}\label{prop_entropy}  Under Assumptions~\ref{assume_d} and \ref{assume_n}, for a smooth and bounded $\sigma\in\C^2(\R)$, and for a nonnegative, $\F_0$-measurable $\rho_0\in L^2(\O;L^2(U))$, there exists a unique solution of \eqref{e_1} in the sense of Definition~\ref{sol_smooth}.  \end{prop}

\begin{proof}  The proof is a simplified version of Theorems~\ref{thm_unique} and \ref{thm_rks_ex} based on a standard Galerkin approximation, the estimates of Proposition~\ref{prop_e}, which hold uniformly for the Galerkin approximates, and the Aubin--Lions--Simon lemma \cite{Aubin,pLions,Simon}. \end{proof}

We will now establish estimates for the solutions of \eqref{e_1} with a smooth coefficient $\sigma$.  Notice, in particular, that these estimates require $\sigma$ and its derivative to be bounded by those of the square root, although the coefficient $2$ is completely arbitrary and is not essential.  Then, in Proposition~\ref{prop_measure}, we establish the local regularity of the solutions in the sense of \eqref{d_s2} of Definition~\ref{sol_def}.

\begin{prop}\label{prop_e}   Under Assumptions~\ref{assume_d} and \ref{assume_n}, for a smooth and bounded $\sigma\in\C^2(\R)$  that satisfies
\begin{equation}\label{assume_sigma}\sigma(\eta)\leq 2\sqrt{\eta}\;\;\textrm{and}\;\;\sigma'(\eta)\leq \frac{1}{\sqrt{\eta}}\;\;\textrm{for every $\eta\in(0,\infty)$,}\end{equation}
and for a nonnegative, $\F_0$-measurable $\rho_0\in L^2(\O;L^2(U))$, let $\rho$ be a solution of \eqref{e_1} in the sense of Definition~\ref{sol_smooth} with initial data $\rho_0$.  Then, $\rho$ satisfies the following four estimates.

\begin{enumerate}[(i)]
\item \emph{Preservation of mass}:  $\P$-a.s.\ for every $t\in[0,T]$,
\begin{equation}\label{e_010}\norm{\rho(\cdot,t)}_{L^1(U)}= \norm{\rho_0}_{L^1(U)}.\end{equation}
\item \emph{The $L^p$-energy estimate}:  for every $p\in[1,\infty)$, for some $c\in(0,\infty)$ depending on $p$ and $T$,
\begin{align}\label{e_10}
& \max_{t\in[0,T]}\E[\medint\int_U\rho^{p+1}]+\E[\medint\int_0^t\medint\int_{U}\rho^{p-1}\abs{\nabla\rho}^2]
\\ \nonumber & \leq c\Big(\E[\medint\int_U\rho_0^{p+1}]+\norm{s(\nabla\cdot s^t)}_{L^\infty(U)}^{p+1}\langle\xi^F\rangle_1^{p+1}+\norm{\langle \nabla\cdot s\xi^F\rangle_1}_{L^\infty(U)}^{p+1}\Big).
\end{align}
\item \emph{The $L^2$-energy estimate}:  for every $q\in[1,\infty)$, for some $c\in(0,\infty)$ independent of $T$ but depending on $q$,
\begin{align}\label{e_011}
& \E\Big[\Big(\max_{t\in[0,T]}\medint\int_{U}\rho^2+\medint\int_0^T\medint\int_{U}\nabla\rho\cdot a\nabla\rho\Big)^q\Big] \leq c\Big(\E\Big[\Big(\medint\int_U\rho_0^2\Big)^q\Big]+\langle\xi^F\rangle^q_1\E[\langle\rho_0\rangle^q_U]\Big)
\\ \nonumber & \quad +c T^q \Big(\langle\xi^F\rangle_1^{2q}\norm{s(\nabla\cdot s^t)}^{2q}_{L^\infty(U)}+\norm{\langle\nabla\cdot s\xi^F\rangle_1}^{q}_{L^\infty(U)}\E[\langle\rho_0\rangle^q_U]\Big).
\end{align}
\item \emph{The entropy estimate}:  for $\Theta$ defined by $\Theta(1)=0$ and $\Theta'(\eta)=\frac{\sigma(\eta)\sigma'(\eta)}{\eta}$, for $\Psi(\eta)=\eta\log(\eta)-\eta$, and for every $q\in[1,\infty)$ there exists $c\in(0,\infty)$ independent of $T\in(0,\infty)$ but depending on $q$ such that
\begin{align}\label{e_11}
& \E \Big[\Big(\max_{t\in[0,T]}\medint\int_{U}\Psi(\rho)+\medint\int_0^T\medint\int_{U}\frac{1}{\rho}\nabla\rho\cdot a\nabla\rho\Big)^q\Big]\leq \E\Big[\Big(\medint\int_{U}\Psi(\rho_0)\Big)^q\Big]+c T^q\norm{\langle \nabla\cdot s\xi^F\rangle_1}^q_{L^\infty(U)}
\\ \nonumber & \quad +c\langle\xi^F\rangle^q_1\Big(1+\Big(\medint\int_0^T\medint\int_{\partial U} \abs{s(\nabla\cdot s^t)\cdot \nu}\abs{\Theta(\rho)}\Big)^q+\Big(\medint\int_0^T\medint\int_{U}\abs{\nabla\cdot (s(\nabla\cdot s^t))}\abs{\Theta(\rho)}\Big)^q\Big).
\end{align}
\end{enumerate}

\end{prop}

\begin{proof}  The preservation of mass \eqref{e_010} is an immediate consequence of the nonnegativity and choosing $\psi=1$ in Definition~\ref{sol_smooth}.  The energy estimate \eqref{e_10} is a consequence of It\^o's formula (see, for example, \cite{Kry2013}) applied to the function $\frac{\rho^{p+1}}{p(p+1)}$, for which we have $\P$-a.s.\ for every $t\in[0,T]$ that
\begin{align}\label{e_100}
& \medint\int_{U}\frac{\rho^{p+1}}{p(p+1)}\Big|_{s=0}^{s=t}+\medint\int_0^t\medint\int_{U}\rho^{p-1}\nabla\rho\cdot a\nabla\phi(\rho)=\medint\int_0^t\medint\int_{U}\sigma(\rho)\rho^{p-1}\nabla\rho\cdot s\dd\xi^F
\\ \nonumber & \quad +\frac{\langle\xi^F\rangle_1}{4}\medint\int_0^t\medint\int_{U}\sigma(\rho)\sigma'(\rho)\rho^{p-1}s(\nabla\cdot s^t)\cdot\nabla\rho+\frac{1}{2}\medint\int_0^t\medint\int_{U}\sigma(\rho)^2\rho^{p-1}\langle \nabla\cdot s\xi^F\rangle_1.
\end{align}
It is a consequence of the bounds on $\sigma$, the uniform ellipticity of $a$, the nondegeneracy of $\phi$ in Assumption~\ref{assume_d}, H\"older's inequality, and Young's inequality that, for some $c\in(0,\infty)$ depending on the ellipticity constants but independent of $p$, for every $t\in[0,T]$,
\begin{align*}
& \max_{s\in[0,t]}\big(\frac{1}{p(p+1)}\E[\rho^{p+1}]\big)+\E[\medint\int_0^t\medint\int_{U}\rho^{p-1}\abs{\nabla\rho}^2]\leq \frac{1}{p(p+1)}\E[\rho_0^{p+1}]
\\ & +\frac{c}{p+1}\big(\norm{s(\nabla\cdot s^t)}_{L^\infty(U)}^{p+1}\langle\xi^F\rangle_1^{p+1}+\norm{\langle \nabla\cdot s\xi^F\rangle_1}_{L^\infty(U)}^{p+1}\big)+\frac{c p}{p+1}\medint\int_0^t\medint\int_{U}\rho^{p+1}.
\end{align*}
It then follows from Gr\"onwall's inequality that, for some $c\in(0,\infty)$ depending on $T$ and $p$,
\begin{align*}
& \max_{t\in[0,T]}\E[\medint\int_U\rho^{p+1}]+\E[\medint\int_0^t\medint\int_{U}\rho^{p-1}\abs{\nabla\rho}^2]
\\ & \leq c\Big(\E[\medint\int_U\rho_0^{p+1}]+\norm{s(\nabla\cdot s^t)}_{L^\infty(U)}^{p+1}\langle\xi^F\rangle_1^{p+1}+\norm{\langle \nabla\cdot s\xi^F\rangle_1}_{L^\infty(U)}^{p+1}\Big),
\end{align*}
which completes the proof of the energy estimate.

We prove \eqref{e_011} by improving the above estimate in the case of $p=1$, which allows for a straightforward treatment of the stochastic integral.  Based on \eqref{e_100} in the case $p=1$, we have $\P$-a.s.\ for every $t\in[0,T]$ that
\begin{align*}
& \frac{1}{2}\medint\int_{U}\rho^2\Big|_{s=0}^{s=t}+\medint\int_0^t\medint\int_{U}\nabla\rho\cdot a\nabla\phi(\rho)=\medint\int_0^t\medint\int_{U}\sigma(\rho)\nabla\rho\cdot s\dd\xi^F
\\ & \quad +\frac{\langle\xi^F\rangle_1}{4}\medint\int_0^t\medint\int_{U}\sigma(\rho)\sigma'(\rho)s(\nabla\cdot s^t)\cdot\nabla\rho+\frac{1}{2}\medint\int_0^t\medint\int_{U}\sigma(\rho)^2\langle \nabla\cdot s\xi^F\rangle_1.
\end{align*}
It then follows from the bounds \eqref{assume_sigma} on $\sigma$, the uniform ellipticity of $a$ in Assumption~\ref{assume_d}, the preservation of mass \eqref{e_010}, H\"older's inequality, and Young's inequality that
\begin{align*}
& \max_{t\in[0,T]}\medint\int_{U}\rho^2+\medint\int_0^T\medint\int_{U}\nabla\rho\cdot a\nabla\phi(\rho) \leq c\Big(\medint\int_U\rho_0^2+\max_{t\in[0,T]}\abs{\medint\int_0^t\medint\int_{U}\sigma(\rho)\nabla\rho\cdot s\dd\xi^F}\Big)
\\ & \quad +c T \Big(\langle\xi^F\rangle_1^2\norm{s(\nabla\cdot s^t)}^2_{L^\infty(U)}+\langle\rho_0\rangle_U\norm{\langle\nabla\cdot s\xi^F\rangle_1}_{L^\infty(U)}\Big),
\end{align*}
and, therefore, for every $q\geq 1$ there exists $c\in(0,\infty)$ depending on $q$ such that
\begin{align*}
& \Big(\max_{t\in[0,T]}\medint\int_{U}\rho^2+\medint\int_0^T\medint\int_{U}\nabla\rho\cdot a\nabla\phi(\rho)\Big)^q \leq c\Big(\big(\medint\int_U\rho_0^2\big)^q+\max_{t\in[0,T]}\abs{\medint\int_0^t\medint\int_{U}\sigma(\rho)\nabla\rho\cdot s\dd\xi^F}^q\Big)
\\ & \quad +c T^q \Big(\langle\xi^F\rangle_1^{2q}\norm{s(\nabla\cdot s^t)}^{2q}_{L^\infty(U)}+\langle\rho_0\rangle^q_U\norm{\langle\nabla\cdot s\xi^F\rangle_1}^q_{L^\infty(U)}\Big).
\end{align*}
For the stochastic integral, we have using the Burkh\"older--Davis--Gundy inequality and the uniform ellipticity of $a$ in Assumption~\ref{assume_d} that, for every $\ve\in(0,1)$, for some $c\in(0,\infty)$ independent of $\ve$,
\begin{align*}
\E\big[\max_{t\in[0,T]}\abs{\medint\int_0^t\medint\int_{U}\sigma(\rho)\nabla\rho\cdot s\dd\xi^F}^q\big] & \leq c \E\big[\big(\medint\int_0^T \textstyle\sum_{k=1}^\infty\textstyle\sum_{i=1}^d\big(\medint\int_U \sigma(\rho)(s^t\nabla\rho)_if_k\big)^2\big)^\frac{q}{2}\big]
\\ & \leq c\langle\xi^F\rangle_1^\frac{q}{2}\E\big[\big(\langle \rho_0\rangle_1^\frac{q}{2}\big(\medint\int_0^T\medint\int_U \nabla\rho\cdot a\nabla\rho\big)^\frac{q}{2}\big]
\\ & \leq c \ve^{-1}\langle\xi^F\rangle^q_1\E[\langle\rho_0\rangle^q_U]+\ve\E\big[\big(\medint\int_0^T\medint\int_U\nabla\rho\cdot a\nabla\rho\big)^q\big].
\end{align*}
By choosing $\ve\in(0,1)$ sufficiently small, we have, for some $c\in(0,\infty)$ independent of $T$,
\begin{align*}
& \E\Big[\Big(\max_{t\in[0,T]}\medint\int_{U}\rho^2+\medint\int_0^T\medint\int_{U}\nabla\rho\cdot a\nabla\phi(\rho)\Big)^q\Big] \leq c\Big(\E\Big[\Big(\medint\int_U\rho_0^2\Big)^q\Big]+\langle\xi^F\rangle^q_1\E[\langle\rho_0\rangle^q_U]\Big)
\\ & \quad +c T^q \Big(\langle\xi^F\rangle_1^{2q}\norm{s(\nabla\cdot s^t)}^{2q}_{L^\infty(U)}+\norm{\langle\nabla\cdot s\xi^F\rangle_1}^{q}_{L^\infty(U)}\E[\langle\rho_0\rangle^q_U]\Big),
\end{align*}
which completes the proof of \eqref{e_011} using the nondegeneracy of $\phi$ in Assumption~\ref{assume_d}.

The entropy estimate \eqref{e_11} is a consequence of It\^o's formula (see, for example, \cite{Kry2013}) applied to the function $\Psi_\d$ defined by the properties that $\Psi_\d(0)=0$ and $\Psi_\d'(\eta)=\log(\eta+\d)$, for $\d\in(0,1)$, which is justified by approximation using the nonnegativity of the solution.  This yields that, $\P$-a.s.\ for every $t\in[0,T]$,
\begin{align}\label{e_15}
& \medint\int_{U}\Psi_\d(\rho)\Big|_{s=0}^{s=t}+\medint\int_0^t\medint\int_{U}\frac{1}{\rho}\nabla\rho\cdot a\nabla\phi(\rho)=\medint\int_0^t\medint\int_{U}\frac{\sigma(\rho)}{\rho+\d}\nabla\rho\cdot s\dd\xi^F
\\ \nonumber & \quad +\frac{\langle\xi^F\rangle_1}{4}\medint\int_0^t\medint\int_{U}\sigma(\rho)\sigma'(\rho)s(\nabla\cdot s^t)\cdot\big(\frac{1}{\rho+\d}\nabla\rho\big)+\frac{1}{2}\medint\int_0^t\medint\int_{U}\frac{\sigma(\rho)^2}{\rho+\d}\langle \nabla\cdot s\xi^F\rangle_1.
\end{align}
Let $\Theta_\d$ be the unique function satisfying $\Theta_\d(1-\d)=0$ and $\Theta_\d'(\eta)=\frac{\sigma(\eta)\sigma'(\eta)}{\eta+\d}$ for every $\eta\in(0,\infty)$.  It follows from the bounds on $\sigma$ that, for $c\in(0,\infty)$ independent of $\d\in(0,1)$, for every $\eta\in(0,\infty)$,
\[\abs{\Theta_\d(\eta)}\leq c\abs{\log(\eta+\d)}.\]
We therefore have using the bounds on $\sigma$, after integrating by parts in the penultimate term of \eqref{e_15}, for the outer unit normal $\nu$ to $\partial U$, for every $q\in[1,\infty)$, and for some $c\in(0,\infty)$ depending on $q$,
\begin{align}\label{e_16}
&  \Big(\max_{t\in[0,T]}\medint\int_{U}\Psi_\d(\rho)+\medint\int_0^T\medint\int_{U}\frac{1}{\rho+\d}\nabla\rho\cdot a\nabla\phi(\rho)\Big)^q
\\ \nonumber & \leq c\Big(\big(\medint\int_{U}\Psi_\d(\rho_0)\big)^q+\max_{t\in[0,T]}\abs{\medint\int_0^t\medint\int_{U}\frac{\sigma(\rho)}{\rho+\d}\nabla\rho\cdot s\dd\xi^F}^q+\langle\xi^F\rangle^q_1\big(\medint\int_0^T\medint\int_{\partial U} \abs{ s(\nabla\cdot s^t)\cdot \nu}\abs{\Theta_\d(\rho)}\big)^q\Big)
\\  \nonumber & \quad +c\Big(\langle\xi^F\rangle^q_1\big(\medint\int_0^T\medint\int_{U}\abs{\nabla\cdot (s(\nabla\cdot s^t))}\abs{\Theta_\d(\rho)}\big)^q+T^q\norm{\langle\nabla\cdot s\xi^F\rangle_1}_{L^\infty(U)}^q\Big),
\end{align}
where the $H^1$-regularity of the solution guarantees the existence of the trace of $\Theta_\d(\rho)$ on $\partial U$.  It remains to treat the stochastic integral in \eqref{e_16}, for which we have using the Burkh\"older--Davis--Gundy inequality, H\"older's inequality, and the structure of the diffusion coefficient $a=ss^t$ in Assumption~\ref{assume_d} that, for some $c\in(0,\infty)$,
\begin{align*}
& \E\big[\max_{t\in[0,T]}\abs{\medint\int_0^t\medint\int_{U}\frac{\sigma(\rho)}{\rho+\d}\nabla\rho\cdot s\dd\xi^F}^q\big] \leq c \E\big[\big(\medint\int_0^T\textstyle\sum_{k=1}^\infty\textstyle\sum_{i=1}^d\big(\medint\int_{U}\frac{\sigma(\rho)}{\rho+\d}(s^t\nabla\rho)_i f_k\big)^2\big)^\frac{q}{2}\big]
\\ & \leq c\langle \xi^F\rangle_1^\frac{q}{2}\E\big[\big(\medint\int_0^T\medint\int_{U}\frac{\sigma(\rho)^2}{(\rho+\d)^2}\abs{s\nabla\rho}^2\big)^\frac{q}{2}\big] \leq c\langle \xi^F\rangle_1^\frac{q}{2}\E\big[\big(\medint\int_0^T\medint\int_{U}\frac{1}{\rho+\d}\nabla\rho\cdot a \nabla\rho\big)^\frac{q}{2}\big].
\end{align*}
Returning to \eqref{e_16}, it is a consequence of Young's inequality, the uniform ellipticity of the coefficients, and the monotone convergence theorem that, after passing $\d\rightarrow 0$,
\begin{align*}
& \E \Big[\Big(\max_{t\in[0,T]}\medint\int_{U}\Psi(\rho)+\medint\int_0^T\medint\int_{U}\frac{1}{\rho}\nabla\rho\cdot a\nabla\phi(\rho)\Big)^q\Big]\leq \E\Big[\Big(\medint\int_{U}\Psi(\rho_0)\Big)^q\Big]+c T^q\norm{\langle \nabla\cdot s\xi^F\rangle_1}^q_{L^\infty(U)}
\\ & \quad +c\langle\xi^F\rangle^q_1\Big(1+\Big(\medint\int_0^T\medint\int_{\partial U} \abs{s(\nabla\cdot s^t)\cdot \nu}\abs{\Theta(\rho)}\Big)^q+\Big(\medint\int_0^T\medint\int_{U}\abs{\nabla\cdot (s(\nabla\cdot s^t))}\abs{\Theta(\rho)}\Big)^q\Big),
\end{align*}
for the smooth, bounded function $\Theta'(\eta)=\frac{\sigma(\eta)\sigma'(\eta)}{\eta}$ and for the entropy $\Psi(\eta)=\eta\log(\eta)-\eta$, which completes the proof using the nondegeneracy of $\phi$ in Assumption~\ref{assume_d}. \end{proof}

\begin{prop}\label{prop_measure} Under Assumptions~\ref{assume_d} and \ref{assume_n}, for a smooth and bounded $\sigma\in\C^2(\R)$  that satisfies \eqref{assume_sigma}, and for a nonnegative, $\F_0$-measurable $\rho_0\in L^2(\O;L^2(U))$, let $\rho$ be a solution of \eqref{e_1} in the sense of Definition~\ref{sol_smooth} with initial data $\rho_0$.  Then, for $\Theta_M$ satisfying $\Theta_M(0)=0$ and $\Theta_M'(\eta)=\sigma(\eta)\sigma'(\eta)\mathbf{1}_{\{M\leq \eta\leq M+1\}}$ for every $\eta\in(0,\infty)$,
\begin{align*}
& \E\big[\medint\int_0^T\medint\int_{U}\mathbf{1}_{\{M\leq \rho\leq M+1\}}\big(\nabla\rho\cdot a \nabla\rho\big)\leq \E\big[\medint\int_{U}(\rho_0-M)_+\big]
\\ & \quad +c\E\big[\norm{\langle \nabla\cdot s\xi^F\rangle_1}_{L^\infty(U)}\medint\int_0^T\medint\int_{U} (M+1)\mathbf{1}_{\{M\leq \rho\leq M+1\}}\big]
\\ & \quad +c\langle\xi^F\rangle_1\E\Big[\medint\int_0^T\medint\int_{\partial U}\abs{s(\nabla\cdot s^t)\cdot \nu}\abs{\Theta_M(\rho)}+\medint\int_0^T\medint\int_{U}\abs{\nabla\cdot(s(\nabla s^t))}\Theta_M(\rho)\Big].
\end{align*}
\end{prop}

\begin{proof}  Let $\Psi_M$ satisfy $\Psi_M'(0)=\Psi_M(0)=0$ and $\Psi''_M(\eta) = \mathbf{1}_{\{M\leq \eta\leq M+1\}}$.  It then follows by an approximation argument and It\^o's formula (see, for example, \cite{Kry2013}) using the methods of Proposition~\ref{prop_e} that, $\P$-a.s., 
\begin{align*}
& \E\big[\medint\int_{U}\Psi_M(\rho(x,T))+\medint\int_0^T\medint\int_{U}\mathbf{1}_{\{M\leq \rho\leq M+1\}}\big(\nabla\rho\cdot a\nabla\phi(\rho)\big)\big]=\E\big[\medint\int_{U}\Psi_M(\rho_0)\big]
\\ & +\E\big[\frac{1}{2}\medint\int_0^T\medint\int_{U}\langle \nabla\cdot s\xi^F\rangle_1 \sigma(\rho)^2 \mathbf{1}_{\{M\leq \rho\leq M+1\}}+\frac{\langle\xi^F\rangle_1}{2}\medint\int_0^T\medint\int_{U}\sigma(\rho)\sigma'(\rho)\mathbf{1}_{\{M\leq \rho\leq M+1\}} s(\nabla\cdot s^t)\cdot\nabla\rho\big].
\end{align*}
Then, using the bounds \eqref{assume_sigma} on $\sigma$ we have that, for some $c\in(0,\infty)$ independent of $\sigma$,
\[\abs{\Theta_M(\eta)}\leq c\mathbf{1}_{\{\eta\geq M\}}\;\;\textrm{and}\;\;\sigma(\eta)^2\leq c\eta.\]
Therefore, since by definition $\Psi_M$ satisfies $0\leq \Psi_M(\eta) \leq (\eta-M)_+$, we have using the $\C^2$-regularity of $s$, the nondegeneracy of $\phi$ in Assumption~\ref{assume_d}, the $H^1(U)$-regularity of $\rho$, and the integration by parts formula that, for the outer unit normal $\nu$ to $\partial U$, for some $c\in(0,\infty)$ independent of $M$,
\begin{align*}
& \E\big[\medint\int_0^T\medint\int_{U}\mathbf{1}_{\{M\leq \rho\leq M+1\}}\abs{\nabla\rho}^2\big]\leq \E\big[\medint\int_{U}(\rho_0-M)_+\big]
\\ & \quad +c\E\big[\norm{\langle \nabla\cdot s\xi^F\rangle_1}_{L^\infty(U)}\medint\int_0^T\medint\int_{U} (M+1)\mathbf{1}_{\{M\leq \rho\leq M+1\}}\big]
\\ & \quad +c\langle\xi^F\rangle_1\E\Big[\medint\int_0^T\medint\int_{\partial U}\abs{s(\nabla\cdot s^t)\cdot \nu}\abs{\Theta_M(\rho)}+\medint\int_0^T\medint\int_{U}\abs{\nabla\cdot(s(\nabla s^t))}\abs{\Theta_M(\rho)}\Big],
\end{align*}
which completes the proof.  \end{proof}

The previous Propositions~\ref{prop_e} and \ref{prop_measure} establish the positive spatial regularity of the solutions in $H^1$-based spaces.  We will now estimate the regularity of the solution in $H^{-1}(U)$ in Proposition~\ref{prop_h1}.  Here for $\sigma$ approximating the square root we have that $\Sigma(\rho)$ defined in Proposition~\ref{prop_h1} approximates $\log(\rho)$, and therefore using the preservation of mass it follows from the second line of the estimate in Proposition~\ref{prop_h1} that, $\P$-a.s.\ on the event $\rho_0\neq 0$, the logarithm of the solution will be integrable.  Here, for large values, we estimate $\log(\rho\vee 1)\leq (\rho\vee 1)$ using the $L^1$-norm.

\begin{prop}\label{prop_h1} Under Assumptions~\ref{assume_d} and \ref{assume_n}, for a smooth and bounded $\sigma\in\C^2(\R)$  that satisfies \eqref{assume_sigma}, and for a nonnegative, $\F_0$-measurable $\rho_0\in L^2(\O;L^2(U))$, let $\rho$ be a solution of \eqref{e_1} in the sense of Definition~\ref{sol_smooth} with initial data $\rho_0$.  Then, for $\Sigma$ satisfying $\Sigma(1)=0$ and $\Sigma'(\eta)=\sigma'(\eta)^2$, for $\langle \rho_0\rangle_{U}=\fint_{U}\rho_0$.
\begin{align*}
& \E\big[\max_{t\in[0,T]}\norm{(\rho-\langle\rho_0\rangle_{U})}^2_{H^{-1}(U)}\big]+\E\big[\medint\int_0^T\medint\int_{U}\rho^2\big]
\\ & +\langle\xi^F\rangle_1\E\big[\medint\int_0^T\medint\int_{U}\rho\Sigma(\rho)\big]+\langle\xi^F\rangle_1\E\big[\langle\rho_0\rangle_{U}\medint\int_0^T\medint\int_{U}\abs{\Sigma(\rho\wedge 1)}\big]
\\ & \leq c\Big(\norm{(\rho_0-\langle\rho_0\rangle_{U})}^2_{H^{-1}(U)}+T\E\big[\langle\rho_0\rangle_U^2\big]+T\norm{s(\nabla\cdot s^t)}^2_{L^\infty(U)}\langle\xi^F\rangle_1^2\Big)
 \\ & \quad +c\Big((1+T)\langle\xi^F\rangle_1\E\big[\langle\rho_0\rangle_U\big]+\langle\xi^F\rangle_1\E\big[\langle\rho_0\rangle_{U}\medint\int_0^T\medint\int_{U}\Sigma(\rho\vee 1)\big]\Big).
\end{align*}
\end{prop}

\begin{proof}  Let $H^1_a(U)$ denote the Sobolev space of mean zero $H^1(U)$-functions equipped with the positive definite inner product $\langle f,g\rangle_{H^1_a(U)} = \int_U\nabla f\cdot a\nabla g$ and let $z$ be the unique mean zero solution of the equation $-\nabla\cdot (a\nabla z)= \rho-\langle\rho_0\rangle_{U}$ on $U\times[0,T]$ with Neumann boundary conditions $(a\nabla z\cdot \nu)=0$ on $\partial U$ for the outer unit normal $\nu$ to $\partial U$, where here we are using the fact that the solution preserves mass to define $z$.  It then follows from It\^o's formula (see, for example, \cite{Kry2013}) and the stochastic product rule that, $\P$-a.s.\ for every $t\in[0,T]$, for the smooth function $\Sigma$ satisfying $\Sigma(1)=0$ and $\Sigma'(\eta)=\sigma'(\eta)^2$,
\begin{align}\label{h1_1}
\medint\int_{U} z(\rho-\langle\rho_0\rangle)\Big|_{s=0}^{s=t} & =-2\medint\int_0^t\medint\int_{U}\big(\rho-\langle\rho_0\rangle_{U}\big)\phi(\rho)-\langle\xi^F\rangle_1\medint\int_0^t\medint\int_{U}\big(\rho-\langle\rho_0\rangle_{U}\big)\Sigma(\rho)
\\ \nonumber & \quad -\langle\xi^F\rangle_1\medint\int_0^t\medint\int_{U}\sigma(\rho)\sigma'(\rho) s(\nabla\cdot s^t)\cdot\nabla z+2\medint\int_0^t\medint\int_{U}\sigma(\rho)\nabla z\cdot s\dd\xi^F
\\ \nonumber & \quad +\medint\int_0^t\textstyle\sum_{j,k=1}^\infty \abs{\medint\int_U \frac{s^t\nabla e_j}{\sqrt{\lambda_j}}f_k\sigma(\rho)}^2,
\end{align}
for $\{e_j\}_{j\in\N}$ an orthonormal $L^2(U)$-basis and orthogonal $H^1_a(U)$-basis of eigenfunctions $-\nabla\cdot a\nabla e_j = \lambda_je_j$ with Neumann boundary conditions $\big(a\nabla e_j\cdot \nu\big)=0$ on $\partial U$, as follows similarly to the computation in Appendix~\ref{sec_Ito} using the definition of $z$.  For the first term on the righthand side of \eqref{h1_1}, we have using the preservation of mass in Definition~\ref{sol_def} and the nondegeneracy of $\phi$ in Assumption~\ref{assume_d}  that, for some $c\in(0,\infty)$,
\begin{equation}\label{h1_2} -\medint\int_0^t\medint\int_{U}\big(\rho-\langle\rho_0\rangle_{U}\big)\phi(\rho) \leq  -c\medint\int_0^t\medint\int_{U}\rho^2+ct\langle\rho_0\rangle_U^2.\end{equation}
%s
For the second term on the righthand side of \eqref{h1_1}, we have that
\begin{equation}\label{h1_3} -\langle\xi^F\rangle_1\medint\int_0^t\medint\int_{U}\big(\rho-\langle\rho_0\rangle_{U}\big)\Sigma(\rho) = -\langle\xi^F\rangle_1\medint\int_0^t\medint\int_{U}\rho\Sigma(\rho)+\langle\xi^F\rangle_1\langle\rho_0\rangle_{U}\medint\int_0^t\medint\int_{U}\Sigma(\rho).\end{equation}
For the third term on the righthand side of \eqref{h1_1}, we have using the bounds on $\sigma$, the $\C^1$-boundedness of $s$ in Assumption~\ref{assume_d}, the energy estimate associated to the equation satisfied by $z$, and H\"older's inequality that, for some $c\in(0,\infty)$,
\begin{align}\label{h1_4}
& \abs{\langle\xi^F\rangle_1\medint\int_0^t\medint\int_{U}\sigma(\rho)\sigma'(\rho)s(\nabla\cdot s^t)\cdot\nabla z}
\\ \nonumber & \leq ct^\frac{1}{2}\norm{s(\nabla\cdot s^t)}_{L^\infty(U)}\langle\xi^F\rangle_1\big(\medint\int_0^t\medint\int_{U}\abs{\nabla z}^2\big)^\frac{1}{2}
\\ \nonumber & \leq ct^\frac{1}{2}\norm{s(\nabla\cdot s^t)}_{L^\infty(U)}\langle \xi^F\rangle_1\big(\medint\int_0^t\medint\int_{U}(\rho-\langle\rho_0\rangle_{U})^2\big)^\frac{1}{2}.
\end{align}
For the stochastic integral, we have using the boundedness of $s$ in Assumption~\ref{assume_d}, the structure of the noise in Assumption~\ref{assume_n}, the preservation of mass in Definition~\ref{sol_def}, the Burkh\"older--Davis--Gundy inequality, H\"older's inequality, and the energy estimate satisfied by $z$ that, for some $c\in(0,\infty)$,
\begin{align}\label{h1_5}
& \E\big[\max_{t\in[0,T]}\abs{\medint\int_0^t\medint\int_{U}\sigma(\rho)\nabla z \cdot s\dd\xi^F}\big] \leq c\E\big[\big(\medint\int_0^T\textstyle\sum_{j=1}^d\textstyle\sum_{k=1}^\infty\big(\medint\int_{U}\textstyle\sum_{i=1}^d\sigma(\rho)f_k\partial_i z s_{ij}\big)^2\big)^\frac{1}{2}\big]
\\ \nonumber & \leq c\langle\xi^F\rangle^\frac{1}{2}_1 \E\big[\langle\rho_0\rangle^\frac{1}{2}_{U}\medint\int_0^T\medint\int_{U}\abs{s}^2\abs{\nabla z}^2\big]^\frac{1}{2} \leq c\langle\xi^F\rangle^\frac{1}{2}_1\E\big[\langle\rho_0\rangle^\frac{1}{2}_{U}\medint\int_0^T\medint\int_{U}(\rho-\langle\rho_0\rangle_U)^2\big]^\frac{1}{2}.
\end{align}
For the final term on the righthand side of \eqref{h1_1}, since the $\sqrt{\lambda_j}^{-1}s^t\nabla e_j$ form a orthonormal $L^2(U)^d$-system, it follows from the preservation of mass in Definition~\ref{sol_def} and the bounds \eqref{assume_sigma} on $\sigma$ that, for some $c\in(0,\infty)$ depending on $d$, for every $t\in[0,T]$,
\begin{align}\label{h1_6}
& \medint\int_0^t\textstyle\sum_{j,k=1}^\infty \abs{\medint\int_U \frac{s^t\nabla e_j}{\sqrt{\lambda_j}}f_k\sigma(\rho)}^2  \leq d\medint\int_0^t\textstyle\sum_{k=1}^\infty \medint\int_U f_k^2\sigma(\rho)^2\leq c\langle\xi^F\rangle_1\medint\int_0^t\medint\int_{U}\rho\leq cT\langle\xi^F\rangle_1\langle\rho_0\rangle_{U}.
\end{align}
Returning to \eqref{h1_1}, it follows from the inequality
\[\int_0^t\int_U(\rho-\langle\rho_0\rangle_U)^2 = \int_0^t\int_U\rho^2-t\abs{U}\langle\rho_0\rangle_U^2\leq \int_0^t\int_U\rho^2,\]
and from Young's inequality, \eqref{h1_2}, \eqref{h1_3}, \eqref{h1_4}, \eqref{h1_5}, and \eqref{h1_6} that, for some $c\in(0,\infty)$,
\begin{align*}
& \E\big[\max_{t\in[0,T]}\medint\int_{U}z(\rho-\langle\rho_0\rangle_{U})\big]+\E\big[\medint\int_0^T\medint\int_{U}\rho^2\big]+\langle\xi^F\rangle_1\E\big[\medint\int_0^T\medint\int_{U}\rho\Sigma(\rho)\big]
 \\ & \leq c\Big(\E\medint\int_{U}z(\cdot,0)(\rho_0-\langle\rho_0\rangle_{U})+T\E\big[\langle\rho_0\rangle_U^2\big]+T\norm{s(\nabla\cdot s^t)}^2_{L^\infty(U)}\langle\xi^F\rangle_1^2\Big)
 \\ & \quad +c\Big((1+T)\langle\xi^F\rangle_1\E\big[\langle\rho_0\rangle_U\big]+\langle\xi^F\rangle_1\E\big[\langle\rho_0\rangle_{U}\medint\int_0^t\medint\int_{U}\Sigma(\rho)\big]\Big).
\end{align*}
Since it follows from the definition of $\Sigma$ that $\Sigma(\eta)\geq 0$ if $\eta\geq 1$ and $\Sigma(\eta)\leq 0$ if $\eta\leq 1$, and since it follows from the uniform ellipticity of $a$ in Assumption~\ref{assume_d} that
\[\norm{(\rho-\langle\rho_0\rangle_{U})}^2_{H^{-1}(U)}= \medint\int_{U}z(\rho-\langle\rho_0\rangle_{U}),\]
we have for some $c\in(0,\infty)$ that
\begin{align*}
& \E\big[\max_{t\in[0,T]}\norm{(\rho-\langle\rho_0\rangle_{U})}^2_{H^{-1}(U)}\big]+\E\big[\medint\int_0^T\medint\int_{U}\rho^2\big]
\\ & +\langle\xi^F\rangle_1\E\big[\medint\int_0^T\medint\int_{U}\rho\Sigma(\rho)\big]+\langle\xi^F\rangle_1\E\big[\langle\rho_0\rangle_{U}\medint\int_0^T\medint\int_{U}\abs{\Sigma(\rho\wedge 1)}\big]
\\ & \leq c\Big(\norm{(\rho_0-\langle\rho_0\rangle_{U})}^2_{H^{-1}(U)}+T\E\big[\langle\rho_0\rangle_U^2\big]+T\norm{s(\nabla\cdot s^t)}^2_{L^\infty(U)}\langle\xi^F\rangle_1^2\Big)
 \\ & \quad +c\Big((1+T)\langle\xi^F\rangle_1\E\big[\langle\rho_0\rangle_U\big]+\langle\xi^F\rangle_1\E\big[\langle\rho_0\rangle_{U}\medint\int_0^T\medint\int_{U}\Sigma(\rho\vee 1)\big]\Big),
\end{align*}
which completes the proof.
\end{proof}

In the final two propositions of this section, before proving that solutions exist in Theorem~\ref{thm_rks_ex}, we establish the regularity of the solution in time.  This is done in Proposition~\ref{prop_kolm}, where we show that the solution is $\P$-a.s.\ a H\"older continuous mapping into $H^{-1}(U)$, and in Proposition~\ref{prop_kcc}, where we establish the spatial regularity of the time-averaged quantities $\int_0^t\big(\rho+\frac{\langle\xi^F\rangle_1}{8}\log(\rho)\big)$, and the H\"older regularity in time of the ensemble as a mapping into the space $H^1(U)$.

\begin{prop}\label{prop_kolm}  Under Assumptions~\ref{assume_d} and \ref{assume_n}, for a smooth and bounded $\sigma\in\C^2(\R)$  that satisfies \eqref{assume_sigma}, and for a nonnegative, $\F_0$-measurable $\rho_0\in L^{2p}(\O;L^2(U))$ for some $p\in[2,\infty)$, let $\rho$ be a solution of \eqref{e_1} in the sense of Definition~\ref{sol_smooth} with initial data $\rho_0$.  Then, for some $c\in(0,\infty)$ depending on $p$, for every $r\leq t\in[0,T]$,
\begin{align*}
& \E\big[\norm{\rho_{r,t}}^{2p}_{H^{-1}(U)}\big] \leq c \Big((t-r)^p\norm{s(\nabla\cdot s^t)}^{2p}_{L^\infty(U)}\langle\xi^F\rangle_1^{2p}+(t-r)^p\langle\xi^F\rangle_1^p\E\big[\langle\rho_0\rangle_U^p\big]\Big)
\\ & \quad + c\Big(\big(1+\langle\xi^F\rangle_1^p\big)(t-r)^p\E\big[\norm{\rho}^{2p}_{L^\infty([0,T];L^2(U))}\big]+\langle \xi^F\rangle_1^\frac{p}{2}(t-r)^\frac{p}{2}\E\big[\langle \rho_0\rangle_U^p+\norm{\rho}^{2p}_{L^\infty([0,T];L^2(U))}\big]\Big).
\end{align*}
In particular, if $p>2$ then for every $\beta\in(0,\frac{1}{4}-\frac{1}{2p})$ we have for some $c\in(0,\infty)$ depending on $\beta$, $p$, $\langle\xi^F\rangle_1$, $\norm{s(\nabla\cdot s^t)}_{L^\infty(U)}$, $\E[\langle\rho_0\rangle_U^p]$, and $\E\big[\norm{\rho}^{2p}_{L^\infty([0,T];L^2(U))}\big]$ that
\[\E\big[\norm{\rho}_{\C^{0,\beta}([0,T];H^{-1}(U))}\big]<c.\]
\end{prop}

\begin{proof}  For every $r\leq t\in[0,T]$ let $\rho_{r,t} = \rho(\cdot,t)-\rho(\cdot,r)$ and let $z_{r,t}$ be the unique mean zero solution of the equation $-\nabla\cdot (a\nabla z_{r,t})=\rho_{r,t}$ on $U\times[r,T]$ with Neumann boundary conditions $\big(a\nabla z_{r,t}\cdot \nu\big)=0$ on $\partial U\times[r,T]$.  We repeat the above argument, for $\rho_r=\rho(\cdot,r)$ playing the role of $\langle\rho_0\rangle_U$ and for $\rho_t=\rho(\cdot,t)$, using the nondegeneracy of $\phi$ in Assumption~\ref{assume_d} and the inequality
\[\int_r^t\int_U(\rho_s-\rho_r)^2 \leq 2\int_r^t\int_U\rho_s^2+2\int_r^t\int_U\rho_r^2,\]
to prove that, $\P$-a.s.\ for every $r\leq t\in[0,T]$, for $\Sigma$ satisfying $\Sigma(1)=0$ and $\Sigma'(\eta)=\sigma'(\eta)^2$ and integrating over $s\in[r,t]$, for some $c\in(0,\infty)$,
\begin{align*}
& \medint\int_U \rho_{r,t}z_{r,t}+\medint\int_r^t\medint\int_{U}\rho_s^2+\langle\xi^F\rangle_1\medint\int_r^t\medint\int_{U}\rho_s\Sigma(\rho_s)+\langle\xi^F\rangle_1\medint\int_r^t\medint\int_{U}\rho_r\abs{\Sigma(\rho_s\wedge 1)}
\\ & \leq c\Big(\medint\int_r^t\medint\int_{U}\rho^2_r+(t-r)\norm{s(\nabla\cdot s^t)}^2_{L^\infty(U)}\langle\xi^F\rangle_1^2+\abs{\medint\int_r^t\medint\int_U\sigma(\rho)\nabla z_{r,s}\cdot s\dd\xi^F}\Big)
 \\ & \quad +c\Big(\langle\xi^F\rangle_1\medint\int_r^t\medint\int_U \rho_r\Sigma(\rho_s\vee 1) +(t-r)\langle\xi^F\rangle_1\langle\rho_0\rangle_{U} \Big).
\end{align*}
We therefore have for every $p\in[2,\infty)$, for some $c\in(0,\infty)$ depending on $p$ and $d$,
\begin{align*}
\Big(\medint\int_U  \rho_{r,t}z_{r,t}\Big)^p & \leq c \Big(\Big(\medint\int_r^t\medint\int_{U}\rho^2_r\Big)^p+(t-r)^p\norm{s(\nabla\cdot s^t)}^{2p}_{L^\infty(U)}\langle\xi^F\rangle_1^{2p}+(t-r)^p\langle\xi^F\rangle_1^p\langle\rho_0\rangle_U^p\Big)
\\ & \quad + c\Big(\langle\xi^F\rangle^p_1\Big(\medint\int_r^t\medint\int_U \rho_r\Sigma(\rho_s\vee 1)\Big)^p +\abs{\medint\int_r^t\medint\int_U\sigma(\rho)\nabla z_{r,s}\cdot s\dd\xi^F}^p \Big).
\end{align*}
Since it follows from the bounds \eqref{assume_sigma} on $\sigma$ that $\Sigma(\rho_t\vee 1)\leq 2 \rho_t$, it is a consequence of H\"older's inequality that
\begin{align*}
\Big(\medint\int_U  \rho_{r,t}z_{r,t}\Big)^p & \leq c \Big((t-r)^p\norm{s(\nabla\cdot s^t)}^{2p}_{L^\infty(U)}\langle\xi^F\rangle_1^{2p}+(t-r)^p\langle\xi^F\rangle_1^p\langle\rho_0\rangle_U^p\Big)
\\ & \quad + c\Big(\big(1+\langle\xi^F\rangle_1^p\big)(t-r)^p\E\big[\norm{\rho}^{2p}_{L^\infty([0,T];L^2(U))}\big]+\E\abs{\medint\int_r^t\medint\int_U\sigma(\rho)\nabla z_{r,s}\cdot s\dd\xi^F}^p\Big).
\end{align*}
For the stochastic integral, we have using the Burkh\"older--Davis--Gundy inequality, H\"older's inequality, the boundedness of $s$, and the structure of the noise in Assumption~\ref{assume_n} that, for some $c\in(0,\infty)$,
\begin{align*}
\E\Big[\abs{\medint\int_r^t\medint\int_U\sigma(\rho)\nabla z_{r,s}\cdot s\dd\xi^F}^p\Big]& \leq c\E\Big[\Big(\medint\int_r^t\textstyle\sum_{k=1}^\infty\textstyle\sum_{i=1}^d\big(\medint\int_U\sigma(\rho)(s^t\nabla z_{r,s})_if_k\big)^2\Big)^\frac{p}{2}\Big]
\\ & \leq c\langle \xi^F\rangle_1^\frac{p}{2}\E\Big[\Big(\medint\int_r^t\big(\medint\int_U\sigma(\rho)^2\big)\big(\medint\int_U \abs{\nabla z_{r,s}}^2\big)\Big)^\frac{p}{2}\Big].
\end{align*}
Since the energy estimate for $z_{r,s}$ proves that, for some $c\in(0,\infty)$,
\[\medint\int_U\abs{\nabla z_{r,s}}^2\leq c\medint\int_U \rho_{r,s}^2,\]
we have using the bounds \eqref{assume_sigma} on $\sigma$, the preservation of mass, and Young's inequality that, for some $c\in(0,\infty)$,
\begin{align*}
\E\Big[\abs{\medint\int_r^t\medint\int_U\sigma(\rho)\nabla z_{r,s}\cdot s\dd\xi^F}^p\Big] & \leq c\langle \xi^F\rangle_1^\frac{p}{2}(t-r)^\frac{p}{2}\E\big[\langle \rho_0\rangle_U^\frac{p}{2}\norm{\rho}^{p}_{L^\infty([0,T];L^2(U))}\big]
\\ & \leq c\langle \xi^F\rangle_1^\frac{p}{2}(t-r)^\frac{p}{2}\E\big[\langle \rho_0\rangle_U^p+\norm{\rho}^{2p}_{L^\infty([0,T];L^2(U))}\big].
\end{align*}
Since it follows from the uniform ellipticity of $a$ in Assumption~\ref{assume_d} that
\[\norm{\rho_{r,t}}^2_{H^{-1}(U)}= \medint\int_U  \rho_{r,t}z_{r,t},\]
we have that
\begin{align*}
& \E\big[\norm{\rho_{r,t}}^{2p}_{H^{-1}(U)}\big] \leq c \Big((t-r)^p\norm{s(\nabla\cdot s^t)}^{2p}_{L^\infty(U)}\langle\xi^F\rangle_1^{2p}+(t-r)^p\langle\xi^F\rangle_1^p\langle\rho_0\rangle_U^p\Big)
\\ & \quad + c\Big(\big(1+\langle\xi^F\rangle_1^p\big)(t-r)^p\E\big[\norm{\rho}^{2p}_{L^\infty([0,T];L^2(U))}\big]+\langle \xi^F\rangle_1^\frac{p}{2}(t-r)^\frac{p}{2}\E\big[\langle \rho_0\rangle_U^p+\norm{\rho}^{2p}_{L^\infty([0,T];L^2(U))}\big]\Big).
\end{align*}
The final claim then follows from the quantitative version of the Kolmogorov continuity theorem (see, for example, Friz and Victoir \cite[Corollaries A.10, A.11]{FrizVictoir}).
\end{proof}

\begin{prop}\label{prop_kcc}  Under Assumptions~\ref{assume_d} and \ref{assume_n}, for a smooth and bounded $\sigma\in\C^2(\R)$  that satisfies \eqref{assume_sigma}, and for a nonnegative, $\F_0$-measurable $\rho_0\in L^{2p}(\O;L^2(U))$ for some $p\in[2,\infty)$, let $\rho$ be a solution of \eqref{e_1} in the sense of Definition~\ref{sol_smooth} with initial data $\rho_0$.  Then, for some $c\in(0,\infty)$ depending on $p$, for every $r\leq t\in[0,T]$,
\begin{align*}
 & \E\Big[\norm{\medint\int_r^t\Big(\phi(\rho)+\frac{\langle\xi^F\rangle_1}{2}\Sigma(\rho)-\langle \phi(\rho)+\frac{\langle\xi^F\rangle_1}{2}\Sigma(\rho)\rangle_U\Big)}^{2p}_{H^1(U)}\Big] \leq c\E\big[\norm{\rho_{r,t}}^{2p}_{H^{-1}(U)}\big]
 \\ &  \quad + c\Big(\langle \xi^F\rangle^{2p}_1(t-r)^p\E\big[\langle\rho_0\rangle^p_{U}\big]+\langle\xi^F\rangle^{2p}_1\norm{s(\nabla\cdot s^t)}^{2p}_{L^\infty(U)}(t-r)^p\Big).
\end{align*}
In particular, if $p>2$ then for every $\beta\in(0,\frac{1}{4}-\frac{1}{2p})$ we have for some $c\in(0,\infty)$ depending on $\beta$, $p$, $\langle\xi^F\rangle_1$, $\norm{s(\nabla\cdot s^t)}_{L^\infty(U)}$, $\E[\langle\rho_0\rangle_U^p]$, and $\E\big[\norm{\rho}^{2p}_{L^\infty([0,T];L^2(U))}\big]$ that the map
\[t\in[0,T]\rightarrow L(\rho)(t)=\Big(\big(\medint\int_0^t\phi(\rho)+\frac{\langle\xi^F\rangle_1}{2}\Sigma(\rho)\big)-\langle \phi(\rho)+\frac{\langle\xi^F\rangle_1}{2}\Sigma(\rho)\rangle_U\Big)\in H^1(U),\]
satisfies
\[\E\big[\norm{L(\rho)}_{\C^{0,\beta}([0,T];H^{1}(U))}\big]<c.\]
\end{prop}

\begin{proof}  Let $\{e_k\}_{k\in\N}$ be an orthonormal basis of $L^2(U)$-eigenfunctions that is an orthogonal $H^1_a(U)$ basis satisfying $-\nabla\cdot a\nabla e_k=\lambda_ke_k$ with Neumann boundary conditions $\big(a\nabla e_k\cdot\nu)=0$ on $\partial U$, for ordered eigenvalues $0\leq \lambda_i\leq\lambda_j$ whenever $i\leq j$.  We then have for every $0\leq r\leq t\leq T$ that, for every $\lambda_j\neq 0$, for $\Sigma$ defined by $\Sigma(1)=0$ and $\Sigma'(\eta)=\sigma'(\eta)^2$,
\begin{align*}
\medint\int_{U}\rho e_j\Big|_{r}^{t} & = -\lambda_j\medint\int_r^t\medint\int_{U}\Big(\phi(\rho)+\frac{\langle\xi^F\rangle_1}{2}\Sigma(\rho)\Big)e_j+\medint\int_r^t\medint\int_U\sigma(\rho)s^t\nabla e_j\cdot\dd\xi^F
\\ & \quad -\frac{\langle\xi^F\rangle_1}{2}\medint\int_r^t\medint\int_{U}\sigma(\rho)\sigma'(\rho)s(\nabla\cdot s^t)\cdot\nabla e_j.
\end{align*}
Dividing by $\sqrt{\lambda_j}$ for every $\lambda_j\neq 0$, for some $c\in(0,\infty)$,
\begin{align*}
& \textstyle\sum_{\lambda_j\neq 0}\lambda_j\big(\medint\int_\R^s\medint\int_U\big(\phi(\rho)+\frac{\langle\xi^F\rangle_1}{2}\Sigma(\rho)\big)\big)e_j\big)^2
\\ & \leq c\textstyle\sum_{\lambda_j\neq 0}\frac{1}{\lambda_j}\big(\medint\int_{U}\rho e_j\Big|_{t=r}^{t=s}\big)^2+c\textstyle\sum_{\lambda_j\neq 0}\big(\medint\int_\R^s\medint\int_U\sigma(\rho)\frac{s^t\nabla e_j}{\sqrt{\lambda_j}}\cdot\dd\xi^F)\big)^2
\\ & \quad +c\textstyle\sum_{\lambda_j\neq 0}\langle\xi^F\rangle_1^2\big(\medint\int_\R^s\medint\int_{U}\sigma(\rho)\sigma'(\rho)(\nabla\cdot s^t)\cdot\frac{s^t\nabla e_j}{\sqrt{\lambda_j}}\big)^2,
\end{align*}
and therefore, using the definition of the $H^1$-, $L^2$- and $H^{-1}$-norms, for every $p\in[2,\infty)$ there exists $c\in(0,\infty)$ depending on $p$ such that
\begin{align}\label{kcc_1}
& \E\big[\norm{\nabla \medint\int_r^t\medint\int_U\big(\phi(\rho)+\frac{\langle\xi^F\rangle_1}{2}\Sigma(\rho)\big)}^{2p}_{L^2(U)}\big]\leq c\E\big[\norm{\rho_{r,t}}^{2p}_{H^{-1}(U)}\big]
\\ \nonumber & \quad +c\E\Big[\Big(\textstyle\sum_{\lambda_j\neq 0}\big(\medint\int_r^t\medint\int_U\sigma(\rho)\frac{s^t\nabla e_j}{\sqrt{\lambda_j}}\cdot\dd\xi^F)\big)^2\Big)^{p}\Big]
\\ \nonumber & \quad +c\E\Big[\Big(\textstyle\sum_{\lambda_j\neq 0}\langle\xi^F\rangle_1^2\big(\medint\int_r^t\medint\int_{U}\sigma(\rho)\sigma'(\rho)(\nabla\cdot s^t)\cdot\frac{s^t\nabla e_j}{\sqrt{\lambda_j}}\big)^2\Big)^p\Big].
\end{align}
For the stochastic integral we have using the structure of the noise in Assumption~\ref{assume_n}, the bounds \eqref{assume_sigma} on $\sigma$, the fact that the $\sqrt{\lambda_j}^{-1}s^t\nabla e_j$ form an orthonormal system, and the preservation of mass in Definition~\ref{sol_def} that, for $c\in(0,\infty)$ depending on $d$,
\begin{align*} 
\E\Big[\Big(\textstyle\sum_{\lambda_j\neq 0}\Big(\medint\int_r^t\medint\int_U\sigma(\rho)\frac{s^t\nabla e_j}{\sqrt{\lambda_j}}\cdot\dd\xi^F)\Big)^2\Big)^p\Big] & =\E\Big[\Big(\textstyle\sum_{\lambda_j\neq 0}\medint\int_r^t\textstyle\sum_{k=1}^\infty\textstyle\sum_{i=1}^d\Big(\medint\int_{U}\sigma(\rho)f_k\frac{(s^t\nabla e_j)_i}{\sqrt{\lambda_j}}\Big)^2\Big)^p\Big]
\\ & \leq c\E\Big[\Big(\medint\int_r^t\medint\int_U\textstyle\sum_{k=1}^\infty \medint\int_{U}\sigma(\rho)^2f_k^2\Big)^p\Big]
\\ &  =c\langle \xi^F\rangle^{p}_1(t-r)^p\E\big[\langle\rho_0\rangle^p_{U}\big].
\end{align*}
Similarly, for the final term on the righthand side of \eqref{kcc_1}, we have using the bounds \eqref{assume_sigma} on $\sigma$ that, for some $c\in(0,\infty)$ depending on $d$,
\begin{align*}
& \E\Big[\Big(\textstyle\sum_{\lambda_j\neq 0}\langle\xi^F\rangle_1^2\Big(\medint\int_r^t\medint\int_{U}\sigma(\rho)\sigma'(\rho)(\nabla\cdot s^t)\cdot\frac{s^t\nabla e_j}{\sqrt{\lambda_j}}\Big)^2\Big)^p\Big]
\\ & \leq c\langle\xi^F\rangle^{2p}_1\norm{s(\nabla\cdot s^t)}^{2p}_{L^\infty(U)}(t-r)^p.
\end{align*}
Therefore, returning to \eqref{kcc_1}, for every $p\in[2,\infty)$, for some $c\in(0,\infty)$ depending on $p$ and $d$,
\begin{align*}
 & \E\Big[\norm{\nabla \medint\int_r^t\big(\phi(\rho)+\frac{\langle\xi^F\rangle_1}{2}\Sigma(\rho)\big)}^{2p}_{L^2(U)}\Big] \leq c\E\big[\norm{\rho_{r,t}}^{2p}_{H^{-1}(U)}\big]
 \\ &  \quad + c\Big(\langle \xi^F\rangle^{2p}_1(t-r)^p\E\big[\langle\rho_0\rangle^p_{U}\big]+\langle\xi^F\rangle^{2p}_1\norm{s(\nabla\cdot s^t)}^{2p}_{L^\infty(U)}(t-r)^p\Big).
\end{align*}
Since Young's inequality proves that, for some $c\in(0,\infty)$ depending on $U$,
\[\norm{\medint\int_r^t \Big(\phi(\rho)+\frac{\langle\xi^F\rangle_1}{2}\Sigma(\rho)-\langle \phi(\rho)+\frac{\langle\xi^F\rangle_1}{2}\Sigma(\rho)\rangle_U\Big)}^2_{H^1(U)}\leq c\norm{\nabla \medint\int_r^t\big(\phi(\rho)+\frac{\langle\xi^F\rangle_1}{2}\Sigma(\rho)\big)}^{2}_{L^2(U)},\]
we have, for every $p\in[2,\infty)$, for some $c\in(0,\infty)$ depending on $p$, $d$, and $U$ that
\begin{align*}
 & \E\Big[\norm{\medint\int_r^t\Big(\phi(\rho)+\frac{\langle\xi^F\rangle_1}{2}\Sigma(\rho)-\langle\phi(\rho)+\frac{\langle\xi^F\rangle_1}{2}\Sigma(\rho)\rangle_U\Big)}^{2p}_{H^1(U)}\Big] \leq c\E\big[\norm{\rho_{r,t}}^{2p}_{H^{-1}(U)}\big]
 \\ &  \quad + c\Big(\langle \xi^F\rangle^{2p}_1(t-r)^p\E\big[\langle\rho_0\rangle^p_{U}\big]+\langle\xi^F\rangle^{2p}_1\norm{s(\nabla\cdot s^t)}^{2p}_{L^\infty(U)}(t-r)^p\Big).
\end{align*}
The final claim then follows from Proposition~\ref{prop_kolm} and the quantitative Kolmogorov continuity theorem (see, for example, \cite[Corollaries A.10, A.11]{FrizVictoir}).  \end{proof}

We will now prove the existence of stochastic kinetic solutions to \eqref{i_eq} in the sense of Definition~\ref{sol_def}.  The construction of a probabilistically strong solution relies also on the following lemma.

\begin{lem}\label{lemma_weak}  Let $(\O,\F,\P)$ be a probability space and let $\overline{X}$ be a complete separable metric space.  A sequence $\{X_n\colon\O\rightarrow \overline{X}\}$ of $\overline{X}$-valued random variables converges in probability, as $n\rightarrow\infty$, if and only if for any two sequences $\{(n_k,m_k)\}_{k=1}^\infty$ that satisfy $n_k,m_k\rightarrow\infty$ as $k\rightarrow\infty$ there exists subsequences $\{(n_{k'},m_{k'})\}_{k'=1}^\infty$ that satisfy $n_{k'},m_{k'}\rightarrow\infty$ as $k'\rightarrow\infty$ such that the joint law of $(X_{n_{k'}},X_{m_{k'}})_{k'\in\N}$ converges weakly, as $k'\rightarrow\infty$, to a measure $\mu$ on $\overline{X}\times\overline{X}$ that satisfies
\[\mu(\{(x,y)\in\overline{X}\times \overline{X}\colon x=y\})=1.\]
\end{lem}

\begin{proof} The proof appears in Gy\"ongy and Krylov \cite[Lemma~1.1]{GyoKry1996}. \end{proof}

\begin{thm}\label{thm_rks_ex} Under Assumptions~\ref{assume_d} and \ref{assume_n}, for every nonnegative, $\F_0$-measurable $\rho_0\in L^1(\O;L^1(U))$ there exists a stochastic kinetic solution of \eqref{i_eq} in the sense of Definition~\ref{sol_def}. \end{thm}

\begin{proof}  We first consider nonnegative, $\F_0$-measurable initial data $\rho_0\in L^{2p}(\O;H^1(U))$ for some $p\in(2,\infty)$, and we let $\{\sigma_n\}_{n\in\N}$ be a sequence of smooth functions $\sigma_n\in\C^\infty(\R)$ satisfying that $\sigma_n(\eta)=0$ if $\eta\leq 0$, that $\sigma'_n\in\C^\infty_c(\R)$ is nonnegative, that for every $n\in\N$ the $\sigma_n$ satisfy the bounds \eqref{assume_sigma}, and that locally uniformly on $(0,\infty)$, as $n\rightarrow\infty$,
\begin{equation}\label{exi_1}\sigma_n(\eta)\rightarrow \sqrt{\eta}\;\;\textrm{and}\;\;\sigma_n'(\eta)\rightarrow\frac{1}{2\sqrt{\eta}}.\end{equation}
For every $n\in\N$ let $\rho_n$ be the solution of \eqref{e_1} in the sense of Definition~\ref{sol_smooth} constructed in Proposition~\ref{prop_entropy} with nonlinearity $\sigma_n$ and initial data $\rho_0$.  We will first show that the $\rho_n$ converge to a weak solution of \eqref{i_eq}, and then use It\^o's formula to derive the kinetic form.  In the last step of the proof, we will extend the solution theory to nonnegative, $\F_0$-measurable initial data $\rho_0\in L^1(\O;L^1(U))$.

\textbf{Tightness of the martingale terms}:  We first establish the tightness of the martingale terms appearing in equation \eqref{sol_eq} of Definition~\ref{sol_def}.  Let $\psi\in\C^\infty_c(\overline{U}\times(0,\infty))$ and let
\begin{align*}
M^{n,\psi}_t = \medint\int_0^t\medint\int_U\sigma_n(\rho_n)(\nabla_x\psi)(x,\rho_n)\cdot s\dd\xi^F+\medint\int_0^t\medint\int_U\sigma_n(\rho_n)(\partial_\eta\psi)(x,\rho_n)\nabla\rho_n\cdot s\dd\xi^F.
\end{align*}
Let $\Psi_n$ be the unique function satisfying $\Psi_n(x,0)=0$ and $\Psi_n'(x,\eta) = \sigma_n(\eta)(\partial_\eta\psi)(x,\eta)$.  After integrating by parts, using the smoothness of $\Psi_n$ and the $H^1$-regularity of $\rho_n$ to apply the trace theorem, for the outer unit normal $\nu$ of $U$,
\begin{align*}
M^{n,\psi}_t & = \medint\int_0^t\medint\int_U\sigma_n(\rho_n)(\nabla_x\psi)(x,\rho_n)\cdot s\dd\xi^F-\medint\int_0^t\medint\int_U\Psi_n(x,\rho_n)\big(\nabla \cdot s\dd\xi^F\big)
\\ & \quad +\medint\int_0^t\medint\int_{\partial U} \Psi_n(x,\rho_n)\big(\nu\cdot s\dd\xi^F\big).
\end{align*}
Since it is a consequence of the bounds on $\sigma_n$ and the compact support of $\psi$ that the $\Psi_n$ are bounded independently of $n$, it follows from the boundedness of $s$, the Burkh\"older--Davis--Gundy inequality, and the conservation of mass \eqref{e_010} that for every $p\in[2,\infty)$ there exists $c\in(0,\infty)$ depending on $p$ and $\psi$ but independent of $n$ such that, for every $s\leq t\in[0,T]$,
\begin{equation}\label{exis_00}\E\big[\abs{M^{n,\psi}_t-M^{n,\psi}_s}^p\big]\leq c(s-t)^\frac{p}{2}\Big(\langle\xi^F\rangle_1^\frac{p}{2}\E\big[\langle\rho_0\rangle_U^\frac{p}{2}\big]+\langle\xi^F\rangle_1^\frac{p}{2}+\langle \nabla\cdot s\xi^F\rangle_1^\frac{p}{2}\Big).\end{equation}
If $p\in(2,\infty)$ it follows from the quantified Kolmogorov's continuity criterion (see \cite[Corollaries A.10, A.11]{FrizVictoir}) that the laws of the $\{M^{n,\psi}_t\}_{n\in\N}$ are tight on $\C^{0,\gamma}([0,T])$ for every $\gamma\in(\nicefrac{1}{2}-\nicefrac{1}{p})$.

\textbf{Applying the Skorokhod representation theorem.}  Let $\{\psi_j\}_{j\in\N}$ be a countable dense subset of $\C^\infty_c(\overline{U}\times(0,\infty))$ in the $H^s_{\textrm{loc}}(U\times(0,\infty))$-topology for some $s>\frac{d+3}{2}$, where $s$ is chosen large enough to guarantee the compact Sobolev embbedding of $H^s_{\textrm{loc}}(U\times(0,\infty))$ into $\C^1_{\textrm{loc}}(U\times(0,\infty))$.  For every $j\in\N$ let $M^{n,j}_t = M^{n,\psi_j}_t$ and observe from the equation that, for the kinetic function $\chi_n$ of $\rho_n$, $\P$-a.s.\ for every $t\in[0,T]$,
\begin{align}\label{6_29}
& M^{n,j}_t  = -\medint\int_\R\medint\int_U \chi_{n,r}\psi_j(x,\eta)\Big|_{r=0}^{r=t}-\medint\int_0^t\medint\int_U (\nabla_x\psi_j)(x,\rho_n) \cdot a\nabla\phi(\rho_n)-\medint\int_0^t\medint\int_\R\medint\int_U (\partial_\eta\psi_j) q_n
\\ \nonumber &  -\frac{\langle\xi^F\rangle_1}{2}\medint\int_0^t\medint\int_U\Big(\sigma_n'(\rho_n)^2(\nabla_x\psi_j)(x,\rho_n)\cdot a \nabla \rho_n+\sigma_n(\rho_n)\sigma_n'(\rho_n)s(\nabla\cdot s^t)\cdot (\nabla_x\psi_j)(x,\rho_n)\Big)
\\ \nonumber & +\frac{1}{2}\medint\int_0^t\medint\int_U \Big(\langle\xi^F\rangle_1\sigma(\rho_n)\sigma'(\rho_n)(\partial_\eta\psi_j)(x,\rho_n)\nabla\rho_n\cdot s(\nabla\cdot s^t)+\langle \nabla\cdot s\xi^F\rangle_1(\partial_\eta \psi_j)(x,\rho_n)\sigma_n^2(\rho_n)\Big).
\end{align}
Furthermore, using Proposition~\ref{prop_kcc}, for every $n\in\N$, for $\Sigma_n$ the unique function satisfying $\Sigma_n(1)=0$ and $\Sigma'_n(\eta) = \sigma'_n(\eta)^2$, let $L(\rho_n)\in\C^{0,\gamma}([0,T];H^1(U))$ be defined by
\[L_n(\rho_n)_t = \medint\int_0^t\big(\phi(\rho_n)+\frac{\langle\xi^F\rangle_1}{2}\Sigma_n(\rho_n)-\langle\phi(\rho_n)+\frac{\langle\xi^F\rangle_1}{2}\Sigma_n(\rho_n)\rangle_U \big).\]
For every $n\in\N$ let $X_n$ be the random variable
\[X_n = (\rho_n, q_n, (M^{n,j}_t)_{j\in\N}, L_n(\rho_n)),\]taking values in the space, for any $\gamma\in(\nicefrac{1}{2}-\frac{1}{2p})$,
\[\overline{X} =  \underline{X}\times \mathcal{M}_+(U\times\R\times[0,T])\times\C^{0,\gamma}([0,T])^{\N}\times\C^{0,\gamma}([0,T];H^1(U)),\]
for the space
\[\underline{X}=L^2(U\times[0,T])\cap L^2([0,T];H^1(U))\cap \C^{0,\gamma}([0,T];H^{-1}(U)),\]
where $\underline{X}$ and $\overline{X}$ are equipped with the intersection and product topologies induced by the strong topology on $L^2(U\times[0,T])$, the weak topology on the space of nonnegative Radon measures $\mathcal{M}_+(U\times\R\times[0,T])$, the weak topology on $L^2([0,T];H^1(U))$, and the topology of component-wise convergence in the strong norm on $\C^{0,\gamma}([0,T])^\N$ induced by the metric
\[D_\gamma((f_k)_{k\in\N},(g_k)_{k\in\N})=\textstyle\sum_{k=1}^\infty2^{-k}\frac{\norm{f_k-g_k}_{\C^{0,\gamma}([0,T])}}{1+\norm{f_k-g_k}_{\C^{0,\gamma}([0,T])}}.\]
We will apply Lemma~\ref{lemma_weak} to construct a probabilistically strong solution.  For this, let $(n_k)_{k\in\N}$ and $(m_k)_{k\in\N}$ be two subsequences satisfying  $n_k,m_k\rightarrow\infty$ as $k\rightarrow\infty$, and consider the laws of
\[(X_{n_k},X_{m_k},B,\rho_0)\;\;\textrm{on}\;\;\overline{Y}=\overline{X}\times\overline{X}\times\C^{0,\gamma}([0,T];\R^d)^\N\times H^1(U),\]
for the infinite-dimensional Brownian motion $B=(B^l)_{l\in\N}$ of Assumption~\ref{assume_n}.  Propositions~\ref{prop_e}, \ref{prop_kolm}, and \ref{prop_kcc} with the tightness of the martingale terms shown above and the Aubin--Lions--Simon lemma \cite{Aubin,pLions,Simon} (see specifically \cite[Theorem~5]{Simon}) prove that the laws of $(X_{n_k}, X_{m_k},B,\rho_0)_{k\in\N}$ on $\overline{Y}$ are tight.  It then follows from Prokhorov's theorem (see, for example, Billingsley \cite[Chapter~1, Theorem~5.1]{Bil1999}) that, after passing to a subsequence still denoted $k\rightarrow\infty$, there exists probability measure $\mu$ on $\overline{Y}$ such that, as $k\rightarrow\infty$,
\[(X_{n_k},X_{m_k},B,\rho_0)\rightarrow\mu\;\;\textrm{in law on $\overline{Y}$.}\]
Since the space $\overline{X}$ is separable, the space $\overline{Y}$ is separable and it follows heuristically from the Skorokhod representation theorem (see, for example, \cite[Chapter~1, Theorem~6.7]{Bil1999}), and specifically from Jakubwoski \cite[Theorem~2]{Jak1997} to overcome the fact that the above weak topologies are not metrizable, that there exists an auxiliary probability space $(\tilde{\O},\tilde{\F},\tilde{\P})$ and $\overline{Y}$-valued random variables $(\tilde{Y}_k,\tilde{Z}_k,\tilde{\beta}_k,\tilde{\rho}_{0,k})_{k\in\N}$ and a $\overline{Y}$-valued random variable $(\tilde{Y},\tilde{Z},\tilde{\beta},\tilde{\rho}_0)$ on $\tilde{\O}$ such that, for every $k\in\N$,
\begin{equation}\label{6_32}(\tilde{Y}_k,\tilde{Z}_k,\tilde{\beta}_k,\tilde{\rho}_{0,k})=(X_{n_k},X_{m_k},B,\rho_0)\;\;\textrm{in law on}\;\;\overline{Y},\end{equation}
such that
\begin{equation}\label{6_34}(\tilde{Y},\tilde{Z},\tilde{\beta},\tilde{\rho}_0)=\mu\;\;\textrm{in law on $\overline{Y}$,}\end{equation}
and such that, $\tilde{\P}$-a.s.\ as $k\rightarrow\infty$,
\begin{equation}\label{6_35}\tilde{Y}_k\rightarrow \tilde{Y},\;\tilde{Z}_k\rightarrow \tilde{Z},\;\tilde{\beta}_k\rightarrow\tilde{\beta},\;\textrm{and}\;\tilde{\rho}_{0,k}\rightarrow\tilde{\rho}_0\;\textrm{in $\overline{X}$, in $\C^{0,\gamma}([0,T];\R^d)^\N$, and in $H^1(U)$.}\end{equation}
We will now show that $\tilde{Y}=\tilde{Z}$ almost surely on $\overline{X}$ with respect to $\tilde{\P}$. We have using the definition of $\overline{X}$ that, for every $k\in\N$, there exists $\tilde{\rho}_k\in L^{2p}(\tilde{\O};L^{2}([0,T];H^1(U)))$, $\tilde{q}_k\in L^1(\tilde{\O};\mathcal{M}(U\times(0,\infty)\times[0,T]))$, and $(\tilde{M}^{k,j})_{j\in\N}\in L^1(\tilde{\O};\C^{0,\gamma}([0,T])^\N)$, and $\tilde{L}_k\in L^1(\tilde{\O};\C^{0,\gamma}([0,T])^\N)$ such that
\[\tilde{Y}_k = (\tilde{\rho}_k,\tilde{q}_k,(\tilde{M}^{k,j})_{j\in\N},\tilde{L}_k),\]
and that
\[\tilde{Y} = (\tilde{\rho},\tilde{q},(\tilde{M}^j)_{j\in\N}, \tilde{L}),\]
and similarly for $\tilde{Z}_k$ and $\tilde{Z}$.  It remains to characterize the functions, measures, and paths defining these quantities.

\textbf{Preservation of mass}:  It is a consequence of the $\tilde{\P}$-a.s.\ convergence of the $\tilde{\rho}_k\rightarrow\tilde{\rho}$ in $\underline{X}$ and the $\tilde{\rho}_{0,k}\rightarrow\tilde{\rho}_0$ in $H^1(U)$ that $\tilde{\rho}$ is $\tilde{\P}$-a.s.\ nonnegative and that, for every $t\in[0,T]$,
\[\medint\fint_U\tilde{\rho}(\cdot,t) = \abs{U}^{-1}\norm{\tilde{\rho}(\cdot,t)}_{L^1(U)}=\langle\tilde{\rho}_0\rangle_U,\]
which completes the proof of mass preservation.

\textbf{The kinetic measures $\tilde{q}$}.  For every nonnegative $\psi\in\C^\infty_c(\overline{U}\times(0,\infty))$, it follows from \eqref{6_32}, $\tilde{\rho}_k\in L^2([0,T];H^1(U))$, and the regularity property \eqref{d_s3} that
\begin{align*}
& \tilde{\P}\big[\medint\int_0^T\medint\int_\R\medint\int_U \psi \tilde{q}_k-\medint\int_0^T\medint\int_U \psi(x,\tilde{\rho}_k)\phi'(\tilde{\rho}_k)\big(\nabla\tilde{\rho}_k\cdot a\nabla\tilde{\rho}_k\big)<0\big]
\\ & = \P\big[\medint\int_0^T\medint\int_\R\medint\int_U \psi q_{n_k}-\medint\int_0^T\medint\int_U \psi(x,\rho_{n_k})\phi'(\rho_{n_k})\big(\nabla\rho_{n_k}\cdot a\nabla\rho_{n_k}\big)<0\big] = 0,
\end{align*}
from which it follows that the measures $\tilde{q}_k$ are $\tilde{\P}$-a.s.\ kinetic measures for the $\tilde{\rho}_k$ in the sense of Definition~\ref{d_measure} and Definition~\ref{sol_def}.  It is then a consequence of \eqref{6_35}, the symmetry of $a$, and the equality
\begin{align*}
\medint\int_0^T\medint\int_U\psi(x,\tilde{\rho}_k)\phi'(\tilde{\rho}_k)\big(\nabla\tilde{\rho}_k\cdot a\nabla\tilde{\rho}_k\big) & = -\medint\int_0^T\medint\int_U \psi(x,\tilde{\rho}_k)\phi'(\tilde{\rho}_k)\big(\nabla\tilde{\rho}\cdot a\nabla\tilde{\rho}\big)
\\ & \quad +2\medint\int_0^T\medint\int_U\psi(x,\tilde{\rho}_k)\phi'(\tilde{\rho}_k)\big(\nabla\tilde{\rho}_k\cdot a\nabla\tilde{\rho}\big)
\\ & \quad + \medint\int_0^T\medint\int_U\psi(x,\tilde{\rho}_k)\phi'(\tilde{\rho}_k)\big(\nabla(\tilde{\rho}_k-\tilde{\rho})\cdot a\nabla(\tilde{\rho}_k-\tilde{\rho})\big),
\end{align*}
that we may pass to the limit $k\rightarrow\infty$ using the dominated convergence theorem for the first two terms, and the nonnegativity of the final term, which can be discarded.  Therefore,
\begin{equation}\label{exis_1} \tilde{\P}\big[\medint\int_0^T\medint\int_0^\infty\medint\int_U\medint\int_\R \psi \tilde{q}-\medint\int_0^T\medint\int_U \psi(x,\tilde{\rho})\phi'(\tilde{\rho})\big(\nabla\tilde{\rho}\cdot a\nabla\tilde{\rho}\big)<0\big]=0,\end{equation}
which proves that $\tilde{q}$ is $\tilde{\P}$-a.s.\ a kinetic measure for $\tilde{\rho}$ in the sense of Definition~\ref{d_measure} and Definition~\ref{sol_def}.

\textbf{The paths $\tilde{L}$}.  It follows from the smoothness of the $\Sigma_n$, owing to the regularity of the $\sigma_n$, and the nondegeneracy of $\phi$ in Assumption~\ref{assume_d} that the map
\[v\in L^2([0,T];H^1(U))\rightarrow L_n(v) = \medint\int_0^t\big(\phi(v)+\frac{\langle\xi^F\rangle_1}{2}\Sigma_n(v)\big)\in \C([0,T];L^2(U)),\]
is continuous, and therefore arguing similarly to the above we have $\tilde{\P}$-a.s.\ for every $k\in\N$ that
\begin{equation}\label{exis_000}\tilde{L}_k(t) = \medint\int_0^t\big(\phi(\tilde{\rho}_k)+\frac{\langle\xi^F\rangle_1}{2}\Sigma_{n_k}(\tilde{\rho}_k)-\langle \phi(\tilde{\rho}_k)+\frac{\langle\xi^F\rangle_1}{2}\Sigma_{n_k}(\tilde{\rho}_k)\rangle_U\big)\in\C^{0,\gamma}([0,T];H^1(U)).\end{equation}
Furthermore, due to the smoothness of $\Sigma_n$, the map
\[v\in L^2([0,T];H^1(U))\rightarrow\Sigma_n(v)\in L^2([0,T];H^1(U)),\]
is continuous, from which it follows from Proposition~\ref{prop_h1}, \eqref{6_32}, and $\tilde{\rho}_{0,k}\in L^p(\tilde{\O};L^2(U))$ that, for some $c\in(0,\infty)$,
\begin{align}\label{exist_0011}
& \tilde{\E}\big[\langle\xi^F\rangle_1\langle\rho_0\rangle_U\medint\int_0^T\medint\int_U\abs{\Sigma_{n_k}(\tilde{\rho}_k)}\big] \leq c\Big(\norm{\rho_0}^2_{L^2(U)}+T\norm{s(\nabla\cdot s^t)}^2_{L^\infty(U)}\langle\xi^F\rangle_1^2\Big)
 \\ \nonumber & \quad +c\Big((1+T)\langle\xi^F\rangle_1\E\big[\langle\rho_0\rangle_{U}\big]+\langle\xi^F\rangle_1\E\big[\langle\rho_0\rangle^2_{U}\big]\Big),
 \end{align}
is uniformly bounded in $k\in\N$.  It is then a consequence of Fatou's lemma, the definition of $\Sigma_n$, and the convergence \eqref{6_35} that, $\tilde{\P}$-a.s.\ for every $t\in[0,T]$, after passing to a random subsequence $k\rightarrow\infty$,
\[\langle\tilde{\rho}_0\rangle_U\medint\int_0^t\medint\int_U\log(\tilde{\rho}\vee 1)\leq \liminf_{k\rightarrow\infty}\langle\tilde{\rho}_{k,0}\rangle_U\medint\int_0^t\medint\int_U\Sigma_{n_k}(\tilde{\rho}_k\vee 1),\]
and
\[\langle\tilde{\rho}_0\rangle_U\medint\int_0^t\medint\int_U\log(\tilde{\rho}\wedge 1)\geq\limsup_{k\rightarrow\infty}\langle\tilde{\rho}_{k,0}\rangle_U\medint\int_0^t\medint\int_U\Sigma_{n_k}(\tilde{\rho}_k\wedge 1),\]
from which it follows from \eqref{exist_0011} that $\langle\tilde{\rho}_0\rangle_U\log(\tilde{\rho})$ is $\tilde{\P}$-a.s.\ $L^1$-integrable on $U\times[0,T]$.  The dominated convergence theorem and the definition of $\Sigma_{n_k}$ then prove that, $\tilde{\P}$-a.s.\ as $k\rightarrow\infty$,
\[\langle\tilde{\rho}_0\rangle_U\Sigma_{n_k}(\tilde{\rho})\rightarrow\langle\tilde{\rho}_0\rangle_U\log(\tilde{\rho})\;\;\textrm{strongly in $L^1(U\times[0,T])$},\]
and since, $\tilde{\P}$-a.s.\ on the event $\{\tilde{\rho}_0\neq 0\}$ the functions $\int_0^t\Sigma_{n_k}(\tilde{\rho}_k)$ are uniformly bounded in $L^2(U)$ in $k\in\N$ and $t\in[0,T]$, and since for every $t\in[0,T]$ these functions converge a.e.\ on $U$ to $\int_0^t\log(\tilde{\rho})$, it follows that, for every $t\in[0,T]$,
\[\lim_{k\rightarrow\infty} \medint\int_0^t\medint\int_U\Sigma_{n_k}(\tilde{\rho}_k)=\medint\int_0^t\medint\int_U\log(\tilde{\rho}).\]
Since the convergence of the remaining terms in the definition of $\tilde{L}_k$ follow from the strong $L^2$-converge of the $\tilde{\rho}_k$ to $\tilde{\rho}$, we have $\tilde{\P}$-a.s.\ on the event $\{\tilde{\rho}_0\neq 0\}$ after passing to the limit $k\rightarrow\infty$ in \eqref{exis_000} that
\begin{equation}\label{exis_0000} \tilde{L}(t) =\medint\int_0^t\big(\phi(\tilde{\rho})+\frac{\langle\xi^F\rangle_1}{2}\log(\tilde{\rho})-\langle \phi(\tilde{\rho})+\frac{\langle\xi^F\rangle_1}{2}\log(\tilde{\rho})\rangle_U\big)\in\C^{0,\gamma}([0,T];H^1(U)),\end{equation}
where this quantity is understood to be zero on the event $\{\tilde{\rho}_0=0\}$.

\textbf{Convergence of the $\sigma_{n_k}\sigma'_{n_k}$ and nonvanishing of $\tilde{\rho}$ in $U$}.  It follows from \eqref{6_35} that there $\tilde{\P}$-a.s.\ exists a random subsequence $k\rightarrow\infty$ such that
\[\tilde{\rho}_k\rightarrow\tilde{\rho}\;\;\textrm{almost everywhere on $U\times[0,T]$.}\]
Since the integrability of $\langle\tilde{\rho}_0\rangle_U\log(\tilde{\rho})$ implies $\tilde{\P}$-a.s.\ on the event $\{\tilde{\rho}_0\neq 0\}$ that the set $\{(x,t)\in U\times[0,T]\colon\tilde{\rho}(x,t)=0\}$ has zero measure, it is a consequence of the convergence \eqref{exi_1} that $\tilde{\P}$-a.s., along the same random sequence $k\rightarrow\infty$, almost everywhere on $U\times[0,T]$ and therefore in $L^p(U\times[0,T])$ for every $p\in[1,\infty)$,
\begin{equation}\label{exis_102}\sigma_{n_k}(\tilde{\rho}_k)\sigma_{n_k}'(\tilde{\rho}_k)\rightarrow\frac{1}{2},\end{equation}
which completes the analysis of this term.

\textbf{The nonvanishing of $\tilde{\rho}$ on $\partial U$}.  We first observe that the map
\[v\in L^2([0,T];H^1(U))\rightarrow\medint\int_0^Tv(\cdot,s)\ds\in H^1(U),\]
is continuous with respect to the strong topologies of $H^1(U)$ and $L^2([0,T];H^1(U))$.  We then observe that the trace operator $\textrm{Tr}$ commutes with this map, in the sense that
\begin{equation}\label{exist_160}\textrm{Tr}\big(\medint\int_0^Tv(\cdot,s)\ds\big) = \medint\int_0^T\textrm{Tr}(v)\ds\;\;\textrm{in}\;\;L^2(\partial U),\end{equation}
which follows from the smoothness of $U$, the continuity of the map $\textrm{Tr}\colon H^1(U)\rightarrow L^2(\partial U)$ (see, for example, \cite[Section~5.5, Theorem~1]{Eva2010}), and the fact that if $v$ is smooth then $\textrm{Tr}(v) = v|_{\partial U}$.  Similarly, for all smooth functions $F\colon\R\rightarrow\R$ we have that $\textrm{Tr}(F(v)) = F(\textrm{Tr}(v))$.  Finally, we observe that the trace operator preserves order, in the sense that if $v_1\leq v_2\in H^1(U)$ then $\textrm{Tr}(v_1)\leq \textrm{Tr}(v_2)$, which follows immediately from the identity $\textrm{Tr}(v) = v|_{\partial U}$ for every smooth function $v\in \C^\infty(U\times[0,T])$.  Let $M\in[2,\infty)$ be arbitrary and let $\tilde{\rho}_M = \big((\tilde{\rho}\wedge 1)\vee \nicefrac{1}{M}\big)$.  Since the $L^2([0,T];H^1(U))$-regularity of $\tilde{\rho}$ implies the $L^2([0,T];H^1(U))$-regularity of $\log(\tilde{\rho}_M)$ for every $M\in[2,\infty)$, we have using the order preservation of the trace operator that
\[\norm{\textrm{Tr}\big(\medint\int_0^T\log(\tilde{\rho}_M)\big)}_{L^2(U)}\leq \norm{\textrm{Tr}\big(\medint\int_0^T\log(\tilde{\rho})\big)}_{L^2(U)},\]
where it follows from \eqref{exis_0000} and the preservation of mass that the righthand side is $\tilde{\P}$-a.s.\ finite on the event $\{\rho_0\neq 0\}$.  It is then a consequence of Chebyshev's inequality on $\partial U$, \eqref{exist_160}, and the commutativity of the trace operator with smooth functions that
\begin{equation}\label{exist_175}\abs{\{(x,t)\in \partial U\times[0,T]\colon \textrm{Tr}(\tilde{\rho}_M)=\nicefrac{1}{M}\}}\leq \abs{\log(\nicefrac{1}{M})}^{-2}\norm{\textrm{Tr}\big(\medint\int_0^T\log(\tilde{\rho})\big)}^2_{L^2(U)}.\end{equation}
This completes the proof of \eqref{s_log} on $\partial U$ in the sense of Remark~\ref{remark_sol}.

\textbf{The Brownian paths $\tilde{\beta}_k$ and $\tilde{\beta}$}.  Let $F\colon\overline{Y}\rightarrow\R$ be a continuous function and let $\tilde{\beta}_k=(\tilde{\beta}^{k,l})_{l\in\N}$ and let $\tilde{\beta}=(\tilde{\beta}^l)_{l\in\N}$.  It is a consequence of \eqref{6_32} that, for every $s\leq t\in[0,T]$, $k,l\in\N$, for the restriction $\tilde{Y}_k|_{[0,s]}$ to the interval $[0,s]$,
\begin{align}\label{6_37}
& \tilde{\E}\big[F\big(\tilde{Y}_k|_{[0,s]},\tilde{Z}_k|_{[0,s]}, \tilde{\beta}_k|_{[0,s]}\big)\big(\tilde{\beta}^{k,l}_t-\tilde{\beta}^{k,l}_s\big)  \big]
\\ \nonumber & =\E\big[F\big(X_{n_k}|_{[0,s]},X_{m_k}|_{[0,s]}, B|_{[0,s]}\big)\big(B^l_t-B^l_s\big)\big]=0.
\end{align}
It follows similarly that, for every $k,l,m\in\N$, $i,j\in\{1,\ldots,d\}$, and $s\leq t\in[0,T]$, for $\beta^{k,l}=(\beta^{k,l}_i)_{i\in\{1,\ldots,d\}}$,
\[\E\big[F\big(\tilde{Y}_k|_{[0,s]},\tilde{Z}_k|_{[0,s]}, \tilde{\beta}_k|_{[0,s]}\big)\big(\tilde{\beta}^{k,l}_{i,t}\tilde{\beta}^{k,m}_{j,t}-\tilde{\beta}^{k,l}_{i,s}\tilde{\beta}^{k,m}_{j,s}-\delta_{ij}\delta_{lm}(t-s)\big)  \big]=0,\]
for the Kronecker delta $\delta_{ij}$.  It is then a consequence of Levy's characterization of Brownian motion (see, for example, \cite[Chapter~4, Theorem~3.6]{RevYor1999}) that the $\tilde{\beta}_k$ are independent $d$-dimensional Brownian motions.  Since then, for every $t\in[0,T]$, the random variables $(\tilde{\beta}_k)_{k\in\N}$ are uniformly integrable and since, after passing to the limit $k\rightarrow\infty$ using \eqref{6_35} and \eqref{6_37}, for every $s\leq t\in[0,T]$ and $l\in\N$,
\begin{equation}\label{6_38}\tilde{\E}\big[F\big(\tilde{Y}|_{[0,s]},\tilde{Z}|_{[0,s]}, \tilde{\beta}|_{[0,s]}\big)\big(\tilde{\beta}^l_t-\tilde{\beta}^l_s\big)\big]=0,\end{equation}
we have that, for every $l,m\in\N$, $i,j\in\{1,\ldots,d\}$, and $s\leq t\in[0,T]$,
\begin{equation}\label{6_39}\tilde{\E}\big[F\big(\tilde{Y}|_{[0,s]},\tilde{Z}|_{[0,s]}, \tilde{\beta}|_{[0,s]}\big)\big(\tilde{\beta}^l_{i,t}\tilde{\beta}^m_{j,t}-\tilde{\beta}^l_{i,s}\tilde{\beta}^m_{j,s}-\delta_{ij}\delta_{lm}(t-s)\big)\big]=0.\end{equation}
Since $\tilde{\P}$-almost surely $\tilde{\beta}^l\in\C([0,T];\R^d)$ for every $l\in\N$, it follows from \eqref{6_38}, \eqref{6_39}, and Levy's characterization of Brownian motion that $(\tilde{\beta}^l)_{l\in\N}$ are standard, independent $d$-dimensional Brownian motions with respect to the filtration $\mathcal{G}_t=\sigma(\tilde{Y}|_{[0,t]},\tilde{Z}|_{[0,t]}, \tilde{\beta}|_{[0,t]})$.  It follows from the continuity and uniform integrability of the Brownian motion in time that $\tilde{\beta}$ is a Brownian motion with respect to the augmented filtration $(\overline{\mathcal{G}}_t)_{t\in[0,T]}$ of $(\mathcal{G}_t)_{t\in[0,T]}$, which is the smallest complete, right-continuous filtration containing $(\mathcal{G}_t)_{t\in[0,T]}$.

\textbf{The paths $(\tilde{M}^j)_{j\in\N}$ are $\overline{\mathcal{G}}_t$-martingales.}  Let $F\colon\overline{Y}\rightarrow\R$ be continuous and let $k,l\in\N$.  It follows from \eqref{6_32} that, for every $s\leq t\in[0,T]$,
\begin{align*}
& \tilde{\E}\big[F\big(\tilde{Y}_k|_{[0,s]},\tilde{Z}_k|_{[0,s]}, \tilde{\beta}_k|_{[0,s]}\big)\big(\tilde{M}^{k,l}_t-\tilde{M}^{k,l}_s\big)\big]
\\ &=\E\big[F\big(X_{n_k}|_{[0,s]},X_{m_k}|_{[0,s]}, B|_{[0,s]}\big)\big(M^{n_k,\psi_l}_t-M^{n_k,\psi_l}_s\big)\big]=0.
\end{align*}
After passing to the limit $k\rightarrow\infty$, where \eqref{exis_00} and \eqref{6_32} prove that the $\tilde{M}^{k,l}_t$ are uniformly bounded in $L^p(\O\times[0,T])$ for every $p\in[1,\infty)$ and hence uniformly integrable,
\begin{equation}\label{6_40}\tilde{\E}\big[F\big(\tilde{Y}|_{[0,s]},\tilde{Z}|_{[0,s]}, \tilde{\beta}|_{[0,s]}\big)\big(\tilde{M}^j_t-\tilde{M}^j_s\big)\big]=0.\end{equation}
This proves that $(\tilde{M}^j)_{t\in[0,\infty)}$ satisfies the martingale property with respect to $(\mathcal{G}_t)_{t\in[0,T]}$.  It then follows from the continuity and uniform integrability of the $(\tilde{M}^l)_{l\in\N}$ that the martingale terms are continuous martingales with respect to the augmentation $(\overline{\mathcal{G}}_t)_{t\in[0,T]}$.

\textbf{The $\tilde{M}$ are stochastic integrals with respect to $\tilde{\beta}$.}  Let $F\colon\overline{Y}\rightarrow\R$ be a continuous function.  It follows from \eqref{6_32} that, for every $s\leq t\in[0,T]$, $k,l,m\in\N$, and $i\in\{1,\ldots,d\}$,
\begin{align*}
& \tilde{\E}\big[F\big(\tilde{Y}_k|_{[0,s]},\tilde{Z}_k|_{[0,s]}, \tilde{\beta}_k|_{[0,s]}\big)\big(\tilde{M}^{k,l}_t\tilde{\beta}^{k,m}_{i,t}-\tilde{M}^{k,l}_s\tilde{\beta}^{k,m}_{i,s}
\\ & \quad \quad -\big(\medint\int_s^t\medint\int_U\sigma_{n_k}(\tilde{\rho}_k)(s^t\nabla_x\psi_l)_i(x,\tilde{\rho}_k)f_m+\medint\int_s^t\medint\int_U\sigma_{n_k}(\tilde{\rho}_k)(\partial_\eta\psi_l)(x,\tilde{\rho}_k)(s^t\nabla\tilde{\rho}_k)_if_m\big)\big)\big]
\\ & =\E\big[F\big(X_{n_k}|_{[0,s]},X_{m_k}|_{[0,s]}, B|_{[0,s]}\big)\big(M^{n_k,\psi_l}_tB^m_{i,t}-M^{n_k,\psi_l}_sB^m_{i,s}
\\ &  \quad\quad -\big(\medint\int_s^t\medint\int_U\sigma_{n_k}(\tilde{\rho}_k)(s^t\nabla_x\psi_l)_i(x,\tilde{\rho}_k)f_m+\medint\int_s^t\medint\int_U\sigma_{n_k}(\tilde{\rho}_k)(\partial_\eta\psi_l)(x,\tilde{\rho}_k)(s^t\nabla\tilde{\rho}_k)_if_m\big)\big)\big]=0.
\end{align*}
Since it follows as above that the $\tilde{M}^{k,l}_t\tilde{\beta}^{k,m}_{i,t}$  are uniformly integrability in $k$ for every time $t\in[0,T]$, after passing to the limit $k\rightarrow\infty$ using the convergence \eqref{6_35} and the definition of the $\sigma_n$,
\begin{align}\label{6_41}
& \tilde{\E}\big[F\big(\tilde{Y}|_{[0,s]},\tilde{Z}|_{[0,s]}, \tilde{\beta}|_{[0,s]}\big)\big(\tilde{M}^{l}_t\tilde{\beta}^{m}_{i,t}-\tilde{M}^{l}_s\tilde{\beta}^{m}_{i,s}
\\ \nonumber & \quad \quad -\big(\medint\int_s^t\medint\int_U\sqrt{\tilde{\rho}}(s^t\nabla_x\psi_l)_i(x,\tilde{\rho})f_m+\medint\int_s^t\medint\int_U\sqrt{\tilde{\rho}_k}(\partial_\eta\psi_l)(x,\tilde{\rho})(s^t\nabla\tilde{\rho})_if_m\big)\big)\big].
\end{align}
We therefore conclude from \eqref{6_41} that, for every $l,m\in\N$ and $i\in\{1,\ldots,d\}$,
\begin{equation}\label{6_42}\tilde{M}^l_t\tilde{\beta}^m_{i,t}-\big(\medint\int_0^t\medint\int_U\sqrt{\tilde{\rho}}(s^t\nabla_x\psi_l)_i(x,\tilde{\rho})f_m+\medint\int_0^t\medint\int_U\sqrt{\tilde{\rho}_k}(\partial_\eta\psi_l)(x,\tilde{\rho})(s^t\nabla\tilde{\rho})_if_m\big)\;\;\textrm{is a $\mathcal{G}_t$-martingale.}\end{equation}
It then follows from the continuity of the process in time and the uniform integrability that the process \eqref{6_42} is also a continuous $\overline{\mathcal{G}}_t$-martingale.

A virtually identical proof shows that, for every $l\in\N$, the process
\begin{equation}\label{6_43} (\tilde{M}^l_t)^2-\medint\int_0^t\textstyle\sum_{m=1}^\infty\textstyle\sum_{i=1}^d\big(\medint\int_0^t\medint\int_U\sqrt{\tilde{\rho}}(s^t\nabla_x\psi_l)_i(x,\tilde{\rho})f_m+\medint\int_0^t\medint\int_U\sqrt{\tilde{\rho}_k}(\partial_\eta\psi_l)(x,\tilde{\rho})(s^t\nabla\tilde{\rho})_if_m\big)^2,\end{equation}
is a continuous $\overline{\mathcal{G}}_t$-martingale.  It then follows \eqref{6_42}, \eqref{6_43}, and an explicit calculation using the quadratic variation and covariation with the Brownian motion that, for every $l\in\N$ and $t\in[0,T]$,
\begin{equation}\label{6_44}\tilde{\E}\big[\big(\tilde{M}^l_t-\medint\int_0^t\medint\int_U\sqrt{\tilde{\rho}}(\nabla_x\psi_l)(x,\tilde{\rho})\cdot s\dd\tilde{\xi}^F+\medint\int_0^t\medint\int_U\sqrt{\tilde{\rho}}(\partial_\eta\psi_l)(x,\tilde{\rho})\nabla\tilde{\rho}\cdot s\dd\tilde{\xi}^F\big)^2\big]=0,\end{equation}
for the noise $\tilde{\xi}^F$ defined as in Assumption~\ref{assume_n} by $\tilde{\beta}$.  It follows from \eqref{6_40} and \eqref{6_44} that the quadratic variation of the martingale appearing under the square in \eqref{6_44} vanishes, and therefore that, $\tilde{\P}$-a.s.\ for every $l\in\N$ and $t\in[0,T]$,
\begin{equation}\label{6_45} \tilde{M}^j_t =\medint\int_0^t\medint\int_U\sqrt{\tilde{\rho}}(\nabla_x\psi_l)(x,\tilde{\rho})\cdot s\dd\tilde{\xi}^F+\medint\int_0^t\medint\int_U\sqrt{\tilde{\rho}}(\partial_\eta\psi_l)(x,\tilde{\rho})\nabla\tilde{\rho}\cdot s\dd\tilde{\xi}^F,\end{equation}
which completes the characterization of the martingale terms.

\textbf{Recovering the kinetic form of the equation}.  We first observe that, $\P$-a.s.\ for every $k\in\N$, $t\in[0,T]$, and $\psi\in\C^\infty(\overline{U})$,
\begin{align}\label{nex_1}
 \medint\int_U\rho_{n_k}\psi\big|_{r=0}^{r=t} & = -\medint\int_U\nabla\psi\cdot a\nabla L_{n_k}(\rho_{n_k})_t+\medint\int_0^t\medint\int_U\sigma_{n_k}(\rho_{n_k})\nabla\psi\cdot\dd\xi^F
 \\ \nonumber & \quad -\frac{\langle\xi^F\rangle_1}{2}\medint\int_0^t\medint\int_U\sigma_{n_k}(\rho_{n_k})\sigma'_{n_k}(\rho_{n_k}) s(\nabla\cdot s^t)\cdot\nabla\psi.
 \end{align}
And, therefore, using the separability of $\C^\infty(\overline{U})$ in the $H^s(U)$-norm, \eqref{6_32}, and \eqref{6_34} that, $\tilde{\P}$-a.s.\ for every $k\in\N$ the $\tilde{\rho}_k$ satisfy \eqref{nex_1} for every $t\in[0,T]$ and $\psi\in\C^\infty(\overline{U})$.  After passing to the limit $k\rightarrow\infty$ using the convergence \eqref{6_35}, \eqref{exis_0000}, and \eqref{exis_102}, $\tilde{\P}$-a.s.\ for every $t\in[0,T]$ and $\psi\in\C^\infty(\overline{U})$,
\begin{equation}\label{nex_11}
 \medint\int_U\tilde{\rho}\psi\big|_{r=0}^{r=t} = -\medint\int_U\nabla\psi\cdot a\nabla L(\tilde{\rho})_t+\medint\int_0^t\medint\int_U\sqrt{\tilde{\rho}}\nabla\psi\cdot\dd\xi^F-\frac{\langle\xi^F\rangle_1}{4}\medint\int_0^t\medint\int_Us(\nabla\cdot s^t)\cdot\nabla\psi.
 \end{equation}
It is a consequence of \eqref{nex_11} and the $H^1(U)$-continuity of $L(\tilde{\rho})$ that, $\tilde{\P}$-a.s.\ for every $\psi\in\C^\infty(\overline{U})$,
\[t\in[0,T]\rightarrow\medint\int_U\tilde{\rho}(x,t)\psi(x)\dx\;\;\textrm{is continuous,}\]
and therefore that $\tilde{\P}$-a.s.\ the map
\[ t\in[0,T]\rightarrow\tilde{\rho}(\cdot,t)\in L^2(U)\;\;\textrm{is continuous w.r.t.\ the weak topology of $L^2(U)$.}\]
It follows similarly from the representation \eqref{6_29}, the convergence \eqref{6_35}, the kinetic measure \eqref{exis_1}, the identity \eqref{6_45}, and the compact support of $\psi_l$ on $U\times(0,\infty)$ that $\tilde{\P}$-a.s.\ there exists a random set $\mathcal{C}\subseteq [0,T]$ of full measure such that, for the kinetic function $\tilde{\chi}$ of $\tilde{\rho}$, $\tilde{\P}$-a.s.\ for every $t\in\mathcal{C}$ and $l\in\N$,
\begin{align}\label{6_46}
& \left.\medint\int_\R\medint\int_{U}\tilde{\chi}(x,\xi,r)\psi_l(x,\xi)\right|_{r=0}^{r=t} =-\medint\int_0^t\medint\int_{U}(\nabla\psi_l)(x,\tilde{\rho})\cdot a\nabla\phi(\tilde{\rho}) -\medint\int_0^t\medint\int_\R\medint\int_{U}\partial_\eta\psi_l\dd\tilde{q}
\\ \nonumber &  -\frac{\langle\xi^F\rangle_1}{4}\medint\int_0^t\medint\int_U\Big(\frac{1}{2}(\nabla_x\psi_l)(x,\rho_n)\cdot a \nabla \log(\tilde{\rho})+s(\nabla\cdot s^t)\cdot (\nabla_x\psi_l)(x,\tilde{\rho})\Big)
\\ \nonumber & +\frac{1}{2}\medint\int_0^t\medint\int_U \Big(\frac{\langle\xi^F\rangle_1}{2}(\partial_\eta\psi_l)(x,\tilde{\rho})\nabla\tilde{\rho}\cdot s(\nabla\cdot s^t)+\langle \nabla\cdot s\xi^F\rangle_1(\partial_\eta \psi_l)(x,\tilde{\rho})\tilde{\rho})\Big).
\\ \nonumber & + \medint\int_0^t\medint\int_U\sqrt{\tilde{\rho}}(\nabla_x\psi_l)(x,\tilde{\rho})\cdot s\dd\tilde{\xi}^F+\medint\int_0^t\medint\int_U\sqrt{\tilde{\rho}}(\partial_\eta\psi_l)(x,\tilde{\rho})\nabla\tilde{\rho}\cdot s\dd\tilde{\xi}^F.
\end{align}
Since every term on the righthand side of \eqref{6_46} is continuous in time except possibly the term containing $\tilde{q}$, it follows from the outer regularity of finite Radon measures on locally compact spaces that, for every $l\in\N$, the map
\begin{equation}\label{exis_130}t\in[0,T]\rightarrow \medint\int_\R\medint\int_U\tilde{\chi}(x,\xi,t)\psi_l(x,\xi)\;\;\textrm{is right continuous,}\end{equation}
where the possible points of discontinuity $t_*\in[0,T]$ are quantified by the integral
\[\medint\int_{\{t^*\}}\medint\int_\R\medint\int_U\partial_\eta\psi_l\dd\tilde{q}.\]
These integral vanish away from the random set
\[\mathcal{A} = \{t\in[0,T]\colon q(U\times\R\times\{t\})>0\}.\]
 Furthermore, while the maps \eqref{exis_130} are always right continuous at $t=0$, their value at zero is
 \[\medint\int_R\medint\int_U\overline{\chi}(\rho_0)\psi_l-\medint\int_{\{0\}}\medint\int_\R\medint\int_U\partial_\eta\psi_l\dd\tilde{q}.\]
 We therefore need to prove that $0\notin \mathcal{A}$ in order to guarantee that the maps \eqref{exis_130} are equal to $\int_\R\int_U\overline{\chi}(\rho_0)\psi_l$ at $t=0$.  We will first do this, and then establish that the set $\mathcal{A}$ is $\tilde{\P}$-a.s.\ empty.

 \textbf{Initial time continuity, $0\notin \mathcal{A}$}.  The density of the $\psi_l$ in $H^s(U)$, the fact that $\tilde{\P}$-a.s.\ we have $\tilde{\rho}\in L^{2p}([0,T];H^1(U))$, and a repetition of the methods used to prove Proposition~\ref{prop_zero} prove $\tilde{\P}$-a.s.\ that $\psi(\xi)=\xi$ is an admissible test function for \eqref{6_46} and therefore that, for every $t\in[0,T]\setminus \mathcal{C}$,
 \begin{align}\label{exis_131}
&  \frac{1}{2}\big(\norm{\tilde{\rho}(\cdot,t)}^2_{L^2(U)}-\norm{\tilde{\rho}_0}^2_{L^2(U)})+\tilde{q}(U\times\R\times[0,t])
 \\ \nonumber & =\medint\int_0^t\medint\int_U\sqrt{\tilde{\rho}}\nabla\tilde{\rho}\cdot s\dd\tilde{\xi}^F +\frac{\langle\xi^F\rangle_1}{4}\medint\int_0^t\medint\int_U s(\nabla\cdot s^t)\cdot \nabla\tilde{\rho}+\frac{1}{2}\medint\int_0^t\medint\int_U\langle \nabla\cdot s\xi^F\rangle_1\tilde{\rho}.
 \end{align}
Then, using the the $\tilde{\P}$-a.s.\ $H^1(U)$-regularity of $\tilde{\rho}_0$, we have from \eqref{nex_11} using the density of $\C^\infty(\overline{U})$ in $H^1(U)$ that, for every $t\in[0,T]$,
\begin{equation}\label{exis_132} \medint\int_U\big(\tilde{\rho}(\cdot,t)\tilde{\rho}_0-\tilde{\rho}_0^2\big) = -\medint\int_U\nabla\tilde{\rho}_0\cdot a \nabla L(\tilde{\rho})_t+\medint\int_0^t\medint\int_U\sqrt{\tilde{\rho}}\nabla\tilde{\rho}_0\cdot\dd\xi^F-\frac{\langle\xi^F\rangle_1}{4}\medint\int_0^t\medint\int_Us(\nabla\cdot s^t)\cdot\nabla\tilde{\rho}_0.\end{equation}
After subtracting \eqref{exis_132} from \eqref{exis_131} and applying the polarization identity, for every $t\in[0,T]\setminus\mathcal{C}$,
\begin{align}\label{exis_133}
&  \frac{1}{2}\norm{\tilde{\rho}(\cdot,t)-\tilde{\rho}_0}^2_{L^2(U)}+\tilde{q}(U\times\R\times[0,t])
 \\ \nonumber & =\medint\int_0^t\medint\int_U\sqrt{\tilde{\rho}}\nabla\tilde{\rho}\cdot s\dd\tilde{\xi}^F +\frac{\langle\xi^F\rangle_1}{4}\medint\int_0^t\medint\int_U s(\nabla\cdot s^t)\cdot \nabla\tilde{\rho}+\frac{1}{2}\medint\int_0^t\medint\int_U\langle \nabla\cdot s\xi^F\rangle_1\tilde{\rho}
 \\ \nonumber & \quad +2\medint\int_U\nabla\tilde{\rho}_0\cdot a \nabla L(\tilde{\rho})_t-2\medint\int_0^t\medint\int_U\sqrt{\tilde{\rho}}\nabla\tilde{\rho}_0\cdot\dd\xi^F+\frac{\langle\xi^F\rangle_1}{2}\medint\int_0^t\medint\int_Us(\nabla\cdot s^t)\cdot\nabla\tilde{\rho}_0.
 \end{align}
Since every term on the righthand side of \eqref{exis_133} is $\tilde{\P}$-a.s.\ continuous in time, after passing to the limit $t\rightarrow 0$ in $[0,T]\setminus \mathcal{C}$, we have using the outer regularity of finite Radon measures on locally compact spaces that, $\tilde{\P}$-a.s.,
\begin{equation}\label{exis_134} q(U\times\R\times\{0\})=0,\end{equation}
which completes the proof that $0\notin\mathcal{A}$.  It remains to show $\tilde{\P}$-a.s.\ that $\mathcal{A}=\emptyset$.

\textbf{The measure $\tilde{q}$ has no atoms in time, $\mathcal{A}=\emptyset$}.   For every $l\in\N$ we define the path $\langle \tilde{\chi},\psi_l\rangle_t = \int_\R\int_{U}\tilde{\chi}(x,\xi,t)\psi_l(x,\xi)$ and let $\tilde{q}_{\psi_l}$ be the measure defined by $\dd \tilde{q}_{\psi_l} = (\partial_\xi\psi_l)(x,\xi)\dd \tilde{q}$.  It then follows from the inner and outer regularity of finite Radon measures on locally compact spaces that the functions $\tilde{Q}^{\pm}_{\psi_l}\colon[0,T]\rightarrow\R$ defined by
\begin{equation}\label{exis_148} \tilde{Q}^+_{\psi_l}(t)=\tilde{q}_{\psi_l}(U\times(0,\infty)\times[0,t])\;\;\textrm{and}\;\;\tilde{Q}^-_{\psi_l}(t)=\tilde{q}_{\psi_l}(U\times(0,\infty)\times[0,t))\end{equation}
are $\tilde{\P}$-a.s.\ right- and left-continuous, satisfy $\tilde{Q}^+_{\psi_l}(0)=\tilde{Q}^-_{\psi_l}(0)=0$, and satisfy $\tilde{Q}^+_{\psi_l}(t)=\tilde{Q}^-_{\psi_l}(t)$ for every $t\in[0,T]\setminus\mathcal{A}$.  It then follows from \eqref{6_46} and \eqref{exis_148} that, $\tilde{\P}$-a.s.\ for every $l\in\N$, the functions $t\in\mathcal{C}\rightarrow\langle\tilde{\chi},\psi_l\rangle_t$ admit right- and left-continuous representatives $\langle\tilde{\chi},\psi_l\rangle^{\pm}_t$ on $[0,T]$. It follows from the nonnegativity of the solutions, the density of the $\psi_l$ in the strong $L^2(U\times[0,\infty))$-topology, and the definition of the kinetic function $\tilde{\chi}$ of $\tilde{\rho}$ that the functions $\tilde{\chi}^\pm$ defined by
\[\medint\int_{U}\tilde{\chi}^{\pm}(x,\xi,t)\psi_l(x,\xi) = \langle \tilde{\chi},\psi_l\rangle^{\pm}_t\]
are almost surely weakly right- and left-continuous in $L^2(U\times[0,\infty))$ in time, and satisfy $\tilde{\chi}^{\pm}(x,\xi,t)=\tilde{\chi}(x,\xi,t)$ for every $t\in[0,T]\setminus\mathcal{A}$.  For $\tilde{Q}^{\pm}_\psi$ defined analogously to \eqref{exis_148} for every $\psi\in\C^\infty_c(\overline{U}\times(0,\infty))$, the density of the $\{\psi_l\}_{l\in\N}$ in the $H^s$-norm, the Sobolev embedding theorem, and \eqref{6_46} prove $\tilde{\P}$-a.s.\ that there exists a subset of full probability such that, for every $t\in[0,T]$ and $\psi\in\C^\infty_c(\overline{U}\times(0,\infty))$,
\begin{align}\label{6_49}
&\langle \tilde{\chi}^{\pm},\psi\rangle_t = \medint\int_\R \medint\int_{U}\overline{\chi}(\rho_0)\psi-\medint\int_0^t\medint\int_{U}(\nabla_x\psi)(x,\tilde{\rho})\cdot a\nabla\phi(\tilde{\rho})-\tilde{Q}^{\pm}_{\psi}(t)
\\ \nonumber &  -\frac{\langle\xi^F\rangle_1}{4}\medint\int_0^t\medint\int_U\Big(\frac{1}{2}(\nabla_x\psi)(x,\tilde{\rho})\cdot a \nabla \log(\tilde{\rho})+s(\nabla\cdot s^t)\cdot (\nabla_x\psi)(x,\tilde{\rho})\Big)
\\ \nonumber & +\frac{1}{2}\medint\int_0^t\medint\int_U \Big(\frac{\langle\xi^F\rangle_1}{2}(\partial_\eta\psi)(x,\tilde{\rho})\nabla\tilde{\rho}\cdot s(\nabla\cdot s^t)+\langle \nabla\cdot s\xi^F\rangle_1(\partial_\eta \psi)(x,\tilde{\rho})\tilde{\rho})\Big).
\\ \nonumber & + \medint\int_0^t\medint\int_U\sqrt{\tilde{\rho}}(\nabla_x\psi)(x,\tilde{\rho})\cdot s\dd\tilde{\xi}^F+\medint\int_0^t\medint\int_U\sqrt{\tilde{\rho}}(\partial_\eta\psi)(x,\tilde{\rho})\nabla\tilde{\rho}\cdot s\dd\tilde{\xi}^F.
\end{align}
It remains to characterize the limiting kinetic functions $\tilde{\chi}^{\pm}$.  We will first show that $\partial_\xi\tilde{\chi}^\pm\leq 0$ in the sense of distributions.  For this, we observe that for every $t\in[0,T]$, with only a right limit if $t=0$ or a left limit if $t=T$, there $\tilde{\P}$-a.s.\ exist random sequences $\{(t^+_k,t^-_k)\}_{k\in\N}$ of positive and negative numbers that satisfy $t^+_k,t^-_k\rightarrow 0$ as $k\rightarrow\infty$ with $\tilde{\chi}^\pm(x,\xi,t+t^\pm_k)=\tilde{\chi}(x,\xi,t+t^\pm_k)$ for every $k\in\N$.  Therefore, for these subsequences, for any nonnegative functions $\alpha\in\C^\infty_c((0,\infty))$ and $\psi\in\C^\infty(\overline{U})$,
\begin{align*}
\medint\int_\R\medint\int_{U}\tilde{\chi}^{\pm}(x,\xi,t)\psi(x)\alpha'(\xi) & \geq \liminf_{k\rightarrow\infty}\medint\int_\R\medint\int_{U}\tilde{\chi}(x,\xi,t+t^{\pm}_k)\psi(x)\alpha'(\xi)
\\ & \geq \liminf_{k\rightarrow\infty}\medint\int_\R\medint\int_{U}\psi(x)\alpha(\tilde{\rho}(x,t+t^{\pm}_k)) \geq 0.
\end{align*}
Since linear combinations of the $\alpha(\xi)\psi(x)$ are dense in $\C^\infty_c(\overline{U}\times(0,\infty))$, $\tilde{\P}$-a.s.\ as distributions,
\begin{equation}\label{exist_172}\partial_\xi\tilde{\chi}^{\pm}(x,\xi,t)\leq 0\;\;\textrm{on}\;\;U\times(0,\infty)\times[0,T].\end{equation}
We will now use \eqref{6_49} and \eqref{exist_172} to show that $\tilde{\chi}^\pm$ are $\{0,1\}$-valued.  The methods of Theorem~\ref{thm_unique}, which rely on the initial time continuity \eqref{exis_134}, \eqref{exist_172}, and the techniques of Appendix~\ref{sec_Ito} with the form of It\^o's formula for semimartingales found in Kallenberg \cite[Theorem~20.7]{Kallenberg} to handle the potential discontinuities introduced by the measure $\tilde{q}$ justify differentiating the equality
\[\medint\int_\R\medint\int_{U}\tilde{\chi}^{\pm}+\tilde{\chi}^{\pm}-2(\tilde{\chi}^{\pm})^2 = 2\medint\int_\R\medint\int_{U}\tilde{\chi}^{\pm}(1-\tilde{\chi}^{\pm}).\]
A repetition of the proof of Theorem~\ref{thm_unique} in this context proves that, $\tilde{\P}$-a.s.\ for every $t\in[0,T]$,
\[\medint\int_\R\medint\int_{U}\tilde{\chi}^{\pm}(x,\xi,t)(1-\tilde{\chi}^{\pm}(x,\xi,t))\dx\dxi\leq \medint\int_\R\medint\int_{U}\overline{\chi}(\rho_0)(1-\overline{\chi}(\rho_0))\dx\dxi=0,\]
and, since the weak convergence implies that $0\leq \tilde{\chi}^{\pm}\leq 1$ almost everywhere, this proves $\tilde{\P}$-a.s.\ that
\begin{equation}\label{exist_170}\textrm{$\tilde{\chi}^{\pm}$ are $\{0,1\}$-valued on $U\times(0,\infty)\times[0,T]$.}\end{equation}
It is a consequence of \eqref{exist_172} and \eqref{exist_170} that there $\tilde{\P}$-a.s.\ exist $\tilde{\rho}^{\pm}\in L^1(U\times[0,T])$ that satisfy $\tilde{\rho}^{\pm}(x,t)=\tilde{\rho}(x,t)$ for almost every $t\in[0,T]$ and that satisfy
\begin{equation}\label{exist_171}\tilde{\chi}^{\pm}(x,\xi,t)=\mathbf{1}_{\{0<\xi<\tilde{\rho}^\pm(x,t)\}}.\end{equation}
We will now prove that $\tilde{\rho}^+=\tilde{\rho}^-$.  We have from \eqref{exis_148} and \eqref{6_49} that, $\tilde{\P}$-a.s.\ for every $\alpha\in\C^\infty_c((0,\infty))$ and $\psi\in\C^\infty(\overline{U})$,
\begin{equation}\label{5_53}\medint\int_\R\medint\int_{U}(\tilde{\chi}^+(x,\xi,t)-\tilde{\chi}^-(x,\xi,t))\psi(x)\alpha(\xi) = \tilde{Q}^-_{\alpha\psi}(t)-\tilde{Q}^+_{\alpha\psi}(t)=-\medint\int_{\{t\}\times U\times(0,\infty)}\alpha'(\xi)\psi(x)\dd \tilde{q}. \end{equation}
Let $\{\alpha_n\}_{n\in\N}$ be smooth functions that satisfy $0\leq \alpha_n\leq 1$, $\alpha_n(\eta)=1$ if $\nicefrac{1}{n}\leq \eta\leq n$, $\alpha_n(\eta)=0$ if $\eta<\nicefrac{1}{2n}$ or if $\eta>n+1$, and such that $\alpha_n'(\eta)\leq\nicefrac{c}{n}$ if $\nicefrac{1}{2n}<\eta<\nicefrac{1}{n}$ and $\alpha_n'(\eta)\leq c$ if $n<\eta<n+1$ for $c\in(0,\infty)$ independent of $n$.  It follows from \eqref{exist_171} and \eqref{5_53} that $\tilde{\P}$-a.s.\ for every $\psi\in\C^\infty(\overline{U})$ there exists $c\in(0,\infty)$ depending on $\psi$ such that, for every $t\in[0,T]$,
\[\abs{\medint\int_{U}(p^+(x,t)-p^-(x,t))\psi(x)\dx}\leq c\liminf_{n\rightarrow\infty}\left(n\tilde{q}(\{t\}\times U\times[\nicefrac{1}{2n},\nicefrac{1}{n}])+\tilde{q}(\{t\}\times U\times[n,n+1])\right).\]
We can now repeat the proof of Proposition~\ref{prop_measure}, which relies on the $\tilde{\P}$-a.s.\ nonvanishing of $\tilde{\rho}$ shown prior to \eqref{exis_1} and in \eqref{exist_175} and the $\tilde{\P}$-a.s.\ finiteness of the measures $\tilde{q}$, which implies that the measures vanish at infinity in the sense of \eqref{d_s4}, to prove that, $\tilde{\P}$-a.s.\ for every $t\in[0,T]$ and $\psi\in\C^\infty(\overline{U})$,
\[\abs{\medint\int_{U}(p^+(x,t)-p^-(x,t))\psi(x)\dx}=0.\]
Therefore, by duality, we have $\tilde{\P}$-a.s\ that $\tilde{\rho}^+=\tilde{\rho}^-$ in $L^1(U\times[0,T])$.   We therefore have $\tilde{\P}$-a.s.\ that $\tilde{\chi}^+=\tilde{\chi}^-$ in $L^2(U\times[0,\infty)\times[0,T])$, and that $\tilde{\chi}^+=\tilde{\chi}^-$ is almost surely weakly $L^2$-continuous in time.  Finally, we will use the weak continuity of $\tilde{\chi}^+$ to prove the strong continuity of $\tilde{\rho}^+$.  Let $t\in[0,T]$ and let $\{t_k\}_{k\in[0,\infty)}$ be an arbitrary sequence in $[0,T]$ that satisfies $t_k\rightarrow t$ as $k\rightarrow\infty$.  We then have $\tilde{\P}$-a.s.\ that
\begin{align*}
& \limsup_{k\rightarrow\infty}\medint\int_{U}\abs{\tilde{\rho}^+(x,t)-\tilde{\rho}^+(x,t_k)}\dx = \limsup_{k\rightarrow\infty}\medint\int_\R\medint\int_{U}\abs{\tilde{\chi}^+(x,\xi,t)-\tilde{\chi}^+(x,\xi,t_k)}^2\dx\dxi
\\ & = \limsup_{k\rightarrow\infty}\medint\int_\R\medint\int_{U}\left(\tilde{\chi}^+(x,\xi,t)+\tilde{\chi}^+(x,\xi,t_k)-2\tilde{\chi}^+(x,\xi,t)\tilde{\chi}^+(x,\xi,t_k)\right)\dx\dxi=0.
\end{align*}
We conclude that $\tilde{\rho}^+\in\C([0,T;L^1(U))$ and that $\tilde{\rho}$ has a representative in $L^1(\O;L^1(U\times[0,T]))$ that $\tilde{\P}$-a.s.\ takes values in $\C([0,T];L^1(U))$.  It then follows from the continuity and \eqref{6_49} that the measure $\tilde{q}$ has no atoms in time, which completes the proof.

\textbf{Conclusion for $H^1(U)$-valued initial data}.  We have shown that $\tilde{\rho}$ is a stochastic kinetic solution of \eqref{i_eq} in the sense of Definition~\ref{sol_def}, with initial data $\tilde{\rho}_0$, Brownian motion $\tilde{\beta}$, and filtration $(\overline{\mathcal{G}}_t)_{t\in[0,\infty)}$.  Returning to \eqref{6_32}, \eqref{6_34}, and \eqref{6_35}, a repetition of the above arguments proves that there exists $\overline{\rho}\in \C([0,T];L^1(U))$ such that
\[\tilde{Z} = (\overline{\rho},\overline{q},(\overline{M}^j)_{j\in\N},L(\overline{\rho})),\]
and that $\overline{\rho}$ is similarly a stochastic kinetic solution of \eqref{i_eq} in the sense of Definition~\ref{sol_def}, with initial data $\tilde{\rho}_0$, Brownian motion $\tilde{\beta}$, and filtration $(\overline{\mathcal{G}}_t)_{t\in[0,\infty)}$.  Theorem~\ref{thm_unique} proves $\tilde{\P}$-a.s.\ that  $\tilde{\rho}=\overline{\rho}$ in $\C([0,T];L^1(U))$.  Finally, returning to \eqref{6_32}, we conclude that along the subsequence $k\rightarrow\infty$ the joint laws of $(X_{n_k},X_{m_k})$ restricted to $L^1([0,T];L^1(U))^2$ converge weakly to a measure $\mu$ on $L^1([0,T];L^1(U))^2$ satisfying the conditions of Lemma~\ref{lemma_weak}.

Returning to the original solutions $\{\rho_n\}_{n\in\N}$ defined on $(\O,\F,\P)$, Lemma~\ref{lemma_weak} proves that, along a subsequence $n_k\rightarrow\infty$, there exists $\rho\in L^1(\O;L^1(U\times[0,T]))$ such that the $\{\rho_{n_k}\}_{k\in\N}$ converge to $\rho$ in probability.  Passing to a further subsequence, the $\rho_{n_k}\rightarrow\rho$ almost surely.  A simplified version of the above argument then proves that $\rho$ is a stochastic kinetic solution of \eqref{i_eq} in the sense of Definition~\ref{sol_def} on $(\O,\F,\P)$, with respect to the original Brownian motion and filtration.  The estimates follow from the same argument and the weak lower semicontinuity of the Sobolev norm.  This completes the proof of existence for initial data in $L^{2p}(\O;H^1(U))$.

\textbf{Conclusion}.  We first observe that the $H^1$-regularity was only used to obtain the initial time continuity in \eqref{exis_134}.  The remaining estimates are stable with respect to initial data in $L^{2p}(\O;L^2(U))$.  Given a nonnegative, $\F_0$-measurable $\rho_0\in L^{2p}(\O;L^2(U))$ we obtain by localizing the data way from the boundary and convolution a sequence of smooth, nonnegative $\rho_{0,n}\in L^{2p}(\O;H^1(U))$ satisfying $\P$-a.s.\ that $\rho_{0,n}\rightarrow \rho_0$ strongly in $L^2(U)$.  We then have using the $L^1(U)$-contraction property of Theorem~\ref{thm_unique} and a repetition of the above argument that $\P$-a.s.\ the solutions $\rho_n$ constructed above converge strongly in $\C([0,T];L^1(U))$ to a solution $\rho$ of \eqref{i_eq} in the sense of Definition~\ref{sol_def}.  We finally consider a nonnegative $\F_0$-measurable $\rho_0\in L^1(\O;L^1(U))$.  For every $n\in\N$ let $\rho_{0,n}=(\rho\wedge n)$, and let $\rho_n$ be the solution constructed above with initial data $\rho_{0,n}$.  It follows from the methods of Theorem~\ref{thm_unique} that the equation preserves order, and therefore that the sequence $\rho_n$ is $\P$-a.s.\ nondecreasing in $n$.   A repetition of the above arguments, where this time the local regularity property of Proposition~\ref{prop_measure} is used to guarantee properties \eqref{d_s2} and \eqref{d_s4} of Definition~\ref{sol_def}, the fact that the $\rho_n$ are nondecreasing is used to establish the nonvanishing property \eqref{s_log}, and the $L^1(U)$-contraction property of Theorem~\ref{thm_unique} is used to establish the $L^1(U)$-continuity in time, establishes the existence of a strong solution to \eqref{i_eq} with initial data $\rho_0$ in the sense of Definition~\ref{sol_def}.  This completes the proof.\end{proof}

\begin{remark}\label{remark_exist_11}  We remark that, under Assumptions~\ref{assume_d} and \ref{assume_n}, Theorem~\ref{thm_rks_ex} applies without changes to the equation
\[\partial_t\rho = \nabla\cdot a\nabla\phi(\rho)-\nabla\cdot (\phi^{\nicefrac{1}{2}}(\rho)\circ s\dd\xi),\]
and to more general noise terms $(\sigma(\rho)\circ s\dd\xi)$ for coefficients $\sigma$ satisfying \cite[Assumption~5.2]{FG21}.\end{remark}

\noindent{\bf  Acknowledgements}.  The author acknowledges support from the National Science Foundation DMS-Probability
Standard Grant 2348650, the Simons Foundation Travel Grant MPS-TSM-00007753, and the
Louisiana Board of Regents RCS Grant 20130014386.

\appendix

\section{The application of It\^o's formula in \eqref{u_10}}\label{sec_Ito}

In this brief section we provide a simple proof of the It\^o's formula used to justify \eqref{u_10}, which can also be justified using the methods of \cite{Kry2013}.  Let $\{e_k\}_{k\in\N}$ be an orthonormal $L^2(U)$-basis that is an orthogonal $H^1(U)$-basis.  Take, for example, the Neumann eigenfunctions of the Laplacian on $U$.  Then, for every $\eta\in (2\d,\infty)$ let $\lambda_{k,i}$ be defined by
\[\lambda_{k,i}(s,\eta) = \medint\int_U \chi^\d_{s,i}(\eta)e_k.\]
It is a consequence of \eqref{sol_def} that the $\lambda_{k,i}$ are solutions to the stochastic differential equation
\begin{align*}
\dd \lambda_{k,i} & = \Big(-\medint\int_U \nabla e_k\cdot a\nabla\phi(\rho_i)\overline{\kappa}^\d_{s,i}(\eta)+\medint\int_{U}\partial_\eta\big(\kappa^\d*q_i\big)(\eta)e_k\Big)\dt
\\  \nonumber & \quad \Big(-\frac{\langle\xi^F\rangle_1}{8}\medint\int_{U} \nabla e_k \cdot a\nabla\log(\rho_i)\overline{\kappa}^\d_{s,i}(\eta)-\frac{\langle\xi^F\rangle_1}{4}\medint\int_{U}s(\nabla\cdot s^t)\cdot \nabla e_k\Big)\dt
\\ \nonumber & \quad \Big(-\frac{1}{2}\medint\int_{U}e_k\big(\partial_\eta\overline{\kappa}^\d_{s,i}\big)(\eta)\rho_i\langle\nabla\cdot s\xi^F\rangle_1-\frac{\langle\xi^F\rangle_1}{4}\medint\int_0^t\medint\int_{U}e_k\big(\partial_\eta\overline{\kappa}^\d_{s,i}\big)(\eta)s(\nabla\cdot s^t)\cdot\nabla\rho_i\Big)\dt
\\ \nonumber & \quad +\medint\int_{U}\nabla e_k\cdot \overline{\kappa}^\d_{s,i}(\eta)\sqrt{\rho_i} s\dd\xi^F-\medint\int_U\big(\partial_\eta\overline{\kappa}^\d_{s,i}\big)(\eta)e_k\sqrt{\rho_i}\nabla\rho_i\cdot s\dd\xi^F,
\end{align*}
where, owing to the choice of orthonormal basis $e_k$, the final martingale term can be rewritten in the form
\begin{align*}
& \medint\int_{U}\nabla e_k\cdot \overline{\kappa}^\d_{s,i}(\eta)\sqrt{\rho_i} s\dd\xi^F-\medint\int_U\big(\partial_\eta\overline{\kappa}^\d_{s,i}\big)(\eta)e_k\sqrt{\rho_i}\nabla\rho_i\cdot s\dd\xi^F
\\ & = -\medint\int_U e_k\nabla\cdot (\overline{\kappa}^\d_{s,i}\sqrt{\rho_i} s\dd\xi^F)+\medint\int_Ue_k\sqrt{\rho_i}\nabla \overline{\kappa}^\d_{s,i}\cdot s\dd\xi^F= -\medint\int_U e_k\overline{\kappa}^\d_{s,i}\nabla\cdot (\sqrt{\rho_i} s\dd\xi^F).
\end{align*}
It is a consequence of the Kolmogorov continuity criterion (see, for example, \cite[Chapter~1, Theorem~2.1]{RevYor1999}) that for every $\d\in(0,1)$ there exists a subset of full probability such that the $\lambda_{k,i}$ satisfy the above equation for every $\eta\in(\d,\infty)$.  It then follows that, $\P$-a.s.\ for every $t\in[0,T]$,
\[\medint\int_U \chi^\d_{s,1}(\eta)\chi^\d_{s,2}(\eta) = \textstyle\sum_{k=1}^\infty \lambda_{k,1}(\eta)\lambda_{k,2}(\eta).\]
Equation \eqref{u_10} then follows by integrating the above equality over $\R$ with respect to the cutoff functions $\varphi_\beta\zeta_M$, and then applying the standard one-dimensional It\^o formula to the products of the eigenvalues.

\bibliography{WhiteNoise}

\begin{thebibliography}{10}

\bibitem{Aubin}
J.-P. Aubin.
\newblock Un th\'eor\`eme de compacit\'e.
\newblock {\em C. R. Acad. Sci. Paris}, 256:5042--5044, 1963.

\bibitem{BenKar2004}
M.~Bendahmane and K.~H.\ Karlsen.
\newblock Renormalized entropy solutions for quasi-linear anisotropic
  degenerate parabolic equations.
\newblock {\em SIAM J. Math. Anal.}, 36(2):405--422, 2004.

\bibitem{BDSGJLL15}
L.~Bertini, A.~De~Sole, D.~Gabrielli, G.~Jona-Lasinio, and C.~Landim.
\newblock Macroscopic fluctuation theory.
\newblock {\em Reviews of Modern Physics}, 87(2):593, 2015.

\bibitem{Bil1999}
P.~Billingsley.
\newblock {\em Convergence of probability measures}.
\newblock Wiley Series in Probability and Statistics: Probability and
  Statistics. John Wiley \& Sons, Inc., New York, second edition, 1999.

\bibitem{BGN16}
F.~Bouchet, K.~Gaw{\c e}dzki, and C.~Nardini.
\newblock Perturbative {{Calculation}} of {{Quasi}}-{{Potential}} in
  {{Non}}-equilibrium {{Diffusions}}: {{A Mean}}-{{Field Example}}.
\newblock {\em Journal of Statistical Physics}, 163(5):1157--1210, June 2016.

\bibitem{CP03}
G.-Q. Chen and B.~Perthame.
\newblock Well-posedness for non-isotropic degenerate parabolic-hyperbolic
  equations.
\newblock {\em Annales de l'Institut Henri Poincar{\'e}. Analyse Non
  Lin{\'e}aire}, 20(4):645--668, 2003.

\bibitem{Cli2023}
A.~Clini.
\newblock Porous media equations with nonlinear gradient noise and dirichlet
  boundary conditions.
\newblock {\em Stochastic Processes and their Applications}, 159:428--498,
  2023.

\bibitem{CliFeh2024}
A.~Clini and B.~Fehrman.
\newblock A central limit theorem for nonlinear conservative spdes.
\newblock {\em arXiv:2310.19924}, 2023.

\bibitem{CF23}
F.~Cornalba and J.~Fischer.
\newblock The {D}ean-{K}awasaki equation and the structure of density
  fluctuations in systems of diffusing particles.
\newblock {\em Arch. Ration. Mech. Anal.}, 247(5):Paper No. 76, 59, 2023.

\bibitem{CFIR23}
F.~Cornalba, J.~Fischer, J.~Ingmanns, and C.~Raithel.
\newblock Density fluctuations in weakly interacting particle systems via the
  dean-kawasaki equation.
\newblock {\em arXiv preprint arXiv:2303.00429}, 2023.

\bibitem{CorShaZim2019}
F.~Cornalba, T.~Shardlow, and J.~Zimmer.
\newblock A regularized {D}ean-{K}awasaki model: derivation and analysis.
\newblock {\em SIAM J. Math. Anal.}, 51(2):1137--1187, 2019.

\bibitem{CorShaZim2020}
F.~Cornalba, T.~Shardlow, and J.~Zimmer.
\newblock From weakly interacting particles to a regularised {D}ean-{K}awasaki
  model.
\newblock {\em Nonlinearity}, 33(2):864--891, 2020.

\bibitem{DG18}
K.~Dareiotis and B.~Gess.
\newblock Nonlinear diffusion equations with nonlinear gradient noise.
\newblock {\em Electron. J. Probab.}, 25:Paper No. 35, 43, 2020.

\bibitem{DeLOttWes2003}
C.~De~Lellis, F.~Otto, and M.~Westdickenberg.
\newblock Structure of entropy solutions for multi-dimensional scalar
  conservation laws.
\newblock {\em Arch. Ration. Mech. Anal.}, 170(2):137--184, 2003.

\bibitem{D96}
D.~Dean.
\newblock Langevin equation for the density of a system of interacting
  {{Langevin}} processes.
\newblock {\em Journal of Physics A: Mathematical and General}, 29(24):L613,
  1996.

\bibitem{DebHofVov2016}
A.~Debussche, M.~Hofmanov\'a, and J.~Vovelle.
\newblock Degenerate parabolic stochastic partial differential equations:
  Quasilinear case.
\newblock {\em Ann. Probab.}, 44(3):1916--1955, 2016.

\bibitem{DebVov2010}
A.~Debussche and J.~Vovelle.
\newblock Scalar conservation laws with stochastic forcing.
\newblock {\em Journal of Functional Analysis}, 259(4):1014--1042, 2010.

\bibitem{DirFehGes19}
N.~Dirr, B.~Fehrman, and B.~Gess.
\newblock Conservative stochastic pde and fluctuations of the symmetric simple
  exclusion process.
\newblock {\em arXiv:2012.02126}, 2021.

\bibitem{DKP22}
A.~Djurdjevac, H.~Kremp, and N.~Perkowski.
\newblock Weak error analysis for a nonlinear spde approximation of the
  dean-kawasaki equation, 2022.

\bibitem{DFVE14}
A.~Donev, T.~G. Fai, and E.~{Vanden-Eijnden}.
\newblock A reversible mesoscopic model of diffusion in liquids: From giant
  fluctuations to {{Fick}}'s law.
\newblock {\em Journal of Statistical Mechanics: Theory and Experiment},
  2014(4):P04004, April 2014.

\bibitem{DVE14}
A.~Donev and E.~{Vanden-Eijnden}.
\newblock Dynamic density functional theory with hydrodynamic interactions and
  fluctuations.
\newblock {\em The Journal of Chemical Physics}, 140(23):234115, June 2014.

\bibitem{Eva2010}
L.~C. Evans.
\newblock {\em Partial differential equations}, volume~19 of {\em Graduate
  Studies in Mathematics}.
\newblock American Mathematical Society, Providence, RI, second edition, 2010.

\bibitem{FG17}
B.~Fehrman and B.~Gess.
\newblock Well-posedness of nonlinear diffusion equations with nonlinear,
  conservative noise.
\newblock {\em Arch. Ration. Mech. Anal.}, 233(1):249--322, 2019.

\bibitem{FehGes20211}
B.~Fehrman and B.~Gess.
\newblock Path-by-path well-posedness of nonlinear diffusion equations with
  multiplicative noise.
\newblock {\em Journal de Mathématiques Pures et Appliquées}, 148:221--266,
  2021.

\bibitem{FehGes19}
B.~Fehrman and B.~Gess.
\newblock Non-equilibrium large deviations and parabolic-hyperbolic {PDE} with
  irregular drift.
\newblock {\em Invent. Math.}, 234(2):573--636, 2023.

\bibitem{FehGes24}
B.~Fehrman and B.~Gess.
\newblock Conservative stochastic pdes on the whole space.
\newblock {\em arXiv:2410.00254}, 2024.

\bibitem{FG21}
B.~Fehrman and B.~Gess.
\newblock Well-posedness of the {D}ean--{K}awasaki and the nonlinear
  {D}awson--{W}atanabe equation with correlated noise.
\newblock {\em Arch. Rational. Mech. Anal.}, 248(20), 2024.

\bibitem{FG16}
P.~K. Friz and B.~Gess.
\newblock Stochastic scalar conservation laws driven by rough paths.
\newblock {\em Ann. Inst. H. Poincar{\'e} Anal. Non Lin{\'e}aire},
  33(4):933--963, 2016.

\bibitem{FrizVictoir}
P.~K. Friz and N.~B. Victoir.
\newblock {\em Multidimensional stochastic processes as rough paths}, volume
  120 of {\em Cambridge Studies in Advanced Mathematics}.
\newblock Cambridge University Press, Cambridge, 2010.
\newblock Theory and applications.

\bibitem{GS14}
B.~Gess and P.~E. Souganidis.
\newblock Scalar conservation laws with multiple rough fluxes.
\newblock {\em Commun. Math. Sci.}, 13(6):1569--1597, 2015.

\bibitem{GS17}
B.~Gess and P.~E. Souganidis.
\newblock Long-time behavior, invariant measures, and regularizing effects for
  stochastic scalar conservation laws.
\newblock {\em Comm. Pure Appl. Math.}, 70(8):1562--1597, 2017.

\bibitem{GS16-2}
B.~Gess and P.~E. Souganidis.
\newblock Stochastic non-isotropic degenerate parabolic--hyperbolic equations.
\newblock {\em Stochastic Process. Appl.}, 127(9):2961--3004, 2017.

\bibitem{GesWuZha2024}
B.~Gess, W.~Wu, and R.~Zhang.
\newblock Higher order fluctuation expansions for nonlinear stochastic heat
  equations in singular limits.
\newblock {\em arXiv:2406.17892}, 2024.

\bibitem{GyoKry1996}
I.~Gy\"ongy and N.~Krylov.
\newblock Existence of strong solutions for {I}t\^{o}'s stochastic equations
  via approximations.
\newblock {\em Probab. Theory Related Fields}, 105(2):143--158, 1996.

\bibitem{Hof2013}
M.~Hofmanov\'a.
\newblock Degenerate parabolic stochastic partial differential equations.
\newblock {\em Stochastic Processes and their Applications},
  123(12):4294--4336, 2013.

\bibitem{Jak1997}
A.~Jakubowski.
\newblock The almost sure {S}korokhod representation for subsequences in
  nonmetric spaces.
\newblock {\em Teor. Veroyatnost. i Primenen.}, 42(1):209--216, 1997.

\bibitem{Kallenberg}
O.~Kallenberg.
\newblock {\em Foundations of Modern Probability}, volume~99 of {\em
  Probability Theory and Stochastic Modelling}.
\newblock Springer, third edition, 2021.

\bibitem{KarRis2003}
K.~H.\ Karlsen and N.~H.\ Risebro.
\newblock On the uniqueness and stability of entropy solutions of nonlinear
  degenerate parabolic equations with rough coefficients.
\newblock {\em Discrete Contin. Dyn. Syst.}, 9(5):1081--1104, 2003.

\bibitem{K94}
K.~Kawasaki.
\newblock Stochastic model of slow dynamics in supercooled liquids and dense
  colloidal suspensions.
\newblock {\em Physica A: Statistical Mechanics and its Applications},
  208(1):35--64, 1994.

\bibitem{LKR19}
V.~Konarovskyi, T.~Lehmann, and M.-K. von Renesse.
\newblock Dean-{K}awasaki dynamics: ill-posedness vs. triviality.
\newblock {\em Electron. Commun. Probab.}, 24:Paper No. 8, 9, 2019.

\bibitem{KR17}
V.~Konarovskyi and M.~{von Renesse}.
\newblock Reversible {{Coalescing}}-{{Fragmentating Wasserstein Dynamics}} on
  the {{Real Line}}.
\newblock {\em arXiv:1709.02839}, September 2017.

\bibitem{KR15}
V.~Konarovskyi and M.-K. von Renesse.
\newblock Modified massive {A}rratia flow and {W}asserstein diffusion.
\newblock {\em Comm. Pure Appl. Math.}, 72(4):764--800, 2019.

\bibitem{Kry2013}
N.~V. Krylov.
\newblock A relatively short proof of {I}t\^{o}'s formula for {SPDE}s and its
  applications.
\newblock {\em Stoch. Partial Differ. Equ. Anal. Comput.}, 1(1):152--174, 2013.

\bibitem{pLions}
J.-L. Lions.
\newblock {\em Quelques m\'ethodes de r\'esolution des probl\`emes aux limites
  non lin\'eaires}.
\newblock Dunod; Gauthier-Villars, Paris, 1969.

\bibitem{LPS13-2}
P.-L. Lions, B.~Perthame, and P.~Souganidis.
\newblock Stochastic averaging lemmas for kinetic equations.
\newblock {\em S{\'e}minaire Laurent Schwartz \textemdash{} EDP et
  applications}, pages 1--17, 2011.

\bibitem{LPS13}
P.-L. Lions, B.~Perthame, and P.~E. Souganidis.
\newblock Scalar conservation laws with rough (stochastic) fluxes.
\newblock {\em Stoch. Partial Differ. Equ. Anal. Comput.}, 1(4):664--686, 2013.

\bibitem{LPS14}
P.-L. Lions, B.~Perthame, and P.~E. Souganidis.
\newblock Scalar conservation laws with rough (stochastic) fluxes: The
  spatially dependent case.
\newblock {\em Stoch. Partial Differ. Equ. Anal. Comput.}, 2(4):517--538, 2014.

\bibitem{LPT94}
P.-L. Lions, B.~Perthame, and E.~Tadmor.
\newblock A kinetic formulation of multidimensional scalar conservation laws
  and related equations.
\newblock {\em J. Amer. Math. Soc.}, 7(1):169--191, 1994.

\bibitem{Per1998}
B.~Perthame.
\newblock Uniqueness and error estimates in first order quasilinear
  conservation laws via the kinetic entropy defect measure.
\newblock {\em J. Math. Pures Appl. (9)}, 77(10):1055--1064, 1998.

\bibitem{PerthameKinetic}
B.~Perthame.
\newblock {\em Kinetic formulation of conservation laws}, volume~21 of {\em
  Oxford Lecture Series in Mathematics and its Applications}.
\newblock Oxford University Press, Oxford, 2002.

\bibitem{Pop24}
S.~Popat.
\newblock Well-posedness of the generalised dean-kawasaki equation with
  correlated noise on bounded domains.
\newblock {\em arXiv:2403.19466}, 2024.

\bibitem{RevYor1999}
D.~Revuz and M.~Yor.
\newblock {\em Continuous martingales and {B}rownian motion}, volume 293 of
  {\em Grundlehren der mathematischen Wissenschaften}.
\newblock Springer-Verlag, Berlin, third edition, 1999.

\bibitem{Simon}
J.~Simon.
\newblock Compact sets in the space {$L^p(0,T;B)$}.
\newblock {\em Ann. Mat. Pura Appl. (4)}, 146:65--96, 1987.

\bibitem{Spo2012}
H.~Spohn.
\newblock {\em Large Scale Dynamics of Interacting Particles}.
\newblock {Springer Science \& Business Media}, 2012.

\bibitem{StuvRe2009}
M.-K. von Renesse and K.-T. Sturm.
\newblock Entropic measure and {W}asserstein diffusion.
\newblock {\em Ann. Probab.}, 37(3):1114--1191, 2009.

\bibitem{DKInt}
L.~Wang, Z.~Wu, and R.~Zhang.
\newblock {D}ean--{K}awasaki equation with singular interactions and
  applications to dynamical ising-kac model.
\newblock {\em arXiv:2207.12774}, 2022.

\bibitem{DKIntLDP}
Z.~Wu and R.~Zhang.
\newblock {M}c{K}ean--{V}lasov pde with irregular drift and applications to
  large deviations for conservative spdes.
\newblock {\em arXiv:2208.13142}, 2022.

\end{thebibliography}
\bibliographystyle{plain}

\end{document}